\documentclass[a4paper, reqno, 11pt]{amsart}   	
\usepackage[DIV=12, oneside]{typearea}			
\usepackage[english]{babel}

\usepackage{csquotes}

\usepackage{amsmath}
\usepackage{amstext,amssymb,amsopn,amsthm}
\usepackage{dsfont}  
\usepackage{thmtools}

\usepackage{tikz}
\usepackage{pgfplots}
\pgfplotsset{compat=newest}
\usetikzlibrary{patterns}
\usepgfplotslibrary{fillbetween}

\usepackage[backgroundcolor=white, bordercolor=blue, linecolor=blue]{todonotes}
\usepackage{enumitem}  
\usepackage[titletoc,title]{appendix}

\theoremstyle{plain}
\declaretheorem[title=Theorem, parent=section]{Theorem}
\declaretheorem[title=Lemma,sibling=Theorem]{Lemma}
\declaretheorem[title=Proposition,sibling=Theorem]{Proposition}

\theoremstyle{definition}
\declaretheorem[title=Definition,sibling=Theorem]{Definition}
\declaretheorem[title=Remark,sibling=Theorem]{Remark}
\declaretheorem[title=Remark, numbered=no]{Remark*}
\declaretheorem[title=Example, sibling=Theorem]{Example}
\declaretheorem[title=Assumption, numbered=no]{Assumption*}

\def\Xint#1{\mathchoice
   {\XXint\displaystyle\textstyle{#1}}%
   {\XXint\textstyle\scriptstyle{#1}}%
   {\XXint\scriptstyle\scriptscriptstyle{#1}}%
   {\XXint\scriptscriptstyle\scriptscriptstyle{#1}}%
   \!\int}
\def\XXint#1#2#3{{\setbox0=\hbox{$#1{#2#3}{\int}$}
     \vcenter{\hbox{$#2#3$}}\kern-.5\wd0}}

\def\dashint{\Xint-}

\numberwithin{equation}{section} 
\addto\extrasenglish{} 
\addto\extrasenglish{}
\addto\extrasenglish{}

\usepackage[ maxbibnames=5, style=alphabetic, doi=false,isbn=false , url=false]{biblatex} 
\AtBeginBibliography{\small}

\makeatletter
\providecommand\@dotsep{5}
\def\listtodoname{List of Todos}
\def\listoftodos{\@starttoc{tdo}\listtodoname}
\makeatother

\usepackage[colorlinks=true, linkcolor=black, citecolor=black, hypertexnames=false]{hyperref} 

\newcommand{\N}{\mathds{N}}
\newcommand{\R}{\mathds{R}}

\newcommand{\Sph}{\mathbb{S}}

\newcommand{\ind}{\mathds{1}}
\newcommand{\intd}{\, \mathrm{d}} 

\DeclareMathOperator{\supp}{supp}
\DeclareMathOperator{\pv}{p.v.}

\DeclareMathOperator{\Tail}{Tail}

\newcommand{\Poincare}{Poincar\'e}

\newcommand{\eps}{\varepsilon}
\renewcommand{\d}{\textnormal{\,d}}
\newcommand{\cE}{\mathcal{E}}

\begin{document}
	\allowdisplaybreaks
	\title{Dirichlet heat kernel estimates for parabolic nonlocal equations}
    \author{Philipp Svinger}
    \author{Marvin Weidner}
	
	\address{Fakult{\"a}t f{\"u}r Mathematik, Universit{\"a}t Bielefeld, Postfach 10 01 31, 33501 Bielefeld, Germany}
	\email{psvinger@math.uni-bielefeld.de}

    \address{Institute for Applied Mathematics, University of Bonn, Endenicher Allee 60, 53115, Bonn, Germany}
    \email{mweidner@uni-bonn.de}
    \urladdr{https://sites.google.com/view/marvinweidner/}

	\makeatletter
	\@namedef{subjclassname@2020}{%
		\textup{2020} Mathematics Subject Classification}
	\makeatother
	
	

\begin{abstract} 
In this article we establish the optimal $C^s$ boundary regularity for solutions to nonlocal parabolic equations in divergence form in $C^{1,\alpha}$ domains and prove a higher order boundary Harnack principle in this setting. Our approach applies to a broad class of nonlocal operators with merely H\"older continuous coefficients, but  our results are new even in the translation invariant case. As an application, we obtain sharp two-sided estimates for the associated Dirichlet heat kernel. Notably, our estimates cover nonlocal operators with time-dependent coefficients, which had remained open in the literature. 
\end{abstract}

\maketitle

\section{Introduction}

The goal of this work is to establish the optimal $C^s$ boundary regularity for solutions to nonlocal parabolic equations of the form
\begin{equation}
    \label{eq:PDE-intro}
        \left\{ \begin{aligned}
        \partial_t u - \mathcal{L}_t u &= f && \text{ in } (-1,0) \times (B_1 \cap \Omega ),\\
        u &= 0 && \text{ in } (-1,0) \times (B_1 \setminus \Omega)
        \end{aligned} \right.
\end{equation}

in $C^{1,\alpha}$ domains $\Omega \subset \R^d$ and to prove fine estimates on the boundary behavior of solutions. In particular, we obtain sharp two-sided bounds on the Dirichlet heat kernel.

We consider integro-differential operators for some $s\in (0,1)$ with $d>2s$ of the form
\begin{equation}\label{eq:OperatorDivergenceForm}
    -\mathcal{L}_t u (t,x)= \pv \int _{\R^d} \big( u(t,x) -u(t,y) \big) K(t,x,y) \intd y,
\end{equation}
where $K : \R \times \R^d \times \R^d \to [0,\infty]$ is symmetric and uniformly elliptic, i.e. for some $0 < \lambda \le \Lambda$,
\begin{equation}\label{eq:KernelDivergenceForm}
    K(t,x,y)=K(t,y,x), \quad   0<\frac{\lambda}{|x-y|^{d+2s}} \leq K(t,x,y) \leq \frac{\Lambda}{|x-y|^{d+2s}}   .
\end{equation}
Under these assumptions $-\mathcal{L}_t$ can be considered as a nonlocal operator in divergence form, that is modeled upon the fractional Laplacian $(-\Delta)^s$, and has a time-dependent coefficient function $K$.
Since we are interested in higher order regularity estimates, it is necessary to impose some regularity condition on $K$. In this paper, we will assume $K$ to be H\"older continuous of class $C^{\sigma}$ for some $\sigma \in (0,s)$ in some domain $\mathcal{A}\subset \R \times \R^d$, as follows:
\begin{equation} \label{eq:KernelHoelderCont}
    |K(t+\tau ,x+h,y+h)-K(t,x,y)| \leq \Lambda \frac{|h|^\sigma + |\tau | ^{\sigma /(2s)}}{|x-y|^{d+2s}} \quad \forall (t,x),(t,y)\in\mathcal{A}, ~~ (\tau , h) \in Q_1.
\end{equation}

\subsection{Optimal boundary regularity}

Over the past two decades, the regularity theory for elliptic and parabolic equations driven by nonlocal operators modeled upon the fractional Laplacian has developed into a highly active area of research. In particular, the interior regularity of solutions is by now well understood in the elliptic case (see \cite{Bass2009,Kassmann2009,DiCastro2016,Cozzi2017,Brasco2017,Brasco2018,Fall2020,Mengesha2021,Nowak2021,Fall2022,Nowak2022,Nowak2023,Diening2024}), and more recently also for parabolic equations (see \cite{Caffarelli2011,Felsinger2013,Lara2014,Jin2015,Serra2015,FernandezReal2017,Grubb2018a,Byun2023,Kassmann2024,Nguyen2024,Diening2025b}). In fact, under suitable regularity assumptions on $K$ and $f$ one can prove interior regularity estimates of Schauder-, Cordes--Nirenberg, or Calder\'on--Zygmund-type for solutions to \eqref{eq:PDE-intro}.

The boundary behavior for nonlocal equations, however, is still much less explored. In fact, it was shown only recently that solutions to \emph{elliptic equations} 
\begin{equation*}
        \left\{ \begin{aligned}
        -\mathcal{L} u  &= f && \text{ in } B_1 \cap \Omega,\\
        u &= 0 && \text{ in } B_1 \setminus \Omega
        \end{aligned} \right.
\end{equation*}
governed by operators of the form $-\mathcal{L}$ as in \eqref{eq:OperatorDivergenceForm}-\eqref{eq:KernelDivergenceForm}-\eqref{eq:KernelHoelderCont} enjoy $C^s$ regularity up to the boundary \cite{RosOton2024,Kim2024} (see also \cite{Byun2025}). Previously, this was only known for the restrictive class of translation-invariant operators with homogeneous kernels, i.e.
\begin{align}
\label{eq:hom-kernel}
    K(x,y) = K\left(\frac{x-y}{|x-y|}\right)|x-y|^{-d-2s},
\end{align}
which are also known as ''$2s$-stable operators`` (see \cite{RosOton2014,RosOton2016a,RosOton2016b,RosOton2017}, \cite{Grubb2014,Grubb2015}, \cite{Grube2024,Song2025b}).
Note that $C^s$ regularity up to the boundary is optimal even for the fractional Laplacian. This stands in stark contrast to the interior regularity of solutions.

The generalization to the natural class of nonlocal operators in divergence form \eqref{eq:OperatorDivergenceForm}-\eqref{eq:KernelDivergenceForm}-\eqref{eq:KernelHoelderCont} was only possible due to the development of several new technical tools, including
\begin{itemize}
    \item suitable 1D barrier functions for translation invariant operators in \cite{RosOton2024},
    \item a nonlocal Morrey-Campanato-type theory at the boundary in \cite{Kim2024}.
\end{itemize}

For parabolic nonlocal equations of the form \eqref{eq:PDE-intro}, the boundary regularity theory remains far from understood and a general theory comparable to the elliptic case is wide open. To date, most of the available results are for the fractional Laplacian (e.g. \cite{FernandezReal2016,Biccari2017,Choi2023,Armstrong2025,Abdellaoui2025}) and there are only few articles considering more general kernels \cite{FernandezReal2017,RosOton2018,Grubb2018a,Grubb2019,Kukuljan2023}. In fact, the optimal $C^s$ boundary regularity in $x$ is only known for kernels of the restrictive form \eqref{eq:hom-kernel} that are translation invariant, with few extensions available for operators in non-divergence form \cite{FernandezReal2017,Grubb2018a}. See also \cite{Biccari2017,Biccari2018,Grubb2018b,Dong2024,Abels2025} for regularity results in spaces of Sobolev type. In these settings, one can employ barrier arguments to obtain regularity, a method that fails for equations in divergence form \eqref{eq:KernelDivergenceForm}. We are not aware of any higher order boundary regularity result for parabolic nonlocal equations with non-constant coefficients satisfying \eqref{eq:KernelDivergenceForm}.

The goal of this article is to develop the \emph{boundary regularity theory for nonlocal parabolic equations} \eqref{eq:PDE-intro}  that includes general nonlocal operators in divergence form \eqref{eq:OperatorDivergenceForm}-\eqref{eq:KernelDivergenceForm}-\eqref{eq:KernelHoelderCont} and dispenses with any kind of homogeneity assumption on the kernel.

It is important to stress that nonlocal effects lead to additional difficulties (and even to new phenomena) in parabolic problems that have no analog in the elliptic setting. This critical aspect was at the core of \cite{Kassmann2024}
\footnote{where it was shown that not all weak solutions, i.e. solutions with tails that are merely in $L^1_t$, to the fractional heat equation are H\"older continuous; a purely nonlocal phenomenon}
(see also \cite{Liao2024b,Byun2023a}), and it is also central to our approach and main results (see \autoref{example:counterex-2}). Roughly speaking, while irregular exterior data in elliptic problems still yields smooth solutions due to the regularization of the nonlocal operator in space, the parabolic case is fundamentally different. Here, rough data in time may propagate to the interior of the solution domain generating discontinuities since there is no regularization in the time variable. 

A natural quantity capturing these long range effects is the nonlocal tail term
\begin{align}
\label{eq:tail-def}
    \Tail(u(t,\cdot);R,x_0) = R^{2s} \int_{\R^d \setminus B_R(x_0)} \frac{|u(t,y)|}{|y-x_0|^{d+2s}} \d y, \qquad x_0 \in \R^d, ~~ R > 0,
\end{align}
and a key element of our approach is the fine study of this object.

Our main result is the optimal $C^s$ boundary regularity that we obtain for the first time for general nonlocal parabolic equations in divergence form. It reads as follows:
\begin{Theorem} \label{Thm:CsBoundaryRegAndHopfLemma}
    Let $s\in (0,1)$, $p >2$, $\alpha,\sigma \in (0,s)$, and $\Omega \subset \R^d$ be a $C^{1,\alpha}$ domain with $0\in \partial \Omega$.
    Furthermore, let $\mathcal{L}_t$ be an operator of the form \eqref{eq:OperatorDivergenceForm}-\eqref{eq:KernelDivergenceForm}, satisfying \eqref{eq:KernelHoelderCont} with $\mathcal{A}=Q_1$.
    Assume that $u$ is a weak solution to 
    \begin{equation*}
        \left\{ \begin{aligned}
          \partial _t u -\mathcal{L}_tu & =f & &\text{in} \quad (-1 ,0) \times (\Omega \cap B_1) , \\
           u & =0 & & \text{in} \quad  (-1 ,0) \times (B_1 \setminus \Omega) 
        \end{aligned} \right.
    \end{equation*}
    for some $f\in L^q_tL^r_x (Q_1^\Omega)$ where $q,r\geq 1$ satisfy $\frac{1}{q}+\frac{d}{2sr} <\frac12$, and $\Tail (u ;1/2,0) \in L^{p}_t ((-1,0))$.
    \begin{itemize}
        \item[(i)] Then, 
    \begin{equation*}
            \| u \| _{C^{s}_{p} (\overline{Q^\Omega_{1/2}})} \leq c \left( \| u\| _{L^2 (Q^\Omega _1)} + \Vert\Tail (u;1/2,0) \Vert_{L^p_t((-1,0))} + \| f\| _{L_{t}^q L_x^{r} (Q^\Omega_1)}\right)
       \end{equation*}
    for some constant $c=c(d,s,\sigma ,\lambda ,\Lambda,p,q,r , \alpha ,\Omega  )>0$.
    \item[(ii)]  If, furthermore, $f\geq 0$ in $Q_1^\Omega$ and $u\geq 0$ in $(-1,0) \times \R^d$, then, either $u\equiv 0$ in $(-1,0)\times \R^d$ or
    \begin{equation*}
        u \geq c d_\Omega ^s \quad \text{in }Q^\Omega _{1/2} 
    \end{equation*}
    for some constant $c=c(d,s,\sigma ,\lambda ,\Lambda,p,q,r , \alpha ,\Omega ,u ,f)>0$.
    \end{itemize}   
\end{Theorem}

Here we have used the notation \smash{$Q_R^\Omega:= (-R^{2s},0)\times (\Omega \cap B_R)$}. Moreover, by \smash{$C^{s}_p$} we denote the space of functions that are \smash{$C^{1/2}_t$} in $t$ and \smash{$C^s_x$} in $x$.
 This space matches the parabolic scaling of the equation and it is the optimal H\"older space for the problem \eqref{eq:PDE-intro} with this property. As we mentioned before, the $C^s_{x}$ regularity and also the lower bound on the growth rate in (ii) cannot be improved, even for the fractional Laplacian in the unit ball with smooth data. It is possible to prove higher regularity in $t$ in the setting of \autoref{Thm:CsBoundaryRegAndHopfLemma} since the $t$-variable is tangential to $(-1,0) \times \partial \Omega$, but we don't pursue this direction here.

In contrast to previous works (see \cite{FernandezReal2016,FernandezReal2017,RosOton2018,Kukuljan2023}), which assume tail terms to be bounded in time, we significantly relax this rather unnatural condition. In fact, we establish the $C^s_p$ regularity for solutions with tails that are merely in $L^{2+\eps}_t$. This result is already new for the fractional Laplacian and the integrability condition on the tail turns out to be optimal.
Quite strikingly, we find that energy solutions, i.e. solutions whose tail terms are in $L^2_t$, may fail to be $C^s_p$ up to the boundary. This phenomenon has no analog in the elliptic case, where all weak solutions are $C^s$ up to the boundary \cite{Kim2024}, nor in the parabolic interior theory, where finite $L^2_t$ tails are sufficient for $C^s_x$ regularity.\footnote{This can be seen by a straightforward generalization of the proof in \cite{Byun2023} when assuming coefficients to be $C^{\sigma}$, proceeding for instance as in the proof of \cite[Theorem 1.7]{Diening2025}. See also \cite[Theorem 1.6]{Diening2025b}.} We construct a counterexample (see \autoref{example:counterex-2}) showing that $C^s_x$ regularity up to the boundary fails if the tail is merely $L^2_t$.

Moreover, note that solutions are not $C^s_p$ if $\frac{1}{q} + \frac{d}{2sr} > \frac{1}{2}$. In this case, in \autoref{Thm:Cs-epsRegWithCoefficients} we establish the sharp $C^{\beta}_p$ regularity of solutions in flat Lipschitz domains, where $\beta = 2s-\frac{2s}{q} - \frac{d}{r}$.
 
 \autoref{Thm:CsBoundaryRegAndHopfLemma} can be regarded as a complete extension of the elliptic results in \cite{Kim2024} to nonlocal parabolic equations. Our approach is entirely parabolic in nature and as a key feature, we are able to consider kernels $K$ that also depend on $t$. Previous papers on boundary regularity were limited to kernels that are stationary in time. 

The parabolic Hopf lemma in \autoref{Thm:CsBoundaryRegAndHopfLemma}(ii) was only known for the fractional Laplacian (see \cite{Wang2022}) and seems to be entirely new for parabolic nonlocal operators in divergence form. 

In contrast to previous works we do not impose any homogeneity assumption \eqref{eq:hom-kernel} on the kernels. Note that the boundary behavior of solutions drastically changes when the kernel is assumed to be homogeneous. Indeed, despite $C^s_p$ regularity being optimal, for homogeneous kernels one can develop a higher order Schauder-type theory at the boundary for $u/d_{\Omega}^s$ (see \cite{RosOton2016a}, \cite{FernandezReal2017}, \cite[Theorem 1.8]{Kim2024}, \cite{Byun2025}). This is in stark contrast to the general case, where $u/d_{\Omega}^s$ might not even be continuous (see \cite[Remark 5.3]{RosOton2024}). Still, in this paper we establish a higher order boundary Harnack-type principle, showing that the quotient of two solutions to \eqref{eq:PDE-intro} enjoys a quantifiable amount of H\"older regularity.

\begin{Theorem} \label{Thm:BoundaryHarnack}
    Let $s\in (0,1)$, $p>2$, $\alpha,\sigma \in (0,s)$, and $\Omega \subset \R^d$ be a $C^{1,\alpha}$ domain with $0\in \partial \Omega$.
    Furthermore, let $\mathcal{L}_t$ be an operator of the form \eqref{eq:OperatorDivergenceForm}-\eqref{eq:KernelDivergenceForm}, satisfying \eqref{eq:KernelHoelderCont} with $\mathcal{A}=Q_1$.
    For $i\in \{ 1,2\}$, assume that $u_i$ is a weak solution to 
    \begin{equation*}
        \left\{ \begin{aligned}
          \partial _t u_i -\mathcal{L}_tu_i & =f_i & &\text{in} \quad (-1 ,0) \times (\Omega \cap B_1) , \\
           u_i & =0 & & \text{in} \quad  (-1 ,0) \times (B_1 \setminus \Omega) 
        \end{aligned} \right.
    \end{equation*}
    for some $f_i\in L^q_tL^r_x (Q_1^\Omega)$ where $q,r\geq 1$ satisfy $\frac{1}{q}+\frac{d}{2sr} < \frac12$.
    Moreover, assume that
    \begin{equation*}
        \| u_i\| _{L^2 (Q^\Omega_1)} +    \| \Tail (u_i;1/2,0) \| _{L^{p}_t ((0,1))} + \| f_i\| _{L_{t}^q L_x^{r} (Q^\Omega_1)} \leq 1,
    \end{equation*}
     $f_2 \geq 0$ in $Q^\Omega _1$, and $u_2 \geq 0$ in $(-1,0) \times \R^d$.
     Let $x_0\in \Omega \cap B _{1/2}$ and assume that we have a lower bound $0<c_0\leq u_2(-(3/4)^{2s},x_0)$.
     
     Then, for every positive $\beta$ satisfying $\beta < \sigma $ and $\beta \leq \min \{ s-\frac{2s}{q}-\frac{d}{r} , s-\frac{2s}{p}\}$, we have
    \begin{equation*}
         \left\| \frac{u_1}{u_2} \right\| _{C^{\beta}_{p} (\overline{Q^\Omega _{1/2}})} \leq c
    \end{equation*}
    for some constant $c=c(d,s,\sigma ,\lambda ,\Lambda,p,q,r , \alpha ,\Omega ,\beta ,c_0, d_\Omega (x_0) )>0$.
\end{Theorem}

Note that \autoref{Thm:BoundaryHarnack} is a higher order regularity result as it provides a quantifiable H\"older exponent $\beta$ of the quotient $u_1/u_2$. From this viewpoint, it is natural to assume the solution domain to be $C^{1,\alpha}$ and the coefficients to be $C^{\sigma}$. By a refinement of the approach in \cite[Theorem 6.8]{Kim2024}, we are able to significantly improve the relation between $\beta$ and the parameters $\alpha,\sigma,q,r,p$ in comparison to \cite{Kim2024}.
In \autoref{Rem:BetaEquSigmaPossible}, we explain how one can improve \autoref{Thm:BoundaryHarnack}, so that it holds true for the full range $\beta \leq \min \{\sigma ,s-\frac{2s}{q}-\frac{d}{r} , s-\frac{2s}{p} \}$.

\autoref{Thm:BoundaryHarnack} seems to be the first higher order boundary Harnack result for nonlocal equations with coefficients. Previous articles (see \cite{Abatangelo2020}, \cite{Kukuljan2023}) are restricted to homogeneous and translation invariant kernels \eqref{eq:hom-kernel}. Moreover, our proof relies on completely different arguments. We believe that \autoref{Thm:BoundaryHarnack} will be useful in the study of nonlocal free boundary problems \cite{Abatangelo2020} governed by operators $-\mathcal{L}_t$ satisfying \eqref{eq:OperatorDivergenceForm}, \eqref{eq:KernelDivergenceForm}, \eqref{eq:KernelHoelderCont}.

Let us point out that the H\"older continuity of the quotient of two solutions to an elliptic (or parabolic) equation is a fundamental property in potential theory. It already holds true in very rough settings, i.e. for domains $\Omega$ that do not even have to be Lipschitz continuous and for equations with bounded measurable coefficients. This theory is rather well understood for nonlocal equations (see \cite{Caffarelli2018,RosOton2019,Figalli2023}, as well as \cite{Bogdan1997,Kim2009,Kim2014,Bogdan2015,Cao2024}). In either of these rough settings, the H\"older exponent is qualitative and cannot be quantified. Therefore, \autoref{Thm:BoundaryHarnack} is of a completely different nature.

\subsection{Dirichlet heat kernel estimates}

Given an initial datum $u_0 \in L^2(\Omega)$ and a forcing term $f \in L^2((0,1) \times \Omega)$, the unique solution $u$ to the parabolic equation 
    \begin{equation}
    \label{eq:Dirichlet-problem}
        \left\{ \begin{aligned}
        \partial_t u - \mathcal{L}_t u &= f && \text{ in } (0,1) \times \Omega,\\
        u &= 0 && \text{ in } (0,1) \times (\R^d \setminus \Omega),\\
        u &= u_0 && \text{ in } \{ 0 \} \times \Omega 
        \end{aligned} \right.
    \end{equation}
is given by the following representation formula
\begin{align}
\label{eq:Duhamel}
    u(t,x) = \int_{\Omega} p_{\Omega}(t,x,y) u_0(y) \d y + \int_0^t \int_{\Omega} p_{\Omega}(\tau,x,y) f(t-\tau,y) \d y \d \tau.
\end{align}
Here, $p_{\Omega}$ denotes the Dirichlet heat kernel associated to \eqref{eq:Dirichlet-problem}. The Dirichlet heat kernel is an important object in partial differential equations, potential analysis, and probability theory. Its relevance in the first two fields is evident from \eqref{eq:Duhamel}, whereas in probability theory, $p_{\Omega}$ arises as the transition density of the jump process generated by $-\mathcal{L}$ killed upon leaving $\Omega$.

In all of these fields it is a crucial task to gain a good understanding of fine properties of $p_{\Omega}$. The first article to establish sharp two-sided pointwise estimates for $p_{\Omega}$ is \cite{Chen2010}, where it was proved that 
\begin{align}
\label{eq:two-sided-intro}
    p_\Omega (t,x,y) \asymp \left( 1 \wedge \frac{d_\Omega (x)}{t^\frac{1}{2s}} \right) ^s  \left( 1 \wedge \frac{d_\Omega (y)}{t^\frac{1}{2s}} \right) ^s \left( t^{-\frac{d}{2s}} \wedge\frac{t}{|x-y|^{d+2s}} \right)
\end{align}
in case $-\mathcal{L} = (-\Delta)^s$ and $\Omega$ is a bounded $C^{1,1}$ domain. This estimate captures the precise boundary behavior of $p_{\Omega}$.

During the last 15 years there has been a huge interest in obtaining generalizations of the result in \cite{Chen2010}. We refer to \cite{Bogdan2010} where an approximate factorization of the Dirichlet heat kernel for $(-\Delta)^s$ was proved, which leads to two-sided bounds in Lipschitz (and less regular) domains in terms of the global heat kernel and exit probabilities of the associated L\'evy process. Further significant contributions include \cite{Chen2012,Chen2014, Bogdan2014,Kim2018,Grzywny2020,Chen2025}, where analogs of \eqref{eq:two-sided-intro} are studied for L\'evy operators which may be of variable order or include drift terms. Moreover, we mention \cite{Kim2022,Kim2023}, which focus on operators whose kernels degenerate or explode at $\partial \Omega$, and \cite{Kim2014b,Chen2022,Song2025}, where \eqref{eq:two-sided-intro} is established in less regular domains. 

Note that in all of the aforementioned works either rotational symmetry of the kernel is assumed, or  the authors impose certain additional structural assumptions under which the operator becomes a perturbation of a homogeneous operator in non-divergence form. In light of these limitations of the theory, we are led to ask the following natural question:
\begin{align*}
    \text{Does \eqref{eq:two-sided-intro} hold for } &\text{nonlocal operators in divergence form \eqref{eq:OperatorDivergenceForm}, \eqref{eq:KernelDivergenceForm}, \eqref{eq:KernelHoelderCont}} \\
   &\text{in $C^{1,\alpha}$ domains for any } \alpha \in (0,s)?
\end{align*}
An analogous question can also be posed for nonlocal elliptic equations and the Green's function, where it was recently resolved in \cite{Kim2024}. By \emph{applying the elliptic results} in \cite{Kim2024} and establishing suitable approximate factorization formulas, in the recent paper \cite{Cho2025}, the authors prove \eqref{eq:two-sided-intro} for the special class of \emph{time-independent} operators $-\mathcal{L}_t = -\mathcal{L}$ in $C^{1,\alpha}$ domains. This gives a positive partial answer to the aforementioned question. Moreover, they establish similar estimates in $C^{1,\text{dini}}$ domains under rather strong structural assumption on the coefficients that effectively reduces $-\mathcal{L}$ to a perturbation of a homogeneous operator in non-divergence form.

In this article, we settle the aforementioned question by \emph{developing a fully parabolic generalization} of the approach in \cite{Kim2024}. This allows us to treat time-dependent operators. Our main result reads as follows:

\begin{Theorem} \label{Thm:HeatKernelBounds}
    Let $s\in (0,1)$, $\alpha ,\sigma \in (0,s)$.
    Furthermore, let $\Omega \subset \R^d$ be a bounded $C^{1,\alpha}$ domain and let $\mathcal{L}_t$ for be an operator of the form \eqref{eq:OperatorDivergenceForm}-\eqref{eq:KernelDivergenceForm} satisfying \eqref{eq:KernelHoelderCont} with $\mathcal{A}= (0,1) \times \Omega '$ for some $\Omega \Subset \Omega ' \subset \R^d$.
    Then, for all $x,y\in \Omega$ and $t\in (0,1)$, the Dirichlet heat kernel $p_\Omega$ satisfies
    \begin{equation*}
        p_\Omega (t,x,y) \asymp _c \left( 1 \wedge \frac{d_\Omega (x)}{t^\frac{1}{2s}} \right) ^s  \left( 1 \wedge \frac{d_\Omega (y)}{t^\frac{1}{2s}} \right) ^s \left( t^{-\frac{d}{2s}} \wedge\frac{t}{|x-y|^{d+2s}} \right)
    \end{equation*}
    for some comparability constant $c=c(d,s, \sigma , \lambda ,\Lambda , \alpha , \Omega , \Omega ' )>0$.
\end{Theorem}

We emphasize that \autoref{Thm:HeatKernelBounds} includes nonlocal operators whose kernels depend on time. This is in stark contrast to all previous results on Dirichlet heat kernel estimates in the literature. Note that the probabilistic approach in \cite{Cho2025} heavily uses the Green's function estimates from \cite{Kim2024} and their relation to exit time estimates for the associated stochastic process, which does not extend in a straightforward way to the time-dependent setting.

While previous articles mainly use probabilistic tools to establish \eqref{eq:two-sided-intro}, our approach is purely analytic. The two-sided estimates follow by a combination of the $C^s_x$ boundary regularity and the Hopf lemma from \autoref{Thm:CsBoundaryRegAndHopfLemma} together with a rather delicate barrier argument and the interior estimates for the heat kernel from \cite{Liao2025}.

\subsection{Strategy of the proofs}

The study of boundary regularity for nonlocal equations is arguably more delicate than in the local case. This is mainly because well-established techniques based on flattening the boundary are not readily applicable to nonlocal problems. Moreover, since solutions are in general not better than $C^s$ up to the boundary, there is a clear qualitative discrepancy in their behavior, which requires appropriate modifications of the standard techniques.

Our proof of the optimal $C^s_p$ regularity in \autoref{Thm:CsBoundaryRegAndHopfLemma}(i) is based on a Morrey-Campanato-type iterative scheme at the boundary, which we develop in the framework of parabolic nonlocal equations. Such an approach was introduced for nonlocal elliptic problems in the recent article \cite{Kim2024} and here we establish its complete parabolic generalization. However, as was already mentioned before, nonlocal effects become more critical in parabolic problems and might also lead to new phenomena (see \autoref{example:counterex-2}, or \cite[Example 5.2]{Kassmann2024}). In particular, any study of regularity results for parabolic problems requires a careful analysis of the nonlocal tail terms. 

In the following, we briefly sketch the main ideas of the proof of \autoref{Thm:CsBoundaryRegAndHopfLemma}(i), focusing on the difficulties that arise specifically from the \emph{parabolic nature of the problem}. We refer to \cite[Subsection 1.3]{Kim2024} for a more detailed description of the approach in the elliptic case.

\subsubsection{Freezing kernels and optimal regularity in the translation invariant case}

In order to prove regularity for solutions to \eqref{eq:PDE-intro} at a point $z_0 = (t_0,x_0) \in (-1,0) \times \Omega$, we freeze the coefficients at $z_0$ and consider the frozen operator $L_{z_0}$ with kernel
\begin{align*}
    K_{z_0}(t,x,y) = \frac{1}{2} \left( K(t_0,x_0+x-y,x_0) + K(t_0,x_0 +y-x,x_0) \right).
\end{align*}

Then, on any scale $R \in (0,1)$, we compare the solution $u$ to \eqref{eq:PDE-intro} with its $L_{z_0}$-caloric replacement $v$, which is given as the unique solution to
\begin{equation}
\label{eq:intro-replacement}
\left\{ \begin{aligned}
        \partial _t v -L_{z_0} v &=0 && \text{in }  (t_0 - R^{2s} , t_0 )  \times (\Omega \cap B_R(x_0)) , \\
           v &= u && \text{in }  (t_0 - R^{2s} , t_0 ) \times  \big(\R^d \setminus (\Omega \cap B_R(x_0)) \big) ,\\
           v &= u && \text{on } \{ t_0 - R^{2s}\} \times \Omega.
        \end{aligned} \right.
    \end{equation}

As usual, the goal is to transfer the regularity from $v$ to $u$.

\begin{Remark}
    Note that the regularity assumption \eqref{eq:KernelHoelderCont} on the kernel is crucial for the comparison between $u$ and $v$ to work. We emphasize that \eqref{eq:KernelHoelderCont} is a rather weak assumption, since it only requires regularity in $t$ and of the shift $h \mapsto K(t,x + h , y + h)$ for given $t,x,y$. This is much less restrictive than assuming regularity of $K$ in all variables (as in \cite[(1.9)]{Byun2025}), or than assuming regularity of $h \mapsto K(t,x,x+h)$ (as in \cite{Kim2014b,Grzywny2020,Kim2023}). In both of these cases, the frozen kernel is always the fractional Laplacian or essentially a homogeneous kernel \eqref{eq:hom-kernel}, whereas in our case, it can be \emph{any} translation invariant kernel satisfying \eqref{eq:KernelDivergenceForm}. 
\end{Remark}

First, we establish the optimal $C^s_p$ regularity for parabolic nonlocal equations \eqref{eq:PDE-intro} governed by the frozen operator $-L_{z_0}$. Note that this was previously only known in case of a homogeneous kernel $K_{z_0}$ \eqref{eq:hom-kernel} (see \cite{FernandezReal2017}).
Our proof relies on a rather general barrier argument that allows to deduce boundary regularity for parabolic problems from elliptic problems. In particular, here we use the $C^s$ regularity for nonlocal elliptic equations from \cite{RosOton2024} (see \autoref{Lem:BoundaryRegFromBoundOnU}).

Next, we employ a blow-up argument to establish an expansion of order $s + \eps$ for any $\eps \in (0,s)$ for solutions $v$ to \eqref{eq:intro-replacement} of the form
\begin{equation}
\label{eq:ExpansionTranslInvariant-intro}
    |v(t,x)-q_0 \psi_{z_0}(x)| \leq C \left( |x|^{s+\eps} + |t| ^\frac{s+\eps}{2s} \right) \qquad \forall (t,x) \in (-1,0)\times B_1
\end{equation}
for some $q_0 \in \R$ (see \autoref{Prop:ExpansionTranslationalInvariant}). Here, $\psi_{z_0}$ is the solution of an elliptic exterior value problem governed by $- L_{z_0}$ that is in $C^s$ and behaves like $d_{\Omega}^s$ at the boundary of $\Omega$ due to the results in \cite{RosOton2024}. The same barrier $\psi_{z_0}$ was already considered in \cite{Kim2024}. Note that we prove \eqref{eq:ExpansionTranslInvariant-intro} for any $\eps < s$, as opposed to \cite{Kim2024}, where it was only shown for $\eps < \alpha s$.

\subsubsection{Optimal regularity for equations in divergence form}

As in \cite{Kim2024}, we measure the behavior of a solution $u$ to \eqref{eq:PDE-intro} at a point $z_0 = (t_0,x_0)$ on scale $\rho$ in terms of suitable nonlocal excess functionals, given by
\begin{align*}
    \Phi _{\sigma}(u;\rho ,z_0)& = \iint_{Q ^\Omega_\rho (z_0) } \left| \frac{u}{d_\Omega ^s}  \right|\intd x \intd t \\
    & \quad + \max \{ \rho ,d_\Omega (x_0)\} ^{-s} \rho ^{d+2s} \sup_{t\in I^\ominus _\rho (t_0)} \Tail _{\sigma ,B_1} (u;\theta ^{-1}\rho  ,x_0) ,\\
    \Psi(u;\rho,z_0) &:= \iint_{Q_{\rho}^{\Omega}(z_0)} \left| \frac{u}{\psi_{z_0}} - \left( \frac{u}{\psi_{z_0}} \right)_{Q_{\rho}^{\Omega}(z_0)} \right| \d x \d t \\
    &\quad + \max\{ \rho , d_{\Omega}(x_0) \}^{-s} \rho^{d+2s} \sup_{t \in I_{\rho}^{\ominus}(t_0)} \Tail \left(u - \psi_{z_0} \left( \frac{u}{\psi_{z_0}} \right)_{Q_{\rho}^{\Omega}(z_0)} ; \theta^{-1} \rho , x_0 \right).
\end{align*}
Here, $\theta = 128$ is a large enough number, the tail is defined in \eqref{eq:tail-def}, and 
\begin{align*}
    \Tail_{\sigma,B_1}(u(t, \cdot );\rho,x_0) = \rho^{2s - \sigma} \int_{B_1 \setminus B_{\rho}(x_0)} |u(t,y)| |y - x_0|^{-d-2s+\sigma} \d y
\end{align*}
is a modification of $\Tail$ taking into account the $C^{\sigma}$ regularity of $K$.

The key difficulty in the study of nonlocal parabolic problems becomes apparent when comparing $\Phi_{\sigma}$ and $\Psi$ to their analogs in the elliptic case (see \cite{Kim2024}). In fact, while all the quantities in $\Phi_{\sigma}$ and $\Psi$ are $L^1$ in $x$, the integrability in $t$ of the local terms differs from the one of the tail terms. 
This is due to a competition between small and large scales that is at the core of the difficulties in regularity for parabolic nonlocal problems. On the one hand, we are forced to keep the order of the local terms as low as possible (below the energy level $L^2$) in order to apply the comparison estimates for $u$ and its replacement $v$. On the other hand, we must consider at least $L^{2+\eps}_t$ tails since $C^s_x$ regularity fails unless the tails possess sufficient integrability (see \autoref{example:counterex-2}). 

Note that any kind of excess decay estimate from scale $\rho$ to scale $R$ (where $0 < \rho \ll R \le 1$) for nonlocal problems requires a fine decomposition of the tail terms into their intermediate and purely nonlocal contributions on scales $[\rho,R]$ and $(R,\infty)$, respectively (see \cite{Nowak2022,Diening2024,Nguyen2024,Diening2025b,Kim2024}). To do so we split the tail terms as follows:
\begin{align}
\label{eq:tail-split-intro}
\begin{split}
    \sup_{I_{\rho}^{\ominus}(t_0)} \Tail(u ; \rho , x_0) &\le C \sum_{k = 0}^{[\log R/\rho]} 2^{-2sk} \sup_{I_{\rho}^{\ominus}(t_0)} \dashint_{B_{2^{k+1}\rho}(x_0) \setminus B_{2^{k}\rho} (x_0)} |u(y)| \d y \\
    &\quad + C(\rho/R)^{2s} \sup_{I_{\rho}^{\ominus}(t_0)} \Tail(u; R , x_0).
    \end{split}
\end{align}
The contributions on intermediate scales have to be estimated by using the underlying PDE, which requires them to possess the same integrability in $t$ as the local quantities within the excess functionals $\Phi_{\sigma}, \Psi$.

The key idea to overcome the apparent incompatibility between the integrability of the local and nonlocal contributions within $\Phi_{\sigma}, \Psi$ is to prove stronger excess decay estimates for the frozen solutions $v$ of the following form (see \autoref{Lem:MorreyEstimateTranslInvariant} and \autoref{Lem:CampanatoEstimateTranslInvariant}):
\begin{align*}
    \sup _{t\in I^\ominus _\rho (t_0)} \int_{B ^\Omega_\rho (x_0) } \left| \frac{v}{d_\Omega ^s}  \right|\intd x & \leq c \rho ^{-2s}\left( \frac{\rho}{R}\right) ^{d+2s-\eps } \Phi_{\sigma}(v;R,z_0),\\
    \sup _{t \in I_\rho ^\ominus(t_0)} \int _{B_\rho (x_0)}  \left| \frac{v}{\psi_{z_0}} - \left( \frac{v}{\psi_{z_0} } \right) _{Q_\rho ^\Omega (z_0) } \right| \intd x  
            & \leq c  \rho ^{-2s}\left( \frac{\rho}{R} \right) ^{d+2s+\eps} \Psi(v;R,z_0).
\end{align*}
Note that these estimates not only encode an excess decay for $v$, but also a gain of integrability in the $t$-variable, when passing from a smaller to a larger scale. Clearly, both estimates crucially rely on the fact that we already know $v \in C^{s-\eps}_p$ and \eqref{eq:ExpansionTranslInvariant-intro}, respectively. 

By comparing $u$ to the frozen solution $v$ also on intermediate scales, we can profit from the gain of integrability in $t$ and thereby estimate the intermediate $L^{\infty}_t$-tails arising in \eqref{eq:tail-split-intro}. Note that this is a key difference compared to the elliptic case, where the intermediate tails can be treated in a more direct way.

Note that while the excess function $\Phi_{\sigma}$ measures the size of $u/d^s_{\Omega}$, $\Psi$ measures oscillations. Hence, $\Phi_{\sigma}$ will be used in Morrey-type iterative schemes (see \autoref{Lem:ExcessFunctionalEstimate}) to deduce regularity for $u$ of order $C^{s - \eps}_p$ (see \autoref{Thm:Cs-epsRegWithCoefficients}). On the other hand, $\Psi$ is employed in Campanato-type estimates (see \autoref{Lem:CampanatoEstimate}) in order to upgrade the regularity from $C^{s-\eps}_p$ to $C^s_p$ (see \autoref{Thm:CsBoundaryRegAndHopfLemma}(i)) and to establish higher order estimates such as \autoref{Thm:BoundaryHarnack}. In fact, by proving a quantified decay for the excess functional $\Psi$, and using closeness of the barriers $\psi_{z_0}$ and $\psi_{z_1}$ in $C^{s-\eps}$ in terms of $|z_0 - z_1|$  (see \autoref{Lem:ClosenessOfBarriers}), we will prove in \autoref{Lem:uOverPsiHoelderContinuous} that 
\begin{align*}
    \left( (t,x) \mapsto \frac{u(t,x)}{\psi_{(t,x)}(x)} \right) \in C^{\beta}_p(\overline{Q_{1/2}^{\Omega}})
\end{align*}
for any $\beta \le \min\{ \sigma, s -\frac{2s}{p} , s - \frac{2s}{q} - \frac{d}{r} \}$. By a combination with the Hopf lemma (see \autoref{Thm:CsBoundaryRegAndHopfLemma}(ii)), we deduce the boundary Harnack-type result in \autoref{Thm:BoundaryHarnack}.

\subsection{Outline}
In \autoref{sec:Preliminaries}, we introduce some notation and define the weak solution concept. Furthermore, we provide some known regularity results for solutions from the literature and prove an iteration lemma on dyadic scales.
Next, in \autoref{sec:TranslationInvariantCase} we deal with translation invariant, time independent operators and prove $C^s_p$-boundary regularity results as well as an expansion of order $s+\eps$.
We establish comparison estimates between translation invariant equations and equations with coefficients in \autoref{sec:FreezingEstimates}.
In \autoref{sec:Cs-epsRegularity}, we prove $C^{s-\eps}_p$-regularity for solutions in flat Lipschitz domains.
We continue by upgrading the $C^{s-\eps}_p$-regularity to the optimal $C^s_p$-regularity in \autoref{sec:CsRegularity}, and proving a Hopf lemma in \autoref{sec:HopfLemma}.
\autoref{sec:BoundaryHarnack} is devoted to the proof of \autoref{Thm:BoundaryHarnack}.
Finally, in \autoref{sec:DirichletHeatKernelBounds}, we prove the Dirichlet heat kernel bounds from \autoref{Thm:HeatKernelBounds}.

\subsection{Acknowledgments}

Philipp Svinger was supported by the Deutsche Forschungsgemeinschaft (DFG,
German Research Foundation) under Project-ID 317210226 – SFB 1283. \\
Marvin Weidner was supported by the European Research Council under the Grant Agreement No. 101123223 (SSNSD), by the AEI project PID2024-156429NB-I00 (Spain), and by the Deutsche Forschungsgemeinschaft (DFG,
German Research Foundation) under Germany’s Excellence Strategy - EXC-2047/1 - 390685813 and through the CRC 1720.

\section{Preliminaries}\label{sec:Preliminaries}
\subsection{Notation and weak solution concept}
We use the following notation to denote the mean value integral
\begin{equation*}
    (u)_A:= \dashint _A u(x) \intd x = |A| ^{-1} \int _A u(x) \intd x , \qquad A\subset \R^d .
\end{equation*}
We will write $a\wedge b := \min \{ a,b\}$ and $a\vee b := \max \{ a,b\}$.
Furthermore, for $a\in \R$ we use the notation $a=a_+ -a_-$ where $a_+:= a\vee 0$ is the positive part and $a_-:=- (0 \wedge a)$ is the negative part.
In addition, we write $a\lesssim b$ if there exists a positive constant $c>0$ such that $a\leq c b$.
We will use the notation $a\asymp b$ if $a\lesssim b$ and $b \lesssim a$.

For $z_0=(t_0,x_0) \in \mathbb{R} \times \mathbb{R}^d$, $R>0$, and a domain $\Omega \subset \mathbb{R}^d$, we set
    \begin{align*}
        I ^\ominus _{R}(t_0) &:= (t_0-R^{2s},t_0], & Q_R(z_0) &:= I ^\ominus _{R}(t_0) \times B_R(x_0), \\
        B_R^\Omega (x_0) &:= \Omega \cap B_R(x_0), & Q^\Omega_R(z_0) &:= I ^\ominus _{R}(t_0) \times B_R^\Omega(x_0).
    \end{align*}
By $d_\Omega$, we denote the distance to the complement $\Omega^c$, i.e., $d_\Omega (x) := \inf \{ |y-x| \mid y\in \Omega^c\}$.
Moreover, for some $\beta \in (0,2s]$, we define the parabolic Hölder seminorm
\begin{equation*}
    [u]_{C^{\beta} _{p}(Q_R(z_0))} :=\sup _{\substack{t,t'\in I^\ominus _R (t_0) \\ x,x' \in B_R (x_0)}} \frac{|u(t,x)-u(t',x')|}{|t-t'|^{\frac{\beta}{2s}} + |x-x'|^\beta},
\end{equation*}
and use the subscript $t$ and $x$ to define $[u]_{C^\beta_x(I\times\Omega)} := \sup_{t\in I} [u(t,\cdot)]_{C^\beta(\Omega)}$ and $[u]_{C^{\beta/2s}_t(I\times\Omega)} := \sup_{x\in\Omega} [u(\cdot,x)]_{C^{\beta/2s}(I)}$.

The mixed Lebesgue norm is given by
\begin{equation*}
        \| f\| _{L_{t}^q L_x^{r} (Q_R(z_0))} :=  \left( \int _{I ^\ominus _{R}(t_0)} \| f(t,\cdot)\|_{L^r(B_R(x_0))}^q \intd t \right) ^{\frac{1}{q}} .
    \end{equation*}

For $\Omega \subset \Omega ' \subset \R^d$, we define
\begin{equation*}
    V^s(\Omega | \Omega ') := \left\{  v\in L^2 (\Omega ) \colon [v]^2_{V^s (\Omega | \Omega ')} := \iint _{\Omega ' \times \Omega ' \setminus (\Omega ^c \times \Omega ^c)} \frac{(v(x)-v(y))^2}{|x-y|^{d+2s}} \intd x \intd y <\infty\right\} 
\end{equation*}
for some $s\in (0,1)$, and equip this space with the norm
\begin{equation*}
     \| v\| _{ V^s(\Omega | \Omega ')} = \left( \| v\| ^2 _{L^2 (\Omega)} +[v]^2_{V^s (\Omega | \Omega ')} \right) ^{\frac{1}{2}} .
\end{equation*}
If $\Omega = \Omega '$, then we write $H^s(\Omega ):=V^s(\Omega | \Omega ) $ and denote the seminorm by $[v]_{H^s(\Omega)} = \| v\| _{\dot{H}^s(\Omega )} := [v]_{V^s (\Omega | \Omega )}$.
Furthermore, we define $H^s_\Omega (\R ^d ):= \{ v \in H^s(\R ^d)  \, : \, v \equiv 0 \text{ in }\R^d \setminus \Omega \}$.

For a given kernel $K(t,x,y)$ satisfying \eqref{eq:KernelDivergenceForm}, we define the bilinear form
\begin{equation*}
    \mathcal{E}_{t} (u,v) = \frac{1}{2}\int _{\R^d} \int _{\R^d}  (u(x)-u(y)) (v(x)-v(y)) K(t,x,y)  \intd y \intd x
\end{equation*}
for measurable functions $u,v \colon \R^d \to \R$.

Furthermore, let $L^1_{2s}(\R^d )$ be the tail space defined by
\begin{equation*}
    L^1_{2s}(\R^d ) := \left\{ v \in L^1_{loc}(\R^d ) \, : \, \| v\| _{L^1_{2s}(\R^d )} :=\int _{\R ^d} \frac{|v(y)|}{(1+|y|)^{d+2s}} \intd y < \infty \right\}.
\end{equation*}
Note that for $v\in L^1_{2s}(\R^d )$, $\Tail (v; R,x_0)$, as defined in \eqref{eq:tail-def}, is finite for any $x_0\in \R^d$, $R>0$.
\begin{Definition}\label{Def:WeakSolConcept}
     Let $\mathcal{L}_t$ be an operator of the form \eqref{eq:OperatorDivergenceForm}-\eqref{eq:KernelDivergenceForm}.
     Given an interval $I\subset \R$, a bounded domain $\Omega \Subset \Omega ' \subset  \R^d$, and a function $f\in ( L^2(I;H^s _\Omega (\R^d)) )^\ast $, we say that a function
    \begin{equation*}
        u\in L^\infty (I;L^2 (\Omega )) \cap L^2(I;V^s(\Omega | \Omega ')) \cap L^1(I;L^1_{2s}(\R ^d))
    \end{equation*}
    is a weak solution to
    \begin{equation*}
        \partial_t u - \mathcal{L}_tu =f \quad \text{in} \quad I \times \Omega
    \end{equation*}
    if for any $\varphi \in H^1(I;L^2 (\Omega)) \cap L^2(I;H^s _\Omega (\R^d))$ with $\supp (\varphi) \Subset I \times \R ^d$ it holds
    \begin{equation*}
        - \int _I \int _{\R^d} u(t,x) \partial_t \varphi (t,x) \intd x \intd t + \int _I \mathcal{E}_{t} (u(t,\cdot ),\varphi (t,\cdot )) \intd t = \int_I \int _{\R^d} f(t,x) \varphi(t,x)   \intd x \intd t .
    \end{equation*}
\end{Definition}

\begin{Remark}
    Note that the solution concept in \autoref{Def:WeakSolConcept} is particularly well-suited for studying the boundary regularity of solutions to $\partial_t u - \mathcal{L}_tu = f$ in $I \times \Omega$.
    The key reason is the assumption on the spatial regularity of $u$, namely that $u(t, \cdot) \in V^s(\Omega | \Omega ')$ rather than the standard $H^s_{loc}(\Omega)$ space used for interior problems. This is a natural condition for analyzing the boundary behavior, as discussed in \cite[pp. 5-6]{Kim2023a}.
 Indeed, if a solution also satisfies the zero exterior condition
    \begin{equation*}
     u=0 \qquad \text{in}\quad I \times \Omega ^c,
    \end{equation*}
    this framework permits testing with functions of the form $\varphi = \eta u$, where $\eta \in C^\infty_c(I)$ is a smooth cutoff function in time, after an appropriate time mollification (see, e.g., \cite[Appendix B]{Liao2023}).
\end{Remark}

\begin{Remark}
If $\Omega$ is unbounded, then we say that $u$ is a solution to $\partial_t u - \mathcal{L}_tu =f$ in $I \times \Omega$, if for any  bounded $\tilde {\Omega }  \subset \Omega$, $u$ is a solution in $I\times \tilde{ \Omega }$ in the sense of \autoref{Def:WeakSolConcept}.
\end{Remark}

As mentioned in the introduction, in order to capture the Hölder continuity of the kernel $K$ from \eqref{eq:KernelHoelderCont} for some $\sigma \in (0,s)$, we use the $\sigma$-Tail
\begin{equation*}
    \Tail _{\sigma ,B_1} (u(t,\cdot );R,x_0) = R^{2s-\sigma} \int_{B_1 \setminus B_R(x_0)} |u(t,y)| |y-x_0| ^{-d-2s+\sigma} \intd y
\end{equation*}
which was introduced in \cite{Kim2024}.
Note that $\Tail _{\sigma ,B_1} (u;R,x_0) \leq R^{-\sigma} \Tail  (u;R,x_0)$, hence, $\Tail _{\sigma ,B_1} (u;R,x_0)$ is finite for every $u\in L^{1}_{2s} (\R^d )$.
Furthermore, we have the estimate
\begin{equation} \label{eq:TailSplitSigmaTail}
     \Tail (u;R,x_0) \leq \Tail _{\sigma ,  B_1} (u;R,x_0) + cR^{2s}\| u\|_{L^\infty (\R^d )} .
\end{equation}

\subsection{Auxiliary results for weak solutions}

We start with the following Caccioppoli-type estimate.
\begin{Proposition}\label{Prop:CaccioppoliEstimate}
    Let $s\in (0,1)$, and $\Omega \subset \R^d$ be a domain.
    Let $z_0=(t_0,x_0) \in \R \times \R^d$ and $R>0$.
    Furthermore, let $\mathcal{L}_t$ be an operator of the form \eqref{eq:OperatorDivergenceForm}-\eqref{eq:KernelDivergenceForm}.
    Assume that $u$ is a weak solution of 
    \begin{equation*}
        \left\{ \begin{aligned}
          \partial _t u -\mathcal{L}_tu & =f & &\text{in} \quad I ^\ominus_R(t_0) \times ( B_R (x_0) \cap \Omega )  , \\
           u & =0 & & \text{in} \quad  I ^\ominus_R(t_0) \times ( B_R (x_0) \setminus \Omega ) 
        \end{aligned} \right.
    \end{equation*}
    for some $f\in L^q_tL^r_x (Q_R^\Omega (z_0))$ where $q,r\geq 1$ satisfy $\frac{1}{q}+\frac{d}{2sr} <1$.      
    Then,
    \begin{align*}
        \| u\|_{L^2_t \dot{H}_x^s(Q_{R/2} (z_0))} +\| u\|_{L^\infty_t L_x^2(Q_{R/2} (z_0))} & \leq c R^{\frac{d}{2}} \Bigg( R^{-\frac{d}{2}-s}\| u\| _{L^2 (Q_R(z_0))} + \sup _{t\in I_R^\ominus (t_0)} \Tail  (u;R/2,x_0)    \\
        & \qquad \qquad \qquad+ R^{2s -\frac{d}{r}-\frac{2s}{q}}\| f\| _{L^{q}_t L^r_{x} (Q_R^\Omega (z_0))} \Bigg) 
    \end{align*}
    for some constant $c=c(d,s,\lambda ,\Lambda ,q,r )>0$.
\end{Proposition}
\begin{proof}
    The proof follows using standard techniques.
    Let $\xi$ be a smooth cut-off function in space with $\ind _{B_{R/2}(x_0)}\leq \xi \leq \ind _{B_{3R/4}(x_0)}$ and $\| \nabla \xi \| _{L^\infty (\R^d)} \lesssim R^{-1}$.
    Given $\tau \in I_{R/2}^\ominus (t_0)$, let $\eta _k$ be a sequence of cut-off functions in time which for $k\to \infty$ approximates a function $\eta$ satisfying $\eta = 1$ on $[t_0-(R/2)^{2s},\tau]$, $\eta = 0$ on $(\tau , t_0 ]$ and $\| \eta ' \| _{L^\infty (t_0-R^{2s},t_0-(R/2)^{2s})} \lesssim R^{-2s}$.
    Note that we can test the equation (after appropriate time mollifications, see, e.g., \cite[Lemma 3.1]{Byun2023}) with $u \xi ^2 \eta _k$.
    From here, one can prove \autoref{Prop:CaccioppoliEstimate} with standard techniques.
    See, for example \cite[Lemma 3.2, p.13]{Byun2023} where also a right-hand side in $L^q_tL^r_x$ was considered (see also \autoref{Lem:EstimateOnHsw}).
    \end{proof}

Next, we have the following local boundedness result up to the boundary.

\begin{Proposition} \label{Prop:LocalBoundedness}
    Under the same assumptions as in \autoref{Prop:CaccioppoliEstimate}, we have
    \begin{equation}\label{eq:LocalBoundedness}
    \begin{aligned}
        \| u\| _{L^\infty (Q_{R/2}(z_0))} & \leq c \Bigg( R^{-d-2s}\| u \|  _{L^1(Q_R(z_0))}  +  \sup _{t\in I^\ominus _R (t_0)}\Tail  (u;R/2,x_0)    \\
        & \qquad \qquad+ R^{2s -\frac{d}{r}-\frac{2s}{q}}\| f\| _{L^{q}_t L^{r}_{x} (Q_R^\Omega (z_0))} \Bigg) 
    \end{aligned}
    \end{equation}
    for some constant $c=c(d,s,\lambda ,\Lambda ,q,r )>0$.
\end{Proposition}
\begin{proof}
    In \cite[Theorem 1.1]{Byun2023}, \eqref{eq:LocalBoundedness} was proven for interior solutions, i.e. in the case $B_R(x_0) \subset \Omega $, and with the $L^1$-norm $R^{-d-2s}\| u \|  _{L^1(Q_R(z_0))}$ on the right-hand side of \eqref{eq:LocalBoundedness} replaced by $R^{-d/2-s}\| u \|  _{L^2(Q_R(z_0))}$.
    First, note that by the same arguments as in the proof of \autoref{Prop:CaccioppoliEstimate}, the proof of \cite[Theorem 1.1]{Byun2023} can be adapted to also hold up to the boundary of $Q_{R/2}^\Omega(z_0)$.
    Second, note that by a standard iteration argument the $L^2$-norm of $u$ on the right-hand side of the estimate can be replaced by the $L^1$-norm (see, e.g., \cite[Theorem 6.2]{KassmannWeidner2022.2} or \cite[(5.1)]{Kassmann2024}).
\end{proof}
 
Finally, let us recall an interior regularity result.

\begin{Proposition} \label{Prop:InteriorRegularity}
Let $\mathcal{L}_t$ be an operator of the form \eqref{eq:OperatorDivergenceForm}-\eqref{eq:KernelDivergenceForm}, satisfying \eqref{eq:KernelHoelderCont} with $\mathcal{A}=Q_R(z_0)$.
Assume that $u$ is a weak solution of 
\begin{equation*}
    \partial _t u - \mathcal{L}_tu =f \quad \text{in}\quad Q_R(z_0) 
\end{equation*}
for some $f\in L^q_tL^r_x (Q_R(z_0))$ where $q,r\geq 1$ satisfy $\frac{1}{q}+\frac{d}{2sr} <1$.
Furthermore, let $\gamma$ be a positive number satisfying $\gamma <1$ and $\gamma \leq 2s- \left( \frac{d}{r}+\frac{2s}{q} \right) $.
Then, 
\begin{align*}
    [u]_{C^{\gamma}_{p} (Q_{R/2}(z_0))} & \leq c  R^{-\gamma }\Bigg( R^{-\frac{d}{2}-s} \| u\|_{L^2(Q_R(z_0)} + \sup _{t\in I^\ominus _R(t_0)} \Tail (u ;R/2 ,x_0)  \\
    & \qquad \qquad +R^{2s- \left( \frac{2s}{q} + \frac{d}{r}  \right)} \| f\| _{L^{q}_{t} L^r_x(Q_R(z_0))}\Bigg)
\end{align*}
for some constant $c=c(d,s,q,r, \lambda , \Lambda ,\sigma ,\gamma )>0$.
\end{Proposition}
\begin{proof}
    First, note that \autoref{Prop:InteriorRegularity} was proven in \cite[Theorem 1.2]{Byun2023} in the case $\gamma \in (0, \min \{ 2s- \left( \frac{d}{r}+\frac{2s}{q} \right)  ,1 \}$ under slightly more general assumptions than \eqref{eq:KernelHoelderCont}.
    
    Note that the endpoint case, namely $\gamma = 2s- \left( \frac{d}{r}+\frac{2s}{q} \right) <1$, also follows from the proof of \cite{Byun2023}, without any change. Alternatively, one could proceed for instance as in the proof of \cite[Theorem 1.7]{Diening2025}.
\end{proof}

\subsection{An iteration lemma}

We need the following version of the iteration lemma from \cite[Lemma 5.13]{Giaquinta2012}, which only takes into account radii on geometric scales $\theta^{-m}$ for some $\theta > 1$.

\begin{Lemma}\label{Lem:IterationLemmaOnDyadicScale}
    Let $A,B,\alpha , \beta, R_0 >0$, $\theta >1$ and $\eps \geq 0 $ be constants such that $\alpha >\beta$.
    Furthermore, let $\Phi \colon \R_+ \to [0, \infty )$ be a function satisfying
    \begin{equation}\label{eq:AssumptionIterationLemmaDyadic}
        \Phi (\theta^{-m}R ) \leq A \left[ \left(  \theta^{-m}\right) ^\alpha + \eps \right] \Phi(R) +BR^\beta \quad \text{for all }0<R \leq R_0 \text{ and } m\in \N .
    \end{equation}
    Then, there exist $\eps _0 >0 $ and $c>0 $ depending on $A,\alpha ,\beta ,\theta$ such that if $\eps \leq \eps _0$, then we have
    \begin{equation}\label{eq:IterationLemmaDyadic}
        \Phi (\theta^{-m}R ) \leq c \left[ \frac{\Phi (R)}{R^\beta} +B \right] \left( \theta^{-m}R\right) ^\beta \quad \text{for all } 0 <R \leq R_0 \text{ and } m\in \N . 
    \end{equation}
\end{Lemma}
\begin{proof}
    Set $\gamma := \frac{\alpha +\beta}{2}$ and choose $k=k(A,\alpha ,\beta ,\theta ) \in \N$ such that $2A\theta^{-k\alpha} \leq \theta^{-k \gamma}$. Furthermore, set $\tau =\theta^{-k}$, so that $2A\tau ^\alpha \leq \tau ^\gamma$.
    Now, fix $0 <R \leq R_0$ and $m \in \N$. The goal is to prove \eqref{eq:IterationLemmaDyadic} for this pair $( m ,R)$.
    First, we will choose $l,i\in \N$ with $i<k$ such that $m=lk+i$, i.e., $\theta^{-m}R =\tau ^l \theta^{-i}R$.
    By \eqref{eq:AssumptionIterationLemmaDyadic}, we have 
    \begin{equation}\label{eq:ProofIterationLemmaRephrasAssumpDyadic}
        \Phi (\tau ^j R) \leq \tau ^\alpha A \left[ 1+\frac{\eps}{\tau ^\alpha} \right] \Phi (\tau ^{j-1}R) +B \tau ^{(j-1)\beta} R^\beta 
    \end{equation}
    for every $j\in \N$.
    Next, we choose $\eps _0 >0$ such that $\frac{\eps _0}{\tau ^\alpha} <1$.
    Then, for $\eps \leq \eps _0$, \eqref{eq:ProofIterationLemmaRephrasAssumpDyadic} becomes
    \begin{equation*}
        \Phi (\tau ^j R) \leq \tau ^\gamma  \Phi (\tau ^{j-1}R) +B \tau ^{(j-1)\beta} R^\beta .
    \end{equation*}
    Iterating this estimate for all $j\in \{ 1, \dots , l \}$, yields
    \begin{align*}
        \Phi (\tau ^l R) & \leq \tau ^\gamma  \Phi (\tau ^{l-1}R) +B \tau ^{(l-1)\beta} R^\beta \\
        & \leq \tau ^{l\gamma} \Phi (R) + B \tau ^{(l-1)\beta} R^\beta \sum _{j=0}^{l-1} \tau ^{j(\gamma - \beta)}\\
        & \leq \left[ \tau ^{-\beta} +\tau ^{-2\beta } \sum _{j=0}^\infty  \tau ^{j(\gamma -\beta )} \right] \tau ^{(l+1)\beta } \left( \Phi (R) +BR^\beta \right) \\
        & = c_1 \tau ^{(l+1)\beta} (\Phi (R) + B R^\beta )
    \end{align*}
    for some constant $c_1=c_1(A,\alpha ,\beta ,\theta )>0$.
    Using $\theta^{-m}R = \tau ^l \theta^{-i}R$ and $i<k$, this yields
    \begin{equation}\label{eq:IterationDyadicFirstEstimate}
        \Phi (\tau ^l R) \leq c_1 \tau ^\beta \theta^{i \beta}\left[ \frac{\Phi (R)}{R^\beta} +B  \right] \left( \theta^{-m}R\right) ^\beta  \leq c_1 \tau ^\beta \theta^{k \beta}\left[ \frac{\Phi (R)}{R^\beta} +B  \right] \left( \theta^{-m}R\right) ^\beta .
    \end{equation}
    Moreover, by applying \eqref{eq:AssumptionIterationLemmaDyadic} and assuming $\eps \leq \eps _0 \leq 1$, we obtain
    \begin{equation} \label{eq:IterationDyadicSecondEstimate}
    \begin{aligned}
        \Phi ( \theta^{-m}R ) =\Phi (\theta^{-i}\tau ^l R) &\leq A \left[ \theta^{-i\alpha} + \eps \right] \Phi (\tau ^l R) + B \tau ^{l\beta }R^\beta \\
        & \leq 2A \Phi (\tau ^l R) + B\theta^{k\beta}\left( \theta^{-m}R\right) ^\beta.
    \end{aligned}
    \end{equation}
    Combining \eqref{eq:IterationDyadicFirstEstimate} and \eqref{eq:IterationDyadicSecondEstimate}, and observing that $\theta^{k\beta} \le C(A,\alpha,\beta,\theta) < \infty$ proves \eqref{eq:IterationLemmaDyadic}.
\end{proof}

\section{Translation invariant case}\label{sec:TranslationInvariantCase}
In this section, we will deal with translation invariant operators of the form
\begin{equation} \label{eq:TranslationalInvariantOperator}
   - Lu(t,x)= \pv \int _{\R ^d} (u(t,x)-u(t,x+h)) K(h) \intd h
\end{equation}
where $K\colon \R^d \to [0,\infty ]$ satisfies for some $0<\lambda \leq \Lambda$,
\begin{equation} \label{eq:TranslationalInvariantKernel}
    0<\frac{\lambda}{|h|^{d+2s}} \leq K(h) \leq \frac{\Lambda}{|h|^{d+2s}} , \quad K(h)=K(-h), \quad s\in (0,1) .
\end{equation}

\subsection{Regularity estimates via barrier functions}

We will use the following lemma to deduce boundary regularity estimates for parabolic equations from boundary regularity for elliptic equations and interior estimates.

\begin{Lemma} \label{Lem:BoundaryRegFromBoundOnU}
    Let $\Omega \subset \R^d$ be a Lipschitz domain with Lipschitz constant $\delta \geq 0$ and $0 \in \partial \Omega$, and let $L$ be a translation invariant operator of the form \eqref{eq:TranslationalInvariantOperator}-\eqref{eq:TranslationalInvariantKernel}.
    Let $D \subset B_1$ be a domain with $B_{4/5} \setminus \Omega \subset D \subset B_{5/6} \setminus \Omega$ and set $\tilde{ \Omega} := B_3 \setminus D$.
    Assume that there exists a solution $\phi$ to the following elliptic problem
    \begin{equation*}
        \left\{ \begin{aligned}
           -L\phi & \geq 1 & &\text{in} \quad \tilde{\Omega} , \\
           \phi & \geq 0 & & \text{in} \quad  \R^d \setminus \tilde{\Omega} , \\
        \end{aligned} \right.
    \end{equation*}
    which satisfies $0 \leq \phi  \leq c_1 d_{\tilde{\Omega}} ^\kappa$ in $\tilde{\Omega}$ and $ c_2 ^{-1}\leq \phi$ in $B_2 \setminus B_1$ for some $c_1 >0$, $c_2\geq 1$, and $\kappa \in (0,s]$.

    Then, given $f\in L^\infty ((-1,0)\times (\Omega \cap B_1))$, any bounded weak solution $u$ to
    \begin{equation*}
        \left\{ \begin{aligned}
          \partial _t u -Lu & =f & &\text{in} \quad (-1 ,0) \times (\Omega \cap B_1) , \\
           u & =0 & & \text{in} \quad  (-1 ,0) \times (B_1 \setminus \Omega)  
        \end{aligned} \right.
    \end{equation*}
    satisfies
    \begin{equation} \label{eq:C-kappa-boundary-regularity}
         \|  u \| _{C^{\kappa}_{p} (\overline{Q^\Omega_{1/2}})} \leq c \left( \| u \| _{L^\infty ((-1,0)\times \R ^d)} + \| f\| _{L^\infty (Q^\Omega_1)} \right)
    \end{equation}
    for some constant $c=c(d,s, \lambda , \Lambda ,\delta , \kappa ,c_1 ,c_2 )>0$.
\end{Lemma}
\begin{proof}
Let $\overline{u}=u \ind _{B_2}$. Then, $\overline{u}$ solves
\begin{equation*}
        \left\{ \begin{aligned}
          \partial _t \overline{u} -L\overline{u} & =\overline{f} & &\text{in} \quad (-1 ,0) \times (\Omega \cap B_1) , \\
           \overline{u} & =0 & & \text{in} \quad  (-1 ,0) \times (B_1 \setminus \Omega)  ,
        \end{aligned} \right.
    \end{equation*}
    for some bounded $\overline{f}$ which satisfies $ \| \overline{f}\| _{L^\infty (Q^\Omega_1)} \leq c \left( \| u \| _{L^\infty ((-1,0)\times \R ^d)} + \| f\| _{L^\infty (Q^\Omega_1)} \right)$ .
After dividing the equation by some constant, we may assume $\| \overline{u} \| _{L^\infty ((-1,0)\times \R ^d)} + \| \overline{f}\| _{L^\infty (Q^\Omega_1)} \leq 1$.

\textbf{Step 1:}
First, we will show 
\begin{equation} \label{eq:kappa-bound-u}
    |u| \leq c d_\Omega ^\kappa \quad \text{in } Q^\Omega _{3/4} 
\end{equation}
for some constant $c=c(c_1, c_2 , s, d, \kappa )>0$.

Recall that $c_2 \phi \geq 1$ in $B_2 \setminus B_1$ for some $c_2\geq 1$.
Fix some $z_0=(t_0,x_0) \in Q^\Omega _{3/4}$. We will show \eqref{eq:kappa-bound-u} at $z_0$.
Set $k:=(t_0+1)^{-1}$ and define
\begin{equation*}
    v(t,x):= c_2 \phi (x) (1+k) +(1-(t+1)k) - \overline{u}(t,x) .
\end{equation*}
Then, we get
\begin{equation*}
        \left\{ \begin{aligned}
          \partial _t v -Lv & \geq  c_2 (1+k) -k -\overline{f} \geq 0 & &\text{in} \quad (-1 ,k^{-1 }-1] \times  B_1^\Omega , \\
           v & \geq 1- \overline{u}(-1,x) \geq 0 & & \text{in} \quad  \{ -1\} \times   B_1^\Omega   , \\
           v & \geq c_2 \phi (1+k) \geq 0 & & \text{in} \quad  (-1, k^{-1}-1] \times \left( \R^d \setminus B_2  \right) , \\
            v & \geq c_2 \phi (1+k) \geq 0 & & \text{in} \quad  (-1, k^{-1}-1] \times \left( B_1 \setminus B_1^\Omega  \right) , \\
            v & \geq c_2 \phi -\overline{u} \geq 0 & & \text{in} \quad  (-1, k^{-1}-1] \times \left( B_2 \setminus B_1  \right) . \\
        \end{aligned} \right.
    \end{equation*}
    By the maximum principle (see, e.g., \cite[Lemma 5.2]{Kassmann2023}), we conclude $v\geq 0$ in $(-1,k^{-1} ] \times B_1^\Omega$, which in particular yields
    \begin{equation*}
        0 \leq v(t_0 ,x_0) =c_2 \phi (x_0) \left( 1+\frac{1}{t_0+1} \right) -\overline{u}(t_0,x_0) \leq c \phi (x_0) -\overline{u}(t_0,x_0) .
    \end{equation*}
    This proves \eqref{eq:kappa-bound-u}.

\textbf{Step 2:}
In a second step, we will show how \eqref{eq:kappa-bound-u} together with interior regularity estimates yields \eqref{eq:C-kappa-boundary-regularity}.
Let $(t_1,x_1),(t_2,x_2) \in Q^\Omega _{1/2}$. Then,
\begin{equation*}
    \frac{|u(t_1,x_1)-u(t_2,x_2)|}{|t_1-t_2|^{\frac{\kappa}{2s}} + |x_1 -x_2| ^{\kappa}} \leq  \frac{|u(t_1,x_1)-u(t_1,x_2)|}{ |x_1 -x_2| ^{\kappa}} + \frac{|u(t_1,x_2)-u(t_2,x_2)|}{|t_1-t_2|^{\frac{\kappa}{2s}} } =:I_1+I_2 .
\end{equation*}
It remains to prove that $I_1$ and $I_2$ are bounded by a constant independent of $(t_1,x_1),(t_2,x_2)$.
We start by estimating $I_1$.
Let us assume that $r:= d_\Omega(x_1) \leq d_\Omega(x_2)$ and set $R=|x_1 -x_2|$.

\emph{Case 1:}  
If $r \leq 4R$, then by \eqref{eq:kappa-bound-u}
\begin{equation*}
    |u(t_1,x_1)-u(t_1,x_2)| \leq c d_\Omega ^\kappa (x_1) + c d_\Omega ^\kappa (x_2) \leq c d_\Omega ^\kappa(x_1) +c ( R +d_\Omega (x_1) )^\kappa \leq c R^\kappa .
\end{equation*}

\emph{Case 2:}
If $r> 4R$, then note that $x_2 \in B_{r/4} (x_1) \subset B_{r/2}(x_1) \subset B_{3/4}^\Omega$.
Furthermore, we can choose $\tilde{t}\in (-1,0)$ such that $(t_1,x_1)\in Q_{r/2}(\tilde{t},x_1) \subset Q_{3/4}^\Omega$. Note that by using \eqref{eq:kappa-bound-u}, we get
\begin{equation*}
    \| \overline{u} \| _{L^\infty (Q_{r/2}(\tilde{t},x_1))} \leq c r ^\kappa .
\end{equation*}
and 
\begin{align*}
     \sup _{t \in I^\ominus _{r/2} (\tilde{t})} \Tail (\overline{u};r/4,x_1) & \leq  c r^{2s} \int _{B^\Omega_{3/4}\setminus B_{r/4}(x_1)}  \frac{|x_1 -y| ^\kappa}{|x_1 -y|^{d+2s}}\intd y + c r^{2s} \int _{B_2\setminus B_{3/4}} \frac{1}{|x_1-y|^{d+2s}} \intd y\\
     & \leq c r^\kappa + c r^{2s} \leq c r ^\kappa .
\end{align*}
Hence, we can apply interior regularity estimates (see \autoref{Prop:InteriorRegularity}), and obtain
\begin{equation} \label{eq:interior-estimates-on-I1}
    I_1 \leq [\overline{u}]_{C^\kappa (Q_{r/4}(\tilde{t},x_1))} \leq c r^{-\kappa} \left( \| \overline{u} \| _{L^\infty (Q_{r/2}(\tilde{t},x_1))}  + \sup _{t \in I^\ominus _{r/2} (\tilde{t})} \Tail (\overline{u};r/4,x_1) + r ^{2s} \| \overline{f}\| _{L^\infty (Q^\Omega_1)}  \right) \leq c.
\end{equation}
Combining Case 1 and 2 yields $I_1\leq c$.

For $I_2$, first note that we can assume that $|t_1-t_2|$ is small.
Next, notice that there exists some $\eps_0 >0$ depending only on $d$ and the Lipschitz constant $\delta$ of $\Omega$, such that whenever $|t_1 -t_2| \leq \xi_0$ for some $\xi_0=\xi _0 (d,\delta ,s )>0$, there exists some $\tilde{x}\in B^\Omega _{3/5}$ such that 
\begin{equation} \label{eq:geometric-conditions-tilde-x}
    \eps _0 |x_2-\tilde{x}| \leq |t_1 -t_2 | ^\frac{1}{2s} \leq  \frac{d_\Omega (\tilde{x})}{2} .
\end{equation}


Then, we have
\begin{equation*}
    |u(t_1,x_2) - u(t_2,x_2)| \leq 2 \sup _{t\in I^\ominus _{1/2}} | u(t,x_2)-u(t,\tilde{x})| + |u (t_1,\tilde{x})-u(t_2,\tilde{x})| =:J_1 +J_2 .
\end{equation*}
For $J_2$, note that by \eqref{eq:geometric-conditions-tilde-x}, we can choose some $\tilde{t} \in I^\ominus _{1/2}$ such that
\begin{equation*}
    (t_1,\tilde{x}) , (t_2,\tilde{x}) \in Q_{d_\Omega(\tilde{x})/2}(\tilde{t}, \tilde{x}) \subset Q_{d_\Omega(\tilde{x})}(\tilde{t}, \tilde{x}) \subset Q^\Omega _{3/4} .
\end{equation*}
Applying interior regularity estimates as in \eqref{eq:interior-estimates-on-I1}, yields
\begin{equation}\label{eq:estimate-J2-t-reg}
    J_2 \leq c |t_1 -t_2| ^{\frac{\kappa}{2s}}.
\end{equation}
For $J_1$, we conclude by the estimates proven for $I_1$ and \eqref{eq:geometric-conditions-tilde-x},
\begin{equation}\label{eq:estimate-J1-t-reg}
    J_1 \leq c |x_2-\tilde{x}| ^{\kappa} \leq c |t_1-t_2| ^{\frac{\kappa}{2s}} .
\end{equation}
Combining \eqref{eq:estimate-J2-t-reg} and \eqref{eq:estimate-J1-t-reg} shows $I_2 \leq c$ which finishes the proof.
\end{proof}

\begin{Theorem} \label{Thm:BoundaryRegularityTranslationInvar}
    Let $L$ be a translation invariant operator of the form \eqref{eq:TranslationalInvariantOperator}-\eqref{eq:TranslationalInvariantKernel} and let $\Omega \subset \R^d$ be a domain with $0\in \partial \Omega$.
    Given $f\in L^\infty ((-1,0)\times (\Omega \cap B_1))$, assume that $u$ is a bounded weak solution to
    \begin{equation*}
        \left\{ \begin{aligned}
          \partial _t u -Lu & =f & &\text{in} \quad (-1 ,0) \times (\Omega \cap B_1) , \\
           u & =0 & & \text{in} \quad  (-1 ,0) \times (B_1 \setminus \Omega) . 
        \end{aligned} \right.
    \end{equation*}
        \begin{itemize}
        \item[(i)] For every $\eps \in (0,s)$ there exists some $\delta _0 =\delta_0 (d,s,\lambda ,\Lambda ,\eps )>0$ such that if $\Omega$ is a Lipschitz domain with Lipschitz constant $\delta \geq 0$ which satisfies $\delta\leq \delta_0$, then
    \begin{equation*}
            \|  u \| _{C^{s-\eps }_{p} (\overline{Q^\Omega_{1/2}})} \leq c \left( \| u \| _{L^\infty ((-1,0)\times \R ^d)} + \| f\| _{L^\infty (Q^\Omega_1)}\right)
       \end{equation*}
       for some constant $c=c(d,s, \lambda ,\Lambda , \eps )>0$.
        
        \item[(ii)] If $\Omega$ is a $C^{1,\alpha}$ domain for some $\alpha \in (0,s)$, then
        \begin{equation}\label{eq:CsBoundaryEstimate}
            \|  u \| _{C^{s}_{p} (\overline{Q^\Omega_{1/2}})} \leq c \left( \| u \| _{L^\infty ((-1,0)\times \R ^d)} + \| f\| _{L^\infty (Q^\Omega_1)}\right)
       \end{equation}
    for some constant $c=c(d,s,\Omega ,\alpha , \lambda , \Lambda )>0$.
        \end{itemize}
\end{Theorem}
\begin{proof}
    \autoref{Thm:BoundaryRegularityTranslationInvar} follows from \autoref{Lem:BoundaryRegFromBoundOnU} and known boundary regularity results for elliptic equations.
    In particular, \autoref{Thm:BoundaryRegularityTranslationInvar}(i) follows from \cite[Theorem 6.8]{RosOton2024} together with \autoref{Lem:BoundaryRegFromBoundOnU}, and \autoref{Thm:BoundaryRegularityTranslationInvar}(ii) follows from \cite[Theorem 1.1]{RosOton2024} together with \autoref{Lem:BoundaryRegFromBoundOnU}.
\end{proof}

\subsection{Liouville theorem in half-space}

We have the following Liouville theorem for translation invariant parabolic equations in the half-space.

\begin{Theorem}\label{Thm:LiouvilleThm}
    Let $L$ be of the form \eqref{eq:TranslationalInvariantOperator}-\eqref{eq:TranslationalInvariantKernel} and $\nu \in \Sph ^{d-1}$. Let $u$ be a weak solution of
    \begin{equation*}
      \left\{ \begin{aligned}
        \partial _t u -Lu & =0 & & \text{in} \quad (-\infty ,0) \times \{ x\cdot \nu >0 \} , \\
            u & =0 & & \text{in} \quad (-\infty ,0) \times \{ x\cdot \nu \leq 0 \} .
       \end{aligned} \right.
    \end{equation*}
    Furthermore, assume that $u$ satisfies the growth condition
    \begin{equation*}
        \| u\| _{L^\infty (Q_R)} \leq c R^\gamma \qquad \forall R\geq 1
    \end{equation*}
    for some $\gamma \in (0,2s)$ and $c>0$.
    Then, $u(t,x)=A b(x)$ for some constant $A\in \R$ where $b$ is the barrier from \cite[Theorem 1.4]{RosOton2024} corresponding to $L$ and $\nu$.
\end{Theorem}
\begin{proof}
Following the proof of \cite[Theorem 4.11]{FernandezReal2017} using the translation invariance in the tangential direction of $\{ x\cdot \nu >0\}$ together with \autoref{Thm:BoundaryRegularityTranslationInvar}(ii), we get that $u$ is constant in time. 
By the elliptic Liouville theorem in the half-space from \cite[Theorem 6.2]{RosOton2024}, we conclude that $u(t,x)= A b(x)$ for some constant $A\in \R$.
\end{proof}

\subsection{Barrier functions for translation invariant elliptic problems}

The goal of this subsection is to prove a higher order boundary expansion for solutions to parabolic nonlocal equations with a translation invariant operator $L$ as in \eqref{eq:TranslationalInvariantOperator}-\eqref{eq:TranslationalInvariantKernel}. To do so, we need to recall the barrier function $\psi$ from \cite[Section 6.1]{Kim2024}.

Let $L$ be a translation invariant operator satisfying \eqref{eq:TranslationalInvariantOperator}-\eqref{eq:TranslationalInvariantKernel}. Let $\Omega \subset \R^d$ be a $C^{1,\alpha}$ domain with $0 \in \partial\Omega$ for some $\alpha \in (0,s)$. We take a set $D \subset \R^d$ with $\partial D \in C^{1,\alpha}$ and $B^\Omega_1 \subset D \subset B^\Omega_2$. Then, we define the barrier $\psi$ with respect to $L$ as the solution to the exterior value problem
\begin{equation} \label{eq:EllipticBarrierUBar}
     \left\{ \begin{aligned}
        -L \psi  &=0  & & \text{in} \quad D , \\
        \psi&= g   & & \text{in} \quad D^c,
    \end{aligned}
    \right.
\end{equation}
where we fix some $g \in C^{\infty}_c(\R^d \setminus B_3)$ such that $0 \le g \le 1$ and $g \not\equiv 0$.
In particular, $\psi = 0$ in $B_3 \setminus D$.

We recall the following lemma from \cite{Kim2024}.
\begin{Lemma}
\label{Lem:psi-properties}
    Let $L,\Omega,D,g,\psi$ be as before. Then, there exists $C = C(d,s,\lambda,\Lambda,\alpha,\Omega  ) > 0$ such that the following hold true:
    \begin{itemize}
        \item[(i)] $\psi \in C^s_x(\R^d)$ and
        \begin{align*}
            \Vert \psi \Vert_{C^s_x(\R^d)} \le C.
        \end{align*}
        \item[(ii)] It holds $\psi \le C d_D^s$ in $D$ and
        \begin{align}
        \label{eq:UBarComparableDPowerS}
            C^{-1} d_{\Omega}^s \le \psi \le C d_{\Omega}^s ~~ \text{ in } B^\Omega_{1/2}.
        \end{align}
        \item[(iii)] For any $0 < R \le 1/4$ and $x_0 \in B^\Omega_{1/4}$ it holds
        \begin{align}\label{eq:TailPsiBound}
            \Tail(\psi ;R,x_0) + \Tail_{\sigma,B_1}(\psi ;R,x_0) \le C \max\{R , d_{\Omega}(x_0)\}^s.
        \end{align}
    \end{itemize}
\end{Lemma}

\begin{proof}
    The claims in (i) and (ii) are contained in \cite[Proposition 6.2]{Kim2024}.
    
    For the claim in (iii), note that by (ii), $|\psi (y) | \leq c|y-x_0|^s+cd_\Omega (x_0) ^s$ for all $y\in B_{1/4}(x_0) \subset B_{1/2}$.
    Hence, assuming $d_\Omega (x_0) \leq R$,
    \begin{align*}
        &\Tail_{\sigma ,B_1} (\psi ;R,x_0) \\
        & \qquad = R^{2s- \sigma} \int _{B_1 \setminus B_{1/4}(x_0)} \frac{|\psi (y)|}{|y-x_0|^{d+2s-\sigma}} \intd y
        +R^{2s- \sigma} \int _{ B_{1/4}(x_0) \setminus B_R(x_0)} \frac{|\psi (y)|}{|y-x_0|^{d+2s-\sigma}} \intd y \\
        &\qquad \leq c R^{2s-\sigma} +c R^s +cd_\Omega (x_0) ^s \leq c R^s,
    \end{align*}
    where we have used $\| \psi \|_{L^\infty (B_1)} \leq c$ and $\sigma <s$.
    For the case $R\leq d_\Omega (x_0)$, the estimate $\Tail_{\sigma ,B_1} (\psi ;R,x_0) \leq c d_\Omega (x_0) ^s$ follows exactly as in \cite[(6.20) and on p.33]{Kim2024}.
    This proves the bound on $\Tail_{\sigma ,B_1}$ in \eqref{eq:TailPsiBound}.
    Together with \eqref{eq:TailSplitSigmaTail} and $\| \psi \| _{L^\infty (\R^d) } \leq c$, this also shows the bound on $\Tail (\psi ;R,x_0 )$ in \eqref{eq:TailPsiBound}.
\end{proof}

The following lemma is a slight modification of \cite[Proposition 6.4]{Kim2024}.

\begin{Lemma}\label{Lem:ClosenessOfBarriers}
    Let $\Omega,D,g$ be as before. Let $L_1,L_2$ be two translation invariant operators as in \eqref{eq:TranslationalInvariantOperator}-\eqref{eq:TranslationalInvariantKernel} with kernels $K_1,K_2$ satisfying for some $\gamma > 0$,
    \begin{align}
    \label{eq:K1K2-ass}
        |K_1(h) - K_2(h)| \le \gamma |h|^{-d-2s} ~~ \forall h \in \R^d.
    \end{align}
    Let $\psi_1,\psi_2$ be as in \eqref{eq:EllipticBarrierUBar} with respect to $L_1,L_2$, respectively. Then, it holds for any $\eps \in (0,s)$,
    \begin{align*}
        [\psi_1 - \psi_2]_{C^{s-\eps}_x(\overline{B^\Omega_{1/2}})} \le c \gamma,
    \end{align*}
    where $c > 0$ depends only on $d,s,\lambda,\Lambda,\alpha,\eps$, and $\Omega$, but not on $\gamma$.
\end{Lemma}

\begin{proof}
    The proof is contained in \cite[Proposition 6.4]{Kim2024}, where the role of $\gamma > 0$ is played by $|x_0 - y_0|^{\sigma}$ and $\psi_1,\psi_2$ are denoted by $\psi_{x_0}, \psi_{y_0}$, respectively and belong to two operators $L_{x_0},L_{y_0}$ with kernels $K_{x_0}, K_{y_0}$. Note that the assumption \eqref{eq:K1K2-ass} is used only in \cite[(6.15)]{Kim2024} and that the indicator function $\ind_{B_{1/2}}(x-y)$ can be estimated from above by one without any further change in the proof.
\end{proof}

\subsection{Expansion in the translation invariant case}

Finally, we establish an expansion of order $s + \eps$ at boundary points (see \eqref{eq:ExpansionTranslInvariant}) for solutions to parabolic translation invariant nonlocal equations. Compared to the corresponding elliptic result in \cite[Theorem 6.9]{RosOton2024}, we improve the range of $\eps$ from $\eps\in (0,\alpha s)$ to $\eps \in (0,s)$. This is achieved by employing the barrier $\psi$ from \eqref{eq:EllipticBarrierUBar} instead of the one-dimensional barrier $b$ used in \cite[Theorem 1.4]{RosOton2024}.

\begin{Proposition}\label{Prop:ExpansionTranslationalInvariant}
Let $L$ be of the form \eqref{eq:TranslationalInvariantOperator}-\eqref{eq:TranslationalInvariantKernel} and let $\Omega \subset \R^d$ be a $C^{1,\alpha}$ domain for some $\alpha \in (0,s)$ with $0\in \partial \Omega$.
Moreover, let $u$ be a weak solution of
\begin{equation*}
    \left\{ \begin{aligned}
        \partial _t u -Lu & =f & &\text{in} \quad (-1 ,0) \times (\Omega \cap B_1) , \\
        u & =0 & & \text{in} \quad  (-1 ,0) \times (B_1 \setminus \Omega) .
    \end{aligned} \right.
\end{equation*}
Then, for any $\eps \in (0, s)$ there exists some $q_0 \in \R$ and constants $c_1,c_2 >0$ depending only on $d,s, \lambda ,\Lambda, \Omega ,\alpha, \eps$ such that $|q_0 | \leq c_1 \left( \| u\| _{L^\infty ((-1,0)\times \R^d)} + \| f\| _{L^\infty (Q^\Omega _1)} +1 \right) $ and
\begin{equation}\label{eq:ExpansionTranslInvariant}
    |u(t,x)-q_0 \psi(x)| \leq c_2 \left( |x|^{s+\eps} + |t| ^\frac{s+\eps}{2s} \right)
    \left( \| u\| _{L^\infty ((-1,0)\times \R^d)} + \| f\| _{L^\infty ( Q^\Omega _1)} \right)  
\end{equation}
for all $(t,x) \in Q_1$, where $\psi$ is the solution of \eqref{eq:EllipticBarrierUBar}.
\end{Proposition}
\begin{proof}
We will prove \eqref{eq:ExpansionTranslInvariant} by contradiction (for similar proofs, see, e.g., \cite[Proposition 2.7.3]{RosOton2024Book}, \cite[Proposition 5.1]{FernandezReal2017}).
Assume that there exist sequences $(\Omega _k)_k$, $(L_k)_k$, $(u_k)_k$, $(f_k)_k$, where $\Omega _k$ is a $C^{1,\alpha}$ domain with $C^{1,\alpha}$-norm bounded by one and normal vector $\nu_k \in \Sph ^{d-1}$ at $0\in \partial \Omega _k$, and $L_k$ is an operator of the form \eqref{eq:TranslationalInvariantOperator}-\eqref{eq:TranslationalInvariantKernel}.
Furthermore, assume that
\begin{equation*}
\| u_k\| _{L^\infty ((-1,0)\times \R^d)} + \| f_k\| _{L^\infty ( Q^{\Omega _k}_1)} \leq 1 ,
\end{equation*}
and $u_k$ satisfies
\begin{equation*}
    \left\{ \begin{aligned}
        \partial _t u_k -L_ku_k & =f_k & & \text{in} \quad (-1 ,0) \times (\Omega_k \cap B_1) , \\
        u_k & =0 & & \text{in} \quad (-1 ,0) \times (B_1 \setminus \Omega _k) .
    \end{aligned} \right.
\end{equation*}
Additionally, we assume that 
\begin{equation*}
    \sup _{k\in \N} \sup _{r\in (0,1)} r^{-s-\eps} \left\| u_k-q_{k,r} \psi_k \right\| _{L^\infty (Q_r)} = \infty
\end{equation*}
where $\psi_k$ is the solution of \eqref{eq:EllipticBarrierUBar} corresponding to $L=L_k$ and $\Omega =\Omega _k$, and 
\begin{equation*}
    q_{k,r} := \frac{\int_{I^\ominus_r} \int _{B_r}u_k \psi_k  \intd x \intd t}{r^{2s} \int _{B_r}\psi_k ^2  \intd x }.
\end{equation*}
Note that these assumptions negate \autoref{Prop:ExpansionTranslationalInvariant} (see, e.g., \cite[Lemma 5.2]{FernandezReal2017}).
We set
\begin{equation*}
    \theta (r) := \sup _{k\in \N} \sup _{\rho \in (r,1)} \| u_k-q_{k,\rho} \psi_k \| _{L^\infty ( Q_\rho )} .
\end{equation*}
Note that $\theta (r)$ is finite for every $r>0$ and $\theta (r)\to \infty$ as $r\searrow 0$.
Hence, we can pick sequences $(r_m)_m$, $(k_m)_m$ such that $r_m\searrow 0$ and
\begin{equation*}
    r_m^{-s-\eps} \| u_{k_m} - q_{k_m,r_m}\psi_{k_m} \| _{L^\infty ( Q_{r_m})} \geq \theta (r_m) /2 , \qquad \forall m \in \N .
\end{equation*}
Let us consider the sequence 
\begin{equation*}
    v_m(t,x):= \frac{u_{k_m}(r_m^{2s}t,r_mx) -q_{k_m,r_m} \psi_{k_m} (r_m x)}{r_m^{s+\eps} \theta (r_m)}
\end{equation*}
which by construction satisfies 
\begin{equation}\label{eq:PropertiesVm}
     \iint _{Q_1}  v_m(t,x) \frac{\psi_{k_m}(r_m x)}{r_m^s} \intd x \intd t =0  , \quad
    \| v_m\| _{L^\infty ( Q_1)} \geq \frac{1}{2}
\end{equation}
for all $m\in \N$.
Following the proofs of \cite[(6.18)]{RosOton2024} or \cite[(5.14)]{FernandezReal2017}, using the bounds $\psi_{k_m}\asymp d_{\Omega _{k_m}}^s$ from \eqref{eq:UBarComparableDPowerS}, we get
\begin{equation}\label{eq:GrowthBoundVm}
    \| v_m \| _{L^\infty ( Q_R)} \leq c R^{s + \eps}, \quad \forall R\geq 1, \forall m\in \N .
\end{equation}
Now, let $\tilde{L}_{m}$ be the operator of the form \eqref{eq:TranslationalInvariantOperator}-\eqref{eq:TranslationalInvariantKernel} with kernel $\tilde{K}_{m}(h):= r_m^{d+2s} K_{k_m}(r_mh)$ where $K_{k_m}$ is the kernel of $L_{k_m}$.
Furthermore, we define 
\begin{equation*}
    \Omega _{R,m} := \left\{ x\in B_R \mid r_mx \in \Omega _{k_m} \right\} .
\end{equation*}
Then, the $v_m$ satisfy
\begin{equation}\label{eq:Equation-vm}
    \left\{ \begin{aligned}
   \partial_t v_m -\tilde{L}_{m} v_m  & = \tilde{f}_m & & \text{in} \quad (-R^{2s},0) \times \Omega _{R,m} ,  \\
        v_m & = 0 & & \text{in} \quad (-R^{2s} ,0) \times (B_R \setminus \Omega _{R,m}) 
    \end{aligned} \right.
\end{equation}
for all $0\leq R \leq r_m^{-1}$, where 
\begin{equation*}
    \begin{aligned}
    \tilde{f}_m(t,x) &:= \frac{r_m^{s-\eps }}{\theta (r_m)} f_{k_m}(r_m^{2s} t, r_m x).  
    \end{aligned}
\end{equation*}
Since $\eps <s$, we have
\begin{equation}\label{eq:FmConvergesToZero}
    \tilde{f}_m \to 0 \qquad \text{locally uniformly as $m\to \infty$}.
\end{equation}

Moreover, from \autoref{Thm:BoundaryRegularityTranslationInvar}(ii) (use a standard truncation argument to replace the global in space $L^\infty$-norm of $u$ in \eqref{eq:CsBoundaryEstimate} by a weighted $L^\infty _{s+\eps}$-norm), we obtain
\begin{equation*}
    \| v_m\| _{C^{s}_{p} (Q_R)} \leq c(R)
\end{equation*}
for some constant $c(R)$ independent of $m$.
As a consequence, using the Arzelà-Ascoli theorem, up to a subsequence, $v_m$ converges locally uniformly to some $v_\infty \in C( (-\infty ,0 ) \times \R^d)$.
By applying a stability result (see \cite[Lemma 3.1]{FernandezReal2017} and \cite[Proposition 2.2.36]{RosOton2024Book}) and using \eqref{eq:GrowthBoundVm}, \eqref{eq:Equation-vm}, and \eqref{eq:FmConvergesToZero}, there exists an operator $L_\infty$ of the form \eqref{eq:TranslationalInvariantOperator}-\eqref{eq:TranslationalInvariantKernel} such that
\begin{equation*}
    \left\{ \begin{aligned}
        \partial _t v_\infty - L_\infty v_\infty &= 0 & &\text{in} \quad (-\infty,0) \times \{ x\cdot \nu >0\} , \\
        v_\infty &= 0 & & \text{in} \quad (-\infty,0) \times \{ x\cdot \nu \leq 0 \} , \\
         \| v_\infty \| _{L^\infty ( Q_R)} & \leq c R^{s+\eps} & & \text{for all} \quad  R\geq 1
    \end{aligned} \right.
\end{equation*}
where $\nu \in \Sph ^{d-1}$ is the limit of the normal vectors $\nu _{k_m}$ of $\Omega _{k_m}$ at $0$ (after passing to a subsequence).
Hence, the Liouville theorem (see \autoref{Thm:LiouvilleThm}) implies that $v_\infty (t,x) =A b(x)$ for some constant $A\in \R$ where $b$ is the barrier from \cite[Theorem 1.4]{RosOton2024} corresponding to $L_\infty$ and $\nu$.

Note that $\phi_m(x):=\psi_{k_m}(r_m x) / r_m^s$ satisfies
\begin{equation*}
    \left\{ \begin{aligned}
        -\tilde{L}_m \phi _m &= 0 & & \qquad \text{ in } \Omega _{R,m},\\
    \phi_m &=  0  & & \qquad \text{ in } B_R \setminus \Omega _{R,m}
    \end{aligned} \right.
\end{equation*}
for all $0\leq R \leq r_m^{-1}$.
Hence, by \autoref{Thm:BoundaryRegularityTranslationInvar}(ii), $\phi_m$ is uniformly bounded in $C^s_\text{loc} (\R^d )$.
As above, $\phi_m$ converges locally uniformly to some $\phi \in C(\R ^d)$ which solves
\begin{equation*}
    \left\{ \begin{aligned}
        - L_\infty \phi &=0 & & \text{in} \quad \{ x\cdot \nu >0\}, \\
        \phi &=0 & & \text{in} \quad \{ x\cdot \nu \leq 0\}, \\
        \| \phi \| _{L^\infty (  B_R)} & \leq c R^{s} & & \text{for all} \quad  R\geq 1.
    \end{aligned} \right.
\end{equation*}
By the Liouville theorem in the half-space (see \cite[Theorem 6.2]{RosOton2024}), we get that $\phi (x) =B b(x)$ for some constant $B$ which by \eqref{eq:UBarComparableDPowerS} is positive $B>0$.
Taking the limit $m\to \infty$ in the first expression of \eqref{eq:PropertiesVm}, we obtain
\begin{equation*}
    0= \lim _{m\to \infty} \int _{-1}^0\int _{B_1} v_m (t,x) \phi _m(x)  \intd x \intd t = \int _{-1}^0\int _{B_1} A b(x) Bb(x)  \intd x \intd t
\end{equation*}
which implies $A=0$.
However, this contradicts the second expression of \eqref{eq:PropertiesVm}.
\end{proof}

\section{Freezing estimates}\label{sec:FreezingEstimates}
Given an operator $\mathcal{L}_t$ of the form \eqref{eq:OperatorDivergenceForm}, \eqref{eq:KernelDivergenceForm}, satisfying \eqref{eq:KernelHoelderCont} with $\mathcal{A}=Q_1$, we assume that $u$ is a weak solution of
\begin{equation}\label{eq:EquationUForFreezingEstimates}
       \left\{ \begin{aligned}
          \partial _t u -\mathcal{L}_tu & =f & &\text{in} \quad (-1 ,0) \times (\Omega \cap B_1) , \\
           u & =0 & & \text{in} \quad  (-1 ,0) \times (B_1 \setminus \Omega) .
        \end{aligned} \right.
\end{equation}
For some $z_0=(t_0,x_0) \in Q_{1/2}$, we fix the translation invariant kernel
\begin{equation} \label{eq:FrozenKernel}
    K_{z_0} (x,y) = \frac{1}{2} \left( K(t_0,x_0+x-y,x_0) + K(t_0,x_0 +y-x,x_0) \right)
\end{equation}
which defines an operator $L_{z_0}$.
Fix some $0<R<1/16$, so that $Q_{8R}(z_0) \subset Q_1$ for all $z_0\in Q_{1/2}$. 
Let $v_{R,z_0}$ be the solution to (which exists by \cite[Corollary 5.5]{Felsinger2014})
\begin{equation} \label{eq:DefinitionV}
        \left\{\begin{aligned}
          \partial _t v_{R,z_0} -L_{z_0}v_{R,z_0} & =0 & &\text{in} \quad Q_R^\Omega (z_0)  , \\
           v_{R,z_0} & =u & & \text{in} \quad  ( (-1 ,t_0 ) \times  \R^d ) \setminus Q^\Omega_R(z_0) .
        \end{aligned} \right.
\end{equation}
Set $w_{R,z_0}:=u-v_{R,z_0}$, then $w_{R,z_0}$ satisfies
\begin{equation}\label{eq:DefinitionW}
     \left\{   \begin{aligned}
          \partial _t w_{R,z_0} -L_{z_0}w_{R,z_0} & =  (\mathcal{L}_t-L_{z_0}) u +f & &\text{in} \quad Q_R^\Omega (z_0) , \\
           w_{R,z_0} & =0 & & \text{in} \quad  ( (-1 ,t_0 ) \times  \R^d ) \setminus Q_R^\Omega (z_0) .
        \end{aligned} \right.
\end{equation}

The following three lemmas will allow us to obtain regularity results for $u$ by comparison with its caloric replacement $v_{R,z_0}$.

\begin{Lemma} \label{Lem:EstimateOnHsw}
Let $\Omega \subset\R^d$ be a domain and let $\mathcal{L}_t$ be an operator of the form \eqref{eq:OperatorDivergenceForm}, \eqref{eq:KernelDivergenceForm}, satisfying \eqref{eq:KernelHoelderCont} with $\mathcal{A}=Q_1$. 
Given $f\in L^q_tL^r_x (Q_1^\Omega)$ where $q,r\geq 1$ satisfy $\frac{1}{q}+\frac{d}{2sr} <1$, and given some $z_0\in Q^\Omega _{1/2}$ and $R\in (0,1/16)$, assume that $u,v_{R,z_0},w_{R,z_0}$ are solutions to \eqref{eq:EquationUForFreezingEstimates}, \eqref{eq:DefinitionV}, \eqref{eq:DefinitionW}. Then, we have
    \begin{align*}
        \sup _{t\in I^\ominus _R(t_0)} & \| w_{R,z_0}\| ^2_{L^2_x (B_R(x_0))}  +  \| w_{R,z_0}\| ^2_{L^2_t \dot{H}^s_x(I^\ominus_R(t_0) \times \R^d )} \\
        &\leq c  R^{2\sigma }\Bigg(  \| u\| ^2_{L^2_t \dot{H}^s_x(I^\ominus_R(t_0) \times B_{2R}(x_0) )}   
        + R^{ d-2s}   \int _{I^\ominus _R(t_0)} \left( \Tail _{\sigma ,B_1} (u -(u)_{B_R(x_0)} ; R ,x_0) \right) ^2 \intd t \Bigg)\\
        & \quad +c  R^{d+4s}    \left[ \| u\|^2 _{L^\infty ((-1,0) \times \R^d)} +  R^{-\frac{2d}{r}-\frac{4s}{q}}\| f\| ^2 _{L^q_t L^r_x (Q_R^\Omega (z_0))} \right]  
    \end{align*}
    for some $c=c(d,s, \lambda , \Lambda ,\sigma ,q,r)>0$.
\end{Lemma}
\begin{proof}
    Set $w:=w_{R,z_0}$ and let $t\in I^\ominus _R(t_0)$.
    We want to test the equation \eqref{eq:DefinitionW} with $\varphi =w \eta_k$, where $\eta_k$ approximates $\ind _{[t_0-R^{2s},t]}$ and $\partial_t \eta_k$ approximates $\delta_{t} - \delta_{t_0-R^{2s}}$. Testing the equation yields (after a proper time-mollification of $w$, see \cite[Appendix B]{Liao2023})
    \begin{align*}
        \frac{1}{2} \int_{I_R^{\ominus}(t_0)} \partial_t \eta_k(\tau) \int_{B_R(x_0)} (w(\tau,x))^2 &\d x \d \tau + \int_{I_R^{\ominus}(t_0)} \eta_k(\tau) \cE^{K_{z_0}}(w,w) \d \tau \\
        &= \int_{I_R^{\ominus}(t_0)} \eta_k(\tau) \cE^{K_{z_0} - K}(u,w) \d \tau + \iint_{Q_R^{\ominus}(z_0)} \eta_k f w \d x \d \tau.
    \end{align*}
    Hence, using that $w(t_0-R^{2s},\cdot) = 0$ in $B_R(x_0)$, we obtain by taking $k \to \infty$,
    \begin{align*}
        \frac{1}{2}\int_{B_R(x_0)} (w(t,x))^2 \d x &+ \int_{I_R^{\ominus}(t_0)}  \cE^{K_{z_0}}(w,w) \d \tau \\
        &= \int_{I_R^{\ominus}(t_0)} \cE^{K_{z_0} - K}(u,w) \d \tau + \int_{I_R^{\ominus}(t_0)}\int_{B_R(x_0)} f w \d x \d \tau.
    \end{align*}
    Since $t \in I_R^{\ominus}(t_0)$ was arbitrary, we deduce
    \begin{align*}
        \sup _{t\in I^\ominus _R(t_0)} \int _{B_R(x_0) \cap \Omega} &\left( w(t,x) \right) ^2 \intd x + \int _{I^\ominus_R(t_0)} [w(t,\cdot )]^2_{H^s(\R^d )} \intd t \\
        & \leq c  \int _{I^\ominus_R(t_0)} |\mathcal{E}^{K_{z_0} - K} (u,w)| \intd t + c\iint _{Q_R(z_0)} |fw| \intd x \intd t \\
        & =: cI .
    \end{align*}
    We decompose $I$ into five parts
    \begin{align*}
        I & \leq \int _{I^\ominus_R(t_0)}   \int _{B_{2R}(x_0)} \int _{B_{2R}(x_0)}  |u(x)-u(y)||w(x)-w(y)| |K_{z_0} (x,y) - K(t,x,y)| \intd y \intd x  \intd t \\
    & \quad + 2 \int _{I^\ominus_R(t_0)}  \int _{B_{R}(x_0)} \int _{B_1\setminus B_{2R}(x_0)}  |u(x)- (u)_{B_R(x_0)}||w(x)| |K_{z_0} (x,y) - K(t,x,y)| \intd y \intd x  \intd t \\
    & \quad + 2 \int _{I^\ominus_R(t_0)}  \int _{B_{R}(x_0)} \int _{B_1\setminus B_{2R}(x_0)}  |(u)_{B_R(x_0)}- u(y)||w(x)| |K_{z_0} (x,y) - K(t,x,y)| \intd y \intd x  \intd t \\
    & \quad +  2 \int _{I^\ominus_R(t_0)}  \int _{B_{R}(x_0)} \int _{\R^d \setminus B_{1}}  |u(x)- u(y)||w(x)| |K_{z_0} (x,y) - K(t,x,y)| \intd y \intd x  \intd t\\
    & \quad + \iint _{Q_R(z_0)} |f||w| \intd x \intd t \\
    & = I_1+I_2+I_3+I_4+I_5 .
    \end{align*}
    Note that by construction of $K_{z_0}$ together with \eqref{eq:KernelDivergenceForm} and \eqref{eq:KernelHoelderCont}, we have
    \begin{equation} \label{eq:DifferenceFrozenKernelToKernel}
      |K_{z_0} (x,y) - K(t,x,y)| \leq \Lambda \frac{\min \{ \frac{|x-x_0|^\sigma + |y-x_0|^\sigma}{2} +|t-t_0|^\frac{\sigma}{2s} , 2 \}}{|x-y|^{d+2s}} .
    \end{equation}
    This implies
    \begin{equation} \label{eq:DifferenceFrozenKernelXYApart}
        |K_{z_0} (x,y) - K(t,x,y)| \leq c |y-x_0| ^{-d-2s+\sigma } \quad \forall x\in B_R(x_0), y\in B_1\setminus B_{2R}(x_0) , t\in I^\ominus _R(t_0)
    \end{equation}
    since $|x-x_0| \leq |y-x_0|/2$ and $ |t-t_0|^{\frac{1}{2s}}\leq |y-x_0|/2 $ as well as $|y-x_0| \leq  2|x-y|$.
    
    For $I_1$, we use \eqref{eq:DifferenceFrozenKernelToKernel} and Young's inequality, to get
    \begin{equation*}
      I_1 \leq c \left( c(\eps ) R^{2\sigma } \| u\| ^2_{L^2_t \dot{H}^s_x(I^\ominus_R(t_0) \times B_{2R} (x_0))} + \eps   \| w\| ^2_{L^2_t \dot{H}^s_x(I^\ominus_R(t_0) \times \R^d )} \right)
    \end{equation*}
    where we choose $\eps>0$ so small that we can absorb the second summand on the left-hand side.
    For $I_2$, we use \eqref{eq:DifferenceFrozenKernelXYApart}, the {\Poincare} inequality
    \begin{equation} \label{eq:EstimateL1toL2toHs}
        \| w(t, \cdot )\| _{L^1(B_R(x_0))} \leq c R^{\frac{d}{2}} \| w(t,\cdot )\| _{L^2(B_R(x_0))} \leq c R^{\frac{d}{2}+s}[w(t,\cdot )] _{H^s(\R^d)}   ,
    \end{equation}
    and the \Poincare -Wirtinger inequality
    \begin{equation}\label{eq:PoincareWirtingerInequality}
         R^{-2s} \int _{B_R(x_0)} | u(t,\cdot ) -(u(t,\cdot ))_{B_R(x_0)} | ^2 \intd x \leq c [u(t,\cdot )]^2_{H^s_x(B_R(x_0))},
    \end{equation}
    to obtain
    \begin{align*}
        I_2 & \leq c   \int _{I^\ominus_R(t_0)}  \int _{B_{R}(x_0)} \int _{B_1\setminus B_{2R}(x_0)}  |u(x)- (u)_{B_R(x_0)}||w(x)| |y-x_0| ^{-d-2s+\sigma } \intd y \intd x  \intd t \\
        & \leq c  \int _{I^\ominus_R(t_0)} \int _{B_{R}(x_0)}R^{\sigma -2s}|u(x)- (u)_{B_R(x_0)}||w(x)| \intd x \intd t  \\
        & \leq c  \Bigg( c(\eps )  \int _{I^\ominus_R(t_0)} \int _{B_{R}(x_0)}R^{2\sigma -2s}|u(x)- (u)_{B_R(x_0)}| ^2 \intd x \intd t  \\
        & \qquad\qquad \qquad + \eps  \int _{I^\ominus_R(t_0)} \int _{B_{R}(x_0)}R^{-2s}|w(t,x)| ^2 \intd x \intd t \Bigg) \\
        & \leq c \left( c(\eps ) R^{2\sigma}   \| u\| ^2_{L^2_t \dot{H}^s_x(I^\ominus_R(t_0) \times B_{2R}(x_0) )}  +\eps   \| w\| ^2_{L^2_t \dot{H}^s_x(I^\ominus_R(t_0) \times \R^d )} \right) .
    \end{align*}
    For $I_3$, we use \eqref{eq:DifferenceFrozenKernelXYApart} and \eqref{eq:EstimateL1toL2toHs} to get
    \begin{align*}
        I_3 & \leq c \int _{I^\ominus_R(t_0)}   \left( \int _{B_1\setminus B_{2R}(x_0)}  \frac{|u(y)-(u)_{B_R(x_0)}|}{|y-x_0|^{d+2s-\sigma}} \intd y \right)  \left( \int _{B_{R}(x_0)}|w(t,x)|   \intd x \right) \intd t \\
        & \leq c \int _{I^\ominus_R(t_0)}  R^{\frac{d}{2}-s+\sigma} \Tail _{\sigma ,B_1} (u-(u)_{B_R(x_0)};R,x_0) [w] _{H^s_x(\R^d)}  \intd t \\
        & \leq c \left( c(\eps ) R^{d-2s+2\sigma}   \int _{I^\ominus _R(t_0)} \left( \Tail _{\sigma ,B_1} (u -(u)_{B_R(x_0)} ; R ,x_0) \right) ^2 \intd t +\eps   \| w\| ^2_{L^2_t \dot{H}^s_x(I^\ominus_R(t_0) \times \R^d )} \right) .
    \end{align*}
    For $I_4$, we use \eqref{eq:DifferenceFrozenKernelToKernel}, $|x-y|\geq |y|/2$, and \eqref{eq:EstimateL1toL2toHs}, to obtain
    \begin{align*}
        I_4 & \leq c \int _{I^\ominus_R(t_0)}  \int _{B_{R}(x_0)} \int _{\R^d \setminus B_{1}}  |u(x)- u(y)||w(x)| |y| ^{-d-2s} \intd y \intd x  \intd t \\
        & \leq c \int _{I^\ominus_R(t_0)} \| u\| _{L^\infty ((-1,0) \times \R^d)} \| w\| _{L^1_x (B_R(x_0))} \intd t \\
        & \leq c \int _{I^\ominus_R(t_0)} \| u\| _{L^\infty ((-1,0) \times \R^d)} R^{\frac{d}{2}+s}[w] _{H^s_x(\R^d)} \intd t \\
        & \leq c \left( c(\eps ) R^{d+4s} \| u\| ^2 _{L^\infty ((-1,0) \times \R^d)} + \eps   \| w\| ^2_{L^2_t \dot{H}^s_x(I^\ominus_R(t_0) \times \R^d )} \right) .
    \end{align*}

   Finally, let us consider $I_5$. By $q'$, $r'$ we denote the Hölder conjugate exponents of $q$ and $r$.
   Then, by the assumption $\frac{1}{q} + \frac{d}{2sr} < 1$, we can find $\bar{r} \in [ \max \{ 2 ,r' \} , \frac{2d}{d-2s}]$ and $\bar{q} \in [ \max \{ 2,q' \},\infty)$ such that
    \begin{align*}
        \frac{1}{\bar{r}} + \frac{2s}{d \bar{q}} = \frac{1}{2}, \qquad \text{ and thus } \qquad  d\left(\frac{1}{r'} - \frac{1}{\bar{r}} \right) + 2s \left( \frac{1}{q'} - \frac{1}{\bar{q}} \right) = \frac{d}{2} + 2s - \left(\frac{d}{r} + \frac{2s}{q} \right).
    \end{align*}
    For example, let $\bar{q}= \max \{ 2, q'\}$ and $\frac{1}{\bar{r}} =\frac{1}{2}-\frac{2s}{d \max \{ 2,q'\} }$.
    Then, using that by Lebesgue interpolation and the Sobolev embedding (see \cite[Lemma 2.2]{Byun2023a}) it holds
   \begin{align*}
       \Vert w \Vert_{L^{\bar{q}}_t L^{\bar{r}}_x(Q_R^{\Omega}(z_0))}^2 \le C \| w\| ^2_{L^\infty_t L^2_x(Q_R^{\Omega}(z_0))}  + C\| w\| ^2_{L^2_t \dot{H}^s_x(I_R^{\ominus}(t_0) \times B_R )} ,
   \end{align*} 
    we obtain
    \begin{align*}
        R^{-d-4s+\frac{2d}{r}+\frac{4s}{q}} \| w \| ^2 _{L^{q'}_t L^{r'}_x (Q^\Omega _R(z_0))} &\le R^{-d-4s+\frac{2d}{r}+\frac{4s}{q}} \| 1 \| ^2 _{L^{(\frac{1}{q'} - \frac{1}{\bar{q}})^{-1}}_t L^{(\frac{1}{r'} - \frac{1}{\bar{r}})^{-1}}_x (Q^\Omega _R(z_0))} \| w \| ^2 _{L^{\bar{q}}_t L^{\bar{r}}_x (Q^\Omega _R(z_0))} \\
        &\le C \| w\| ^2_{L^\infty_t L^2_x(Q^\ominus_R(z_0) )}  + C\| w\| ^2_{L^2_t \dot{H}^s_x(I^\ominus_R(t_0) \times \R^d )}.
    \end{align*}
    Hence, altogether, we deduce that    
    \begin{align*}
        I_5 &\le \| f\|  _{L^q_t L^r_x (Q^\Omega _R(z_0))} \| w \|  _{L^{q'}_t L^{r'}_x (Q^\Omega _R(z_0))}  \\
        &\le c(\eps) R^{d+4s-\frac{2d}{r}-\frac{4s}{q}}\| f\| ^2 _{L^q_t L^r_x (Q^\Omega _R(z_0))}  + \eps \big(\| w\| ^2_{L^\infty_t L^2_x(Q^\ominus_R(z_0) )}  + \| w\| ^2_{L^2_t \dot{H}^s_x(I^\ominus_R(t_0) \times \R^d )}   \big).
    \end{align*}

    Combining these estimates and absorbing $\| w\| ^2_{L^\infty_t L^2_x(I^\ominus_R(t_0) \times \R^d )} + \| w\| ^2_{L^2_t \dot{H}^s_x(I^\ominus_R(t_0) \times \R^d )}$ concludes the proof of this lemma.
\end{proof}

\begin{Lemma}\label{Lem:EstimateOnHswWithuL1}
 Under the same assumptions as in \autoref{Lem:EstimateOnHsw}, we have
    \begin{align*}
       \sup _{t\in I^\ominus _R(t_0)}  \| w_{R,z_0}\| _{L^2_x (B_R(x_0))}+& \| w_{R,z_0}\| _{L^2_t \dot{H}_x^s ( I_R^{\ominus}(t_0) \times \R^d )}  \\
       &\leq c \Bigg( R^{\sigma-2s-\frac{d}{2}} \| u\| _{L^1 (Q_{8R} (z_0))} + R^{\sigma +\frac{d}{2}} \sup _{t\in {I^\ominus _{8R}(t_0)}} \Tail _{\sigma ,B_1}(u;4R,x_0) \\
        & \qquad + R^{\frac{d}{2}+2s}    \left[ \| u\| _{L^\infty ((-1,0) \times \R^d)} +  R^{-\frac{d}{r}-\frac{2s}{q}}\| f\| _{L^q_t L^r_x (Q^\Omega _{8R}(z_0))} \right]  \Bigg)
    \end{align*}
    for some $c=c(d,s, \lambda , \Lambda ,\sigma ,q,r )>0$.
\end{Lemma}
\begin{proof}
    Set $w:=w_{R,z_0}$.
    We start with the estimate from \autoref{Lem:EstimateOnHsw}. Using \autoref{Prop:CaccioppoliEstimate} and \eqref{eq:TailSplitSigmaTail},
    we get
    \begin{align*}
        \sup _{t\in I^\ominus _R(t_0)} & \| w\| _{L^2_x (B_R(x_0))}+ \| w\| _{L^2_t \dot{H}_x^s ( I_R^{\ominus}(t_0) \times \R^d )}  \\
        & \leq cR^{\sigma }\Bigg(  \| u\| _{L^2_t \dot{H}^s_x(Q_{2R}(z_0))}   
        + R^{ \frac{d}{2}-s}  \left( \int _{I^\ominus _R(t_0)} \left( \Tail _{\sigma ,B_1} (u -(u)_{B_R(x_0)} ; R ,x_0) \right) ^2 \intd t \right) ^\frac{1}{2} \Bigg)\\
        & \quad +c  R^{\frac{d}{2}+2s}    \left[ \| u\| _{L^\infty ((-1,0) \times \R^d)} +  R^{-\frac{d}{r}-\frac{2s}{q}}\| f\|  _{L^q_t L^r_x (Q^\Omega_{R}(z_0))} \right] \\
         & \leq c  R^{\sigma-s} \| u\| _{L^2 (Q_{4R}(z_0))} + R^{\sigma +\frac{d}{2}} \sup _{t\in I^\ominus _{4R}(t_0)} \Tail _{\sigma ,B_1}(u;2R,x_0) \\
        & \qquad +c R^{\sigma +\frac{d}{2}-s}  \left( \int _{I^\ominus _R(t_0)} \left( \Tail _{\sigma ,B_1} (u -(u)_{B_R(x_0)} ; R ,x_0) \right) ^2 \intd t\right) ^\frac{1}{2} \\
        & \qquad + c R^{\frac{d}{2}+2s}    \left[ \| u\| _{L^\infty ((-1,0) \times \R^d)} +  R^{-\frac{d}{r}-\frac{2s}{q}}\| f\| _{L^q_t L^r_x (Q^\Omega _{4R}(z_0))} \right]   \\
        & \leq c \Bigg( R^{\sigma-s} \| u\| _{L^2 (Q_{4R}(z_0))} + R^{\sigma +\frac{d}{2}} \sup _{t\in I^\ominus _{4R}(t_0)} \Tail _{\sigma ,B_1}(u;R,x_0) \\
        & \qquad + R^{\frac{d}{2}+2s}    \left[ \| u\| _{L^\infty ((-1,0) \times \R^d)} +  R^{-\frac{d}{r}-\frac{2s}{q}}\| f\| _{L^q_t L^r_x (Q^\Omega _{4R}(z_0))} \right]  \Bigg) .
    \end{align*}
    Note that 
    \begin{align*}
        \sup _{t\in I^\ominus _{4R}(t_0)} \Tail _{\sigma ,B_1}(u;R,x_0) & \leq c \sup _{t\in I^\ominus _{4R}(t_0)} \Tail _{\sigma ,B_1}(u;4R,x_0) + c \| u\| _{L^\infty (Q_{4R}(z_0))} .
    \end{align*}
    With \autoref{Prop:LocalBoundedness}, we obtain
    \begin{align*}
        \| u\| & _{L^2 (Q_{4R}(z_0))}  \leq c  R^{\frac{d}{2}+s}\| u\|_{L^\infty (Q_{4R}(z_0))} \\
        & \leq c R^{\frac{d}{2}+s}\Bigg( R^{-d-2s}\| u \|  _{L^1(Q_{8R} (z_0))}  +  \sup _{t\in I^\ominus_{8R}(t_0)}\Tail  (u;4R,x_0)  + R^{2s -\frac{d}{r}-\frac{2s}{q}}\| f\| _{L^{q}_t L^{r}_{x}(Q^\Omega _{8R}(z_0))} \Bigg) .
    \end{align*}
    Combining the estimates above and using \eqref{eq:TailSplitSigmaTail}, proves the lemma.
\end{proof}

\begin{Lemma} \label{Lem:FreezingEstimateIntW/ds}
Assume that we are in the setting of \autoref{Lem:EstimateOnHsw}.
In addition, assume that $\Omega \subset \R^d$ is a Lipschitz domain with Lipschitz constant bounded by $\delta> 0$.
Then, we have
\begin{equation}\label{eq:FreezingEstimateIntW/ds}
    \begin{aligned}
        &\left\|\frac{w_{R,z_0}}{d_\Omega ^s} \right\|  _{L^1 (Q^\Omega_R(z_0))}  + \max \{  R,d_\Omega (x_0) \} ^{-s} R^{2s+\frac{d}{2}} \| w_{R,z_0} \| _{L^\infty_t L^2_x (Q_R(z_0))} \\
        & \qquad \leq c \Bigg(  R ^{\sigma} \left[ \left\|\frac{u}{d_\Omega ^s} \right\|_{L^1 (Q^\Omega_{8R}(z_0))}  + R^{d+2s} \max \{ 8R, d_\Omega (x_0) \} ^{-s} \sup _{t\in {I^\ominus _{8R}(t_0)}} \Tail _{\sigma , B_1}(u;4R,x_0) \right] \\
        & \qquad \qquad + R^{d+3s}    \left[ \| u\| _{L^\infty ((-1,0) \times \R^d)} +  R^{-\frac{d}{r}-\frac{2s}{q}}\| f\| _{L^q_t L^r_x (Q^\Omega _{8R}(z_0))} \right]  \Bigg) 
    \end{aligned}
\end{equation}
for some $c=c(d,s, \lambda , \Lambda ,\sigma ,q,r, \delta )>0$.
\end{Lemma}
\begin{proof}
Set $w:=w_{R,z_0}$.
We start by estimating the first summand on the left-hand side of \eqref{eq:FreezingEstimateIntW/ds}.
Assume that $B_{16R}(x_0)\cap \Omega ^c \neq \emptyset$.
Then, using Hardy's inequality (see \cite[Lemma 2.4]{Kim2024}) and \autoref{Lem:EstimateOnHswWithuL1}, we get
\begin{align*}
   \| w/d_\Omega ^s \|_{L^1 (Q^\Omega _R(z_0))} & \leq
        c R^{\frac{d}{2}+s} \| w/d_{\Omega\cap B_R(x_0)} ^s \|_{L^2 (Q^\Omega _R(z_0))} 
        \leq c R^{\frac{d}{2}+s}   \| w\| _{L^2_t \dot{H}_x^s ( I_R^{\ominus}(t_0) \times \R^d )} \\
          &\leq c R^{\frac{d}{2}+s}\Bigg( R^{\sigma-2s-\frac{d}{2}} \| u\| _{L^1 (Q_{8R} (z_0))} + R^{\sigma +\frac{d}{2}} \sup _{t\in {I^\ominus _{8R}(t_0)}} \Tail _{\sigma ,B_1}(u;4R,x_0) \\
        & \qquad + R^{\frac{d}{2}+2s}    \left[ \| u\| _{L^\infty ((-1,0) \times \R^d)} +  R^{-\frac{d}{r}-\frac{2s}{q}}\| f\| _{L^q_t L^r_x (Q^\Omega _{8R}(z_0))} \right]  \Bigg) .
\end{align*}
Consequently, the first part of \eqref{eq:FreezingEstimateIntW/ds} follows in the case $B_{16R}(x_0)\cap \Omega ^c \neq \emptyset$ since $d_\Omega \leq 24 R$ in $Q_{8R}(z_0)$ and $ \max \{ 8R, d_\Omega (x_0) \} \leq 16 R$.

Now, let us assume that $B_{16R}(x_0) \subset \Omega$.
Then, we have $d_\Omega(x_0)\leq R+d_\Omega (x)$ for all $x\in B_R(x_0)$ which implies $d_\Omega(x)\geq d_\Omega (x_0) -R \geq \frac{15}{16} d_\Omega (x_0)$ for all $x\in B_R(x_0)$.
Together with \eqref{eq:EstimateL1toL2toHs}, and \autoref{Lem:EstimateOnHswWithuL1}, we obtain
\begin{align*}
    \| w/d_\Omega ^s \|&_{L^1 (Q^\Omega_R(z_0))}  \leq 
    c d_\Omega (x_0)^{-s} \| w \| _{L^1 (Q^\Omega _R(z_0))} \\
    & \leq  c d_\Omega (x_0)^{-s} R^{\frac{d}{2}+s}\| w \| _{L^1_t \dot{H}_x^s (I^\ominus _R(t_0)\times \R ^d )} \\
    &\leq  c d_\Omega (x_0)^{-s} R^{\frac{d}{2}+2s}\| w \| _{L^2_t \dot{H}_x^s (I^\ominus _R(t_0)\times \R ^d )} \\
    & \leq c d_\Omega (x_0)^{-s} R^{\frac{d}{2}+2s} \Bigg( R^{\sigma-2s-\frac{d}{2}} \| u\| _{L^1 (Q_{8R} (z_0))} + R^{\sigma +\frac{d}{2}} \sup _{t\in {I^\ominus _{8R}(t_0)}} \Tail _{\sigma ,B_1}(u;4R,x_0) \\
        & \qquad + R^{\frac{d}{2}+2s}    \left[ \| u\| _{L^\infty ((-1,0) \times \R^d)} +  R^{-\frac{d}{r}-\frac{2s}{q}}\| f\| _{L^q_t L^r_x (Q^\Omega _{8R}(z_0))} \right]  \Bigg) .
\end{align*}
Hence, the first part of \eqref{eq:FreezingEstimateIntW/ds} follows also in the case $B_{16R}(x_0) \subset \Omega$ using $\frac{3}{2}d_\Omega (x_0) \geq  d_\Omega(x_0) +8R\geq d_\Omega (x)$ for all $x\in B_{8R}(x_0)$ and $ \max \{ 8R, d_\Omega (x_0) \} =d_\Omega (x_0)$.

The second summand on the left-hand side of \eqref{eq:FreezingEstimateIntW/ds} can be estimated similarly by distinguishing between the cases $B_{16R}(x_0)\cap \Omega ^c \neq \emptyset$ and $B_{16R}(x_0) \subset \Omega$ and using \autoref{Lem:EstimateOnHswWithuL1}.
\end{proof}

\section{\texorpdfstring{$C^{s-\eps}$}{Cs-eps} regularity for operators with coefficients}
\label{sec:Cs-epsRegularity}
The goal of this section is to prove the following theorem.
\begin{Theorem}\label{Thm:Cs-epsRegWithCoefficients}
    Let $s\in (0,1)$, $\sigma \in (0,s)$, and $\Omega \subset \R^d$ be a Lipschitz domain with Lipschitz constant $\delta >0$.
    Furthermore, let $\mathcal{L}_t$ be an operator of the form \eqref{eq:OperatorDivergenceForm}-\eqref{eq:KernelDivergenceForm}, satisfying \eqref{eq:KernelHoelderCont} with $\mathcal{A}=Q_1$.
    Assume that $u$ is a weak solution of 
    \begin{equation*}
        \left\{ \begin{aligned}
          \partial _t u -\mathcal{L}_tu & =f & &\text{in} \quad (-1 ,0) \times (\Omega \cap B_1) , \\
           u & =0 & & \text{in} \quad  (-1 ,0) \times (B_1 \setminus \Omega) 
        \end{aligned} \right.
    \end{equation*}
    for some $f\in L^q_tL^r_x (Q_1^\Omega)$ where $q,r\geq 1$ satisfy $\frac{1}{q}+\frac{d}{2sr} <1$. 
    \begin{itemize}
        \item If $\frac{1}{q}+\frac{d}{2sr} \leq \frac{1}{2}$, then for every $\eps >0$ there exists $\delta _0 = \delta _0(d,s,\sigma ,\lambda ,\Lambda,q,r , \eps )>0$ such that if $\delta \leq \delta _0$, then
        \begin{equation*}
            [ u ] _{C^{s- \eps}_{p} (\overline{Q^\Omega_{1/2}})} \leq c \left( \| u \| _{L^\infty ((-1,0)\times \R ^d)} + \| f\| _{L_{t}^q L_x^{r} (Q^\Omega_1)}\right)
        \end{equation*}
        for some constant $c=c(d,s,\sigma ,\lambda ,\Lambda,q,r , \eps )>0$.
        \item If $\frac{1}{2}<\frac{1}{q}+\frac{d}{2sr} < 1$, then there exists $\delta _0 = \delta _0 (d,s,\sigma ,\lambda ,\Lambda,q,r )>0$ such that if $\delta \leq \delta _0$, then for $\beta = 2s-\frac{2s}{q}-\frac{d}{r}$
        \begin{equation*}
            [ u ] _{C^{\beta }_{p} (\overline{Q^\Omega_{1/2}})} \leq c \left( \| u \| _{L^\infty ((-1,0)\times \R ^d)} + \| f\| _{L_{t}^q L_x^{r} (Q^\Omega_1)}\right)
        \end{equation*}
        for some constant $c=c(d,s,\sigma ,\lambda ,\Lambda,q,r )>0$.
    \end{itemize}
\end{Theorem}
We start with the following Morrey-type estimate.
\begin{Lemma}\label{Lem:MorreyEstimateTranslInvariant}
    Let $\Omega \subset \R^d$ be a Lipschitz domain with Lipschitz constant $\delta >0$.
    Let $L$ be a translation invariant operator of the form \eqref{eq:TranslationalInvariantOperator}-\eqref{eq:TranslationalInvariantKernel}.
    Given some $z_0=(t_0,x_0) \in \R \times \R ^d$ and $R \in (0,1)$, assume that $v$ solves
    \begin{equation*}
        \left\{ \begin{aligned}
          \partial _t v -Lv & =0 & &\text{in} \quad I^\ominus _R(t_0) \times (B_R(x_0)\cap \Omega ) , \\
           v & =0 & & \text{in} \quad  I^\ominus _R(t_0) \times ( B_R(x_0) \setminus \Omega ) .
        \end{aligned} \right.
    \end{equation*}
    Then, for every $\eps \in (0,s)$, there exist $\delta _0=\delta _0(d,s,\lambda , \Lambda ,\eps )>0$ and $c=c(d,s,\lambda , \Lambda ,\eps )>0$, such that whenever $\delta \leq \delta_0$, it holds that
    \begin{equation}\label{eq:MorreyEstimateTranslInvariant}
    \begin{aligned}
        \sup _{t\in I^\ominus _\rho (t_0)}\int_{B ^\Omega_\rho (x_0) } \left| \frac{v}{d_\Omega ^s}  \right|\intd x &  \leq c \rho ^{-2s}\left( \frac{\rho}{R}\right) ^{d+2s-\eps } \Bigg[  \iint_{Q_R^\Omega (z_0) } \left| \frac{v}{d_\Omega ^s}  \right|\intd x \intd t  \\
        & \qquad  + \max \{ R,d_\Omega (x_0)\} ^{-s}  R^{d+2s} \sup _{t \in I^\ominus _{R}(t_0)} \Tail (v;R/2,x_0)  \Bigg] 
    \end{aligned}
    \end{equation}
    for all $0<\rho \leq R/16$.
\end{Lemma}
\begin{proof}
First, we will assume that $d_\Omega(x_0) \leq R/8$ and $\rho \leq R/16$.
Then, for every $x\in B^\Omega_\rho (x_0)$ there exists $y\in B_{R/4}(x_0) \cap \partial \Omega$ such that $d_\Omega (x)=|x-y|$.
Note that $v(t,y)=0$.
Hence, we can estimate
\begin{align*}
     \sup _{t\in I^\ominus _\rho (t_0)}\int_{B^\Omega_\rho (x_0)} \left| \frac{v}{d_\Omega ^s}  \right|\intd x & \leq  \sup _{t\in {I^\ominus _\rho (t_0)}}  \int _{B_{\rho}^\Omega (x_0) } d_\Omega^{-\eps} (x) \sup _{y\in B^\Omega_{R/4}(x_0)} \frac{|v(t,x)-v(t,y)|}{|x-y|^{s- \eps }} \intd x \\
    & \leq \left( \int _{B^\Omega _{\rho}(x_0) } d_\Omega^{-\eps} (x)\intd x \right) [v]_{C_x^{s-\eps} (Q^\Omega _{R/4} (z_0))} \\
    &\leq c \rho ^{d-\eps}  [v]_{C_x^{s-\eps} (Q^\Omega_{R/4} (z_0))},
\end{align*}
where we have used \cite[Lemma 2.5]{Kim2024} in the last step.
Recall that $[\cdot ]_{C_x^{s-\eps} (Q^\Omega _{R/4} (x_0))}$ denotes the Hölder semi-norm in space with the supremum in time.
Next, we want to use the boundary regularity estimate for translation invariant operators from \autoref{Thm:BoundaryRegularityTranslationInvar}(i).
By scaling and truncating, we get 
\begin{equation*}
    [v]_{C_x^{s-\eps} (Q^\Omega _{R/4} (x_0))} \leq c R^{-(s-\eps )} \left(  \| v\| _{L^\infty (Q_{R/2}(z_0))} + \sup _{t \in I^\ominus _{R/2}(t_0)} \Tail (v;R/2,x_0) \right) .
\end{equation*}
Hence, together with \autoref{Prop:LocalBoundedness}, we obtain
\begin{equation*} 
\begin{aligned}
\sup _{t\in I^\ominus _\rho (t_0)}&\int_{B_\rho ^\Omega (x_0) }  \left| \frac{v}{d_\Omega ^s}  \right|\intd x  \leq
    c  \rho ^{d-\eps} R^{-(s-\eps )} \left(  \| v\| _{L^\infty (Q_{R/2}(z_0))} + \sup _{t \in I^\ominus _{R/2}(t_0)} \Tail (v;R/2,x_0) \right) \\
    & \leq  c \rho ^{d-\eps}R^{-(s-\eps )} \Bigg( R^{-d-2s} \| v\| _{L^1(Q_R(z_0))} +\sup _{t \in I^\ominus _{R}(t_0)} \Tail (v;R/2,x_0) \Bigg) \\
    & \leq c \rho ^{-2s}R^{-s} \left(\frac{\rho}{R} \right) ^{d+2s-\eps} \Bigg(  \| v\| _{L^1(Q_R(z_0))} +  R^{d+2s}\sup _{t \in I^\ominus _{R}(t_0)} \Tail (v;R/2,x_0) \Bigg)  .
\end{aligned}
\end{equation*}
Therefore, \eqref{eq:MorreyEstimateTranslInvariant} follows in the case $d_\Omega (x_0) \leq R/8$ since $d_\Omega (x) \leq 2R$ for $x\in B_R(x_0)$ and $\max \{ R,d_\Omega (x_0)\}=R$.

Now, let us assume that $d_\Omega (x_0) \geq R/8$ and $\rho \leq R/16$.
Then, $d_\Omega (x_0) \leq \rho +d_\Omega (x)$ for all $x\in B_\rho (x_0)$ which implies $d_\Omega (x) \geq d_\Omega (x_0) -\rho \geq \frac{1}{2}d_\Omega(x_0)$ for all $x\in B_\rho (x_0)$. 
Hence, together with \autoref{Prop:LocalBoundedness}, we obtain
\begin{equation*} 
\begin{aligned}
    \sup _{t\in I^\ominus _\rho (t_0)}\int_{B ^\Omega _\rho (x_0) } \left| \frac{v}{d_\Omega ^s}  \right|\intd x & \leq c d_\Omega(x_0)^{-s} \rho ^d \| v\| _{L^\infty (Q_{R/2}(z_0))} \\
    & \leq c d_\Omega (x_0) ^{-s} \rho ^d \left( R^{-d-2s} \| v\| _{L^1(Q_R(z_0))} + \sup _{t\in I_R^\ominus (t_0)} \Tail (v;R,x_0)  \right) \\
    &\leq c \rho^{-2s} \left( \frac{\rho}{R} \right) ^{d+2s-\eps} \Bigg( \left\| \frac{v}{d_\Omega ^s}\right\| _{L^1 (Q^\Omega_R(z_0))} \\
    & \qquad \qquad \qquad+ \max \{ R,d_\Omega (x_0)\} ^{-s} R^{d+2s} \sup _{t\in I_R^\ominus (t_0)} \Tail (v;R,x_0)  \Bigg)      ,
\end{aligned}
\end{equation*}
where in the last step we have used that $d_\Omega (x) \leq d_\Omega (x_0) +R \leq 9 d_\Omega (x_0)$ for all $x\in B_R(x_0)$ and $\max \{ R,d_\Omega (x_0)\} \leq 8 d_\Omega(x_0)$.
Hence, \eqref{eq:MorreyEstimateTranslInvariant} also follows in the case $d_\Omega (x_0) \geq R/8$.
\end{proof}

Set $\theta =128$.
For $\rho >0$ and $z_0=(t_0,x_0)\in Q_1^\Omega$, we define the following excess functional
\begin{equation}\label{eq:ExcessFunctional}
\begin{aligned}
    \Phi _\sigma (u;\rho )= \Phi _{\sigma}(u;\rho ,z_0)& = \iint_{Q ^\Omega_\rho (z_0) } \left| \frac{u}{d_\Omega ^s}  \right|\intd x \intd t \\
    & \qquad + \max \{ \rho ,d_\Omega (x_0)\} ^{-s} \rho ^{d+2s} \sup_{t\in I^\ominus _\rho (t_0)} \Tail _{\sigma ,B_1} (u;\theta ^{-1}\rho  ,x_0) .
    \end{aligned}
\end{equation}

We are now in a position to establish an excess decay estimate for solutions to nonlocal parabolic equations with H\"older continuous coefficients.

\begin{Lemma} \label{Lem:ExcessFunctionalEstimate}
    Assume that we are in the same setting as in \autoref{Thm:Cs-epsRegWithCoefficients}.
    Let $\eps \in (0,s)$.
   Then, there exist $R_0 \in (0,1/16)$, $c>0$, and $\delta _0>0$ depending only on $d,s,\lambda , \Lambda , \sigma ,q,r $, and $\eps$, such that whenever the Lipschitz constant $\delta$ of $\Omega$ satisfies $\delta \leq \delta _0$, we have 
    \begin{equation} \label{eq:ExcessFunctionalEstimate}
        \begin{aligned}
        \Phi _\sigma (u;\theta^{-m}R,z_0) &\leq c \left( \theta^{-m}\right) ^{d+2s+\min \{ s-\frac{d}{r}-\frac{2s}{q},-\eps \}} \Phi _\sigma (u;R,z_0) \\
        & \quad+ c \left( \theta^{-m}R\right) ^{d+2s+\min \{ s-\frac{d}{r}-\frac{2s}{q},-\eps \}} \left( \| u \| _{L^\infty ((-1,0)\times \R ^d)} + \| f\| _{L_{t}^q L_x^{r} (Q^\Omega _1)} \right) 
        \end{aligned}
    \end{equation}
    for all $z_0=(t_0,x_0)\in Q_{1/2}^\Omega$, $0<R\leq R_0$, and $m\in \N$.
\end{Lemma}
\begin{proof}
    Let us denote
\begin{equation*}
    B= \| u \| _{L^\infty ((-1,0)\times \R ^d)} + \| f\| _{L_{t}^q L_x^{r} (Q^\Omega _1)} .
\end{equation*}
By \autoref{Lem:IterationLemmaOnDyadicScale} with $\alpha = d+2s-\frac{\eps}{2}$ and $\beta = d+2s+\min \{ s-\frac{d}{r}-\frac{2s}{q},-\eps \}$ (note that $\beta < \alpha $), it is enough to prove that for all $0<R\leq 1/2$ and all $m\in \N$ it holds
\begin{equation}\label{eq:ExcessFunctionalEstimateBeforeIteration}
    \begin{aligned}
        \Phi _\sigma (u;\theta^{-m}R) & \leq c \left[ \left( \theta^{-m} \right) ^{d+2s-\frac{\eps}{2}}+ R^\sigma  \right] \Phi _\sigma (u;R) + c R^{d+3s-\frac{d}{r}-\frac{2s}{q}} B .
    \end{aligned}
\end{equation}

Fix some $0<R\leq 1/2$ and let $v:=v_{R/8,z_0}$ and $w:=w_{R/8,z_0}$ be defined as in \eqref{eq:DefinitionV} and \eqref{eq:DefinitionW}.
First, we claim that for all $0< \rho \leq \theta ^{-1}R$, we have
\begin{equation}\label{eq:estimate-iteration-for-v-w-seperated}
    \begin{aligned}
       & \rho ^{2s}\left\| \frac{v}{d_\Omega ^s} \right\| _{L^\infty _t L^1_x (Q^\Omega_\rho (z_0))} + \left\| \frac{w}{d_\Omega ^s} \right\| _{L^1 (Q^\Omega _{R/8}(z_0))} + \max \{ R,d_\Omega (x_0)\} ^{-s} R^{2s+\frac{d}{2}} \| w \| _{L^\infty_t L^2_x (Q_{R/8}(z_0))} \\
       & \qquad  \leq c \left[ \left(\frac{\rho}{R} \right) ^{d+2s-\frac{\eps}{2}}+ R^\sigma  \right] \Phi _\sigma (u;R) + c R^{d+3s-\frac{d}{r}-\frac{2s}{q}} B .
    \end{aligned}
\end{equation}
To prove \eqref{eq:estimate-iteration-for-v-w-seperated}, note that by \autoref{Lem:MorreyEstimateTranslInvariant}, we obtain
\begin{align*}
    \rho ^{2s}  \left\| \frac{v}{d_\Omega ^s} \right\| _{L^\infty _t L^1_x (Q^\Omega_\rho (z_0))}    &\leq   c \left( \frac{\rho}{R}\right) ^{d+2s-\eps } \Bigg[  \iint_{Q_{R/8}^\Omega (z_0) } \left| \frac{v}{d_\Omega ^s}  \right|\intd x \intd t  \\
        & \qquad  + \max \{ R,d_\Omega (x_0)\} ^{-s}  R^{d+2s} \sup _{t \in I^\ominus _{R/8}(t_0)} \Tail (v;R/16,x_0)  \Bigg]  \\
    & \leq c \left( \frac{\rho}{R}\right) ^{d+2s-\eps } \Phi _\sigma (u;R) + c R^{d+3s}B \\
    & \qquad  + c \left\| \frac{w}{d_\Omega ^s} \right\| _{L^1 (Q_{R/8}(z_0))} + c \max \{ R,d_\Omega (x_0)\} ^{-s} R^{2s+\frac{d}{2}} \| w \| _{L^\infty_t L^2_x (Q_{R/8}(z_0))} ,
\end{align*}
where in the last step, we have again used that $u=v+w$ and the following estimate:
\begin{align*}
    &\max \{ R,d_\Omega (x_0)\} ^{-s}  R^{d+2s} \sup _{t \in I^\ominus _{R/8}(t_0)} \Tail (v;R/16,x_0) \\
    &\quad \le c \max \{ R,d_\Omega (x_0)\} ^{-s}  R^{d+2s} \left( \sup _{t \in I^\ominus _{R}(t_0)} \Tail (u;\theta^{-1} R,x_0) + R^{-d} \sup _{t \in I^\ominus _{R/8}(t_0)} \int_{B_{R/8}(x_0)} |w| \d x \right) \\
    &\quad \le c\Phi _\sigma (u;R) +c R^{d+3s} \| u\| _{L^\infty ((-1,0)\times \R^d )}+ c \max \{ R,d_\Omega (x_0)\} ^{-s} R^{2s+\frac{d}{2}} \| w \| _{L^\infty_t L^2_x (Q_{R/8}(z_0))}.
\end{align*}
Here we used that $ \theta^{-1} R < R/16$ and $w = 0$ in $I_{R/8}^{\ominus}(t_0) \times (\R^d \setminus B_{R/8}(x_0))$ and applied \eqref{eq:TailSplitSigmaTail}.
Hence, from here, \eqref{eq:estimate-iteration-for-v-w-seperated} follows from \autoref{Lem:FreezingEstimateIntW/ds}.

Note that \eqref{eq:estimate-iteration-for-v-w-seperated} proves that the first summand in the excess functional \eqref{eq:ExcessFunctional} satisfies \eqref{eq:ExcessFunctionalEstimateBeforeIteration} since $u=v+w$ and
\begin{equation*}
     \left\| \frac{u}{d_\Omega ^s} \right\| _{L^1 (Q^\Omega _{\theta ^{-m}R}(z_0))} \leq (\theta ^{-m}R) ^{2s}\left\| \frac{v}{d_\Omega ^s} \right\| _{L^\infty _t L^1_x (Q^\Omega_{\theta ^{-m}R} (z_0))} + \left\| \frac{w}{d_\Omega ^s} \right\| _{L^1 (Q^\Omega_{R/8}(z_0))} .
\end{equation*}

It remains to prove that
\begin{equation} \label{eq:TailEstimateForIteration}
\begin{aligned}
    \max \{ \theta^{-m}R ,d_\Omega (x_0)\} ^{-s} & (\theta^{-m}R) ^{d+2s} \sup_{t\in I^\ominus _{\theta^{-m}R} (t_0)} \Tail _{\sigma ,B_1} (u;\theta^{-m-1}R  ,x_0) \\
    & \leq c \left[ \left( \theta^{-m} \right) ^{d+2s-\frac{\eps}{2}}+ R^\sigma  \right] \Phi _\sigma (u;R) + c R^{d+3s-\frac{d}{r}-\frac{2s}{q}} B .
\end{aligned}
\end{equation}
We start by proving \eqref{eq:TailEstimateForIteration} for the case $d_\Omega (x_0) \leq \theta^{-m}R$.
Then, we have
\begin{align*}
    \max \{ & \theta^{-m}R ,d_\Omega (x_0)\} ^{-s}  (\theta^{-m}R) ^{d+2s} \sup_{t\in I^\ominus _{\theta^{-m}R} (t_0)} \Tail _{\sigma ,B_1} (u;\theta^{-m-1}R  ,x_0) \\
    & = c\left( \theta^{-m}R\right) ^{d+s} \sup_{t\in I^\ominus _{\theta^{-m} R} (t_0)} \left( \theta^{-m}R \right) ^{2s-\sigma} \sum _{k=0}^{m-1} \int _{B_{\theta^{k-m}R}(x_0) \setminus B_{\theta^{k-m-1}R}(x_0)} \frac{|u(t,y)|}{|x_0-y|^{d+2s-\sigma }} \intd y \\
    & \qquad + c\left( \theta^{-m}R\right) ^{d+s} \sup_{t\in I^\ominus _{\theta^{-m} R} (t_0)} \left( \theta^{-m}R \right) ^{2s-\sigma}  \int _{B_1 \setminus B_{\theta^{-1}R}(x_0)} \frac{|u(t,y)|}{|x_0-y|^{d+2s-\sigma }} \intd y \\
    & =I_1+ I_2 .
\end{align*}
Note that we can estimate $I_2$ by
\begin{equation*}
    I_2 \leq c  \left( \theta^{-m} \right) ^{d+3s-\sigma} R^{d+s} \sup_{t\in I^\ominus _R (t_0)} \Tail _{\sigma ,B_1} (u;\theta^{-1}R  ,x_0) \leq c \left( \theta^{-m} \right) ^{d+2s-\frac{\eps}{2}} \Phi _\sigma (u;R),
\end{equation*}
where we have used that $\sigma <s$.
For $I_1$, we use the decomposition $u=v+w$ and \eqref{eq:estimate-iteration-for-v-w-seperated}, to obtain
\begin{align*}
    I_1 & \leq c \left( \theta^{-m}R\right) ^{d+3s-\sigma} \sum _{k=0}^{m-1} \left( \theta^{k-m} R \right) ^{-d-2s+\sigma} \| u\| _{L^\infty _t L^1_x (Q_{\theta^{k-m}R} (z_0))}  \\
    & \leq  c \left( \theta^{-m}R\right) ^{d+3s-\sigma} \sum _{k=0}^{m-1} \left( \theta^{k-m} R \right) ^{-d-2s+\sigma} \left[ \| v\| _{L^\infty _t L^1_x (Q_{\theta^{k-m}R} (z_0))} + \| w\| _{L^\infty _t L^1_x (Q_{\theta^{k-m}R} (z_0))} \right] \\
    & \leq c  \sum _{k=0}^{m-1} \left( \theta^{k}  \right) ^{-d-3s+\sigma} \left[ \left( \theta^{k-m} R \right) ^{2s}\| v /d_\Omega ^s\| _{L^\infty _t L^1_x (Q_{\theta^{k-m}R} (z_0))} +  R^{s+\frac{d}{2}} \| w\| _{L^\infty _t L^2_x (Q_{R/8} (z_0))}\right] \\
    & \leq c \sum _{k=1}^{m}  \theta^{k(-d-3s+\sigma )} \Bigg( \left[ \left( \theta^{k-m} \right) ^{d+2s-\frac{\eps}{2}}+ R^\sigma  \right] \Phi _\sigma (u;R) + c R^{d+3s-\frac{d}{r}-\frac{2s}{q}} B \Bigg) \\
    & = c \left( \sum _{k=1}^{m} \theta^{k(\sigma-s-\frac{\eps}{2})} \right) \left( \theta^{-m} \right) ^{d+2s-\frac{\eps}{2}} \Phi _\sigma (u;R) \\
    & \qquad + c \left( \sum _{k=1}^{m}  \theta^{k(-d-3s+\sigma )}  \right) \left( R^\sigma \Phi _\sigma (u;R) + R^{d+3s-\frac{d}{r}-\frac{2s}{q}} B \right) .
\end{align*}
Note that the last term is bounded by the right-hand side of \eqref{eq:TailEstimateForIteration} because both of the series $\sum _{k=1}^{\infty} \theta^{k(\sigma-s-\frac{\eps}{2})}$ and $\sum _{k=1}^{\infty}  \theta^{k(-d-3s+\sigma )}$ are finite since $\sigma <s$.
This proves \eqref{eq:TailEstimateForIteration} for the case $d_\Omega (x_0) \leq \theta^{-m}R$.

Now, let us assume $d_\Omega (x_0) \geq \theta^{-m}R$.
Then, \eqref{eq:TailEstimateForIteration} can be proved exactly as above.
We only have to treat the term $\max \{  \theta^{-m}R ,d_\Omega (x_0)\} ^{-s} = d^{-s}_\Omega (x_0)$ on the left-hand side of \eqref{eq:TailEstimateForIteration} differently.
For the term $I_2$ we keep the term $d^{-s}_\Omega (x_0)$ in the case $d_\Omega(x_0)\geq R$ and use  the trivial estimate $d^{-s}_\Omega (x_0) \leq (\theta^{-m}R)^{-s}$ in the case $R\geq d_\Omega(x_0) \geq \theta^{-m}R$.
In the computation for the term $I_1$, we use the estimate
\begin{align*}
\| v\| _{L^\infty_t L^1_x (Q_{\theta ^{k-m}R}(z_0))} &\leq c \max \{ d_\Omega (x_0) ,\theta ^{k-m}R\} ^s \| v /d_\Omega ^s\| _{L^\infty _tL^1_x (Q_{\theta ^{k-m}R}(z_0))}\\
&\leq c \theta ^{ks}  d^s_\Omega (x_0)  \| v /d_\Omega ^s\| _{L^\infty _t L^1_x (Q_{\theta ^{k-m}R}(z_0))} 
\end{align*}
for the summand containing $v$.
For the summand with $w$, we again keep the term $d^{-s}_\Omega (x_0)$ from the left-hand side of \eqref{eq:TailEstimateForIteration} in the case $d_\Omega(x_0)\geq R$ and use the trivial estimate $d^{-s}_\Omega (x_0) \leq (\theta^{-m}R)^{-s}$ in the case $R\geq d_\Omega(x_0) \geq \theta^{-m}R$.
\end{proof}

The following lemma is a consequence of the interior regularity for solutions to nonlocal parabolic equations (see \autoref{Prop:InteriorRegularity}).

\begin{Lemma}\label{Lem:InteriorCampanatoEstimate}
    Let $\mathcal{L}_t$ be an operator of the form \eqref{eq:OperatorDivergenceForm}-\eqref{eq:KernelDivergenceForm}, satisfying \eqref{eq:KernelHoelderCont} with $\mathcal{A}=Q_R(z_0)$.
    Let $R>0$ be given.
Assume that $u$ is a weak solution of 
\begin{equation*}
    \partial _t u - \mathcal{L}_tu =f \quad \text{in}\quad Q_R(z_0) 
\end{equation*}
for some $f\in L^q_tL^r_x (Q_R(z_0))$ where $q,r\geq 1$ satisfy $\frac{1}{q}+\frac{d}{2sr} <1$.
Furthermore, let $\gamma$ be a positive number satisfying $\gamma <1$ and $\gamma \leq  2s- \left( \frac{d}{r}+\frac{2s}{q} \right) $.
    Then, for all $0<\rho < R$
    \begin{align*}
        \dashint _{Q_\rho (z_0)} |u -(u) _{Q_\rho (z_0)} |\intd z & \leq c \left( \frac{\rho}{R} \right) ^\gamma  \Bigg( R^{-d-2s} \| u\|_{L^1(Q_R(z_0))} + \sup _{t\in I^\ominus _R(t_0)} \Tail (u ;R/2 ,x_0)  \\
    & \qquad \qquad \qquad \qquad +R^{2s- \left( \frac{2s}{q} + \frac{d}{r}  \right)} \| f\| _{L^{q}_{t} L^r_x(Q_R(z_0))}\Bigg)
    \end{align*}
    for some constant $c=c(d,s,\lambda ,\Lambda, \sigma , \gamma ,q ,r)>0$
\end{Lemma}
\begin{proof}
    First, note that we can assume $0<\rho \leq R/4$.
    Using the interior regularity estimate from \autoref{Prop:InteriorRegularity} and local boundedness from \autoref{Prop:LocalBoundedness}, we get
    \begin{align*}
        \dashint _{Q_\rho (z_0)} |u -(u) _{Q_\rho (z_0)} |\intd z & \leq 
        \sup _{z,\tilde{z} \in Q_\rho (z_0)} |u(z)-u(\tilde{z})| 
        \leq c \rho ^\gamma [u]_{C^{\gamma} _{p} (Q_\rho (z_0))} 
         \leq c \rho ^\gamma [u]_{C^{\gamma} _{p} (Q_{R/4} (z_0))} \\
         & \leq c \left( \frac{\rho}{R} \right) ^\gamma  \Bigg( R^{-d-2s} \| u\|_{L^1(Q_R(z_0))} + \sup _{t\in I^\ominus _R(t_0)} \Tail (u ;R/2 ,x_0)  \\
    & \qquad \qquad \qquad \qquad +R^{2s- \left( \frac{2s}{q} + \frac{d}{r}  \right)} \| f\| _{L^{q}_{t} L^r_x(Q_R(z_0))}\Bigg) .
    \end{align*}
\end{proof}

\begin{proof}[Proof of \autoref{Thm:Cs-epsRegWithCoefficients}]
First, we will show \autoref{Thm:Cs-epsRegWithCoefficients} for the case $\frac{1}{q}+\frac{d}{2sr} \leq \frac{1}{2}$.
Set $\beta = s-\eps$.
Let $\delta _0$ and $R_0$ be given from \autoref{Lem:ExcessFunctionalEstimate}.
Assume that 
\begin{equation*}
   \| u \| _{L^\infty ((-1,0)\times \R ^d)} + \| f\| _{L_{t}^q L_x^{r} (Q^\Omega_1)}  \leq 1.
\end{equation*}
It is enough to show (see \cite[Lemma 2.1]{Byun2023}) that for all $z_0\in Q^\Omega _{1/2}$, we have
\begin{equation} \label{eq:CampanatoDecayCs-eps}
    \int _{Q ^\Omega_\rho (z_0) }  | u -(u) _{Q^\Omega_\rho (z_0) } | \intd z \leq c \rho ^{d+2s+ \beta } 
\end{equation}
for all $0 < \rho \leq R_0$.

First, assume that $d_\Omega(x_0) /2 \leq \rho < R_0$.
Choose $m\in \N$ such that $\theta ^{-m-1}R_0 \leq \rho \leq \theta ^{-m}R_0$.
Then, by applying \autoref{Lem:ExcessFunctionalEstimate}, we obtain 
\begin{equation} \label{eq:CampanatoDecayCs-epsFirstCase}
\begin{aligned}
    \int _{Q ^\Omega_\rho (z_0) }  | u -(u) _{Q^\Omega_\rho (z_0) } | \intd z & \leq 2\int _{Q ^\Omega_\rho (z_0) }  | u  | \intd z 
    \leq c \rho ^s \int _{Q ^\Omega_\rho (z_0) }  \left| u/d_\Omega^s  \right| \intd z
    \leq c \rho ^s \int _{Q ^\Omega_{\theta ^{-m}R_0} (z_0) }  \left| u/d_\Omega^s  \right| \intd z \\
    & \leq c \rho ^{d+3s+ \min \{  s-\frac{2s}{q}-\frac{d}{r},- \eps \}} \left(  \Phi _\sigma (u;R_0,z_0) +1 \right) \\
    & \leq c \rho ^{d+2s+ \beta } \left(  \Phi _\sigma (u;R_0,z_0) +1 \right) .
\end{aligned}
\end{equation}
Let $\eta \in C^\infty _c (B_{7/8})$ be a cut-off function with $\eta = 1$ on $B_{3/4}$, $0 \leq \eta \leq 1$.
Then, using Hardy's inequality (see \cite[Lemma 2.4]{Kim2024}), and \autoref{Prop:CaccioppoliEstimate} (and assuming $R_0\leq 1/4$), we get
\begin{align*}
    \| u/d_\Omega^s \| _{L^1(Q_{R_0}(z_0))} & \leq c \| u\eta / d_\Omega ^s \| _{L^2 (I_{3/4}^\ominus \times B_1)} \leq c \| u\eta \| _{L^2_t \dot{H}_x^s(I_{3/4}^\ominus \times \R^d)} \\
    & \leq c \| u\| _{L^2_t \dot{H}^s_x (I_{3/4}^\ominus \times B_{15/16})} +c \| u\| _{L^2(I_{3/4}^\ominus \times B_{7/8})} \leq c  .
\end{align*}
Since we also have $\sup _{t\in I^\ominus_{R_0}(t_0)} \Tail _{\sigma ,B_1}(u;\theta ^{-1} R_0,x_0 ) \leq c$, we obtain $\Phi (u;R_0 ,z_0) \leq c $ where $c$ also depends on $R_0$.
Hence, this proves \eqref{eq:CampanatoDecayCs-eps} in the case $d_\Omega(x_0) /2 \leq \rho < R_0$.

Now, let us assume $\rho < d_\Omega(x_0) /2 \leq R_0$ and set $K:=d_\Omega(x_0)$.
We use \autoref{Lem:InteriorCampanatoEstimate} with $\rho := \rho$, $R:=K/2$, and $\gamma = \beta $, to obtain
\begin{align*}
    \int _{Q ^\Omega_\rho (z_0) }  | u -(u) _{Q^\Omega_\rho (z_0) } | \intd z & \leq 
    c \left( \frac{\rho}{K} \right) ^{d+2s+\beta} \| u\|_{L^1 (Q_{K/2}(z_0))} \\
    & \quad + c \left( \frac{\rho}{K} \right) ^{d+2s+\beta}  K^{d+2s}\sup _{t\in I^\ominus _{K/2}(t_0)} \Tail (u ;K/4 ,x_0) \\
    & \quad + c \left( \frac{\rho}{K} \right) ^{d+2s+\beta}   K^{d+4s- \left( \frac{2s}{q} + \frac{d}{r}  \right)} \| f\| _{L^{q}_{t} L^r_x(Q_{K/2}(z_0))} .
\end{align*}
Applying \autoref{Lem:ExcessFunctionalEstimate} as in \eqref{eq:CampanatoDecayCs-epsFirstCase}, yields
\begin{align*}
     \| u\|_{L^1 (Q_{K/2}(z_0))} & \leq  c K^s \| u/d_\Omega ^s\|_{L^1 (Q_{K/2}(z_0))} \leq c K^{d+2s+\beta  } .
\end{align*}
Furthermore, choose $\mu \in [1/2,1]$ and $m\in \N$ such that $\theta ^{-m-1}\mu R_0 \leq K/4 \leq K/2 \leq \theta ^{-m} \mu R_0$.
Using \autoref{Lem:ExcessFunctionalEstimate} and $\Phi _\sigma (u;\mu R_0,z_0)\leq c$ for a constant $c$ independent of $\mu$, we obtain
\begin{align*}
   K^{d+s} \sup _{t\in I^\ominus _{K/2}(t_0)} \Tail (u ;K/4 ,x_0) & \leq c \Phi_\sigma (u; \theta ^{-m}\mu R_0,z_0) \leq c (\theta^{-m}) ^{d+s+\beta}(\Phi _\sigma (u;\mu R_0,z_0)+1) \\
   & \leq c K^{d+s+\beta }.
\end{align*}
Since $K^{2s-\frac{2s}{q}-\frac{d}{r}-\beta} \leq c$, this proves \eqref{eq:CampanatoDecayCs-eps} for the case $\rho < d_\Omega(x_0) /2 \leq R_0$.

Finally, if we assume $d_\Omega(x_0) \geq 2 R_0$, then \eqref{eq:CampanatoDecayCs-eps} follows immediately from \autoref{Lem:InteriorCampanatoEstimate}.
This proves \autoref{Thm:Cs-epsRegWithCoefficients} for the case $\frac{1}{q}+\frac{d}{2sr} \leq \frac{1}{2}$.

If we assume $\frac{1}{2}<\frac{1}{q}+\frac{d}{2sr} < 1$, then the proof follows exactly as above for $\beta = 2s-\frac{2s}{q}-\frac{d}{r}$. Note that we can apply \autoref{Lem:ExcessFunctionalEstimate} with $\eps = s-\beta$.
\end{proof}

\section{\texorpdfstring{$C^{s}$}{Cs} regularity} \label{sec:CsRegularity}
In this section we will show the following theorem.
\begin{Theorem}\label{Thm:CsBoundaryRegularityGlobalAAssumption}
    Let $s\in (0,1)$, $\alpha,\sigma \in (0,s)$, and $\Omega \subset \R^d$ be a $C^{1,\alpha}$ domain.
    Furthermore, let $\mathcal{L}_t$ be an operator of the form \eqref{eq:OperatorDivergenceForm}-\eqref{eq:KernelDivergenceForm}, satisfying \eqref{eq:KernelHoelderCont} with $\mathcal{A}=(-1,0)\times \R^d$.
    Assume that $u$ is a weak solution to 
    \begin{equation*}
        \left\{ \begin{aligned}
          \partial _t u -\mathcal{L}_tu & =f & &\text{in} \quad (-1 ,0) \times (\Omega \cap B_1) , \\
           u & =0 & & \text{in} \quad  (-1 ,0) \times (B_1 \setminus \Omega) 
        \end{aligned} \right.
    \end{equation*}
    for some $f\in L^q_tL^r_x (Q_1^\Omega)$ where $q,r\geq 1$ satisfy $\frac{1}{q}+\frac{d}{2sr} < \frac12$.
    Then, 
    \begin{equation*}
            [ u ] _{C^{s}_{p} (\overline{Q^\Omega_{1/2}})} \leq c \left( \| u \| _{L^\infty ((-1,0)\times \R ^d)} + \| f\| _{L_{t}^q L_x^{r} (Q^\Omega_1)}\right)
       \end{equation*}
    for some constant $c=c(d,s,\sigma ,\lambda ,\Lambda,q,r , \alpha ,\Omega )>0$.
\end{Theorem}
Note that in comparison to \autoref{Thm:CsBoundaryRegAndHopfLemma}(i) from the introduction, in \autoref{Thm:CsBoundaryRegularityGlobalAAssumption} we assume that the Hölder continuity of the kernel from \eqref{eq:KernelHoelderCont} holds true globally on $\mathcal{A}=(-1,0)\times \R^d$.
First, we will show how \autoref{Thm:CsBoundaryRegularityGlobalAAssumption} implies \autoref{Thm:CsBoundaryRegAndHopfLemma}(i) by a cut-off argument.
\begin{proof} [Proof of \autoref{Thm:CsBoundaryRegularityGlobalAAssumption} $\Rightarrow$ \autoref{Thm:CsBoundaryRegAndHopfLemma}(i)] 
    Assume that the kernel $K$ satisfies \eqref{eq:KernelHoelderCont} with $\mathcal{A}=Q_1$.
    Then, we define the kernel
    \begin{equation*}
        \tilde{K}(t,x,y)=K(t,x,y) \varphi (x) \varphi (y) + \frac{\lambda+\Lambda}{2} (1- \varphi (x) \varphi (y)) |x-y|^{-d-2s}
    \end{equation*}
    where $\varphi$ is a smooth cut-off function, which satisfies $0\leq \varphi \leq 1$, $\varphi (x) =1$ for all $x\in B_{7/9}$, and $\varphi (x)=0$ for all $x\in B_{8/9}^c$.
    Note that $\tilde{K}$ satisfies \eqref{eq:KernelDivergenceForm} and \eqref{eq:KernelHoelderCont} with $\mathcal{A}=I^\ominus_1 \times \R ^d$.
    
    Let us denote by $\tilde{\mathcal{L}}_t$ the operator with kernel $\tilde{K}$ and set $\tilde{u} = u \ind _{B_{7/9}}$.
    Then $\tilde{u}$ solves the equation
    \begin{equation*}
        \partial_t \tilde{u} -\tilde{\mathcal{L}}_t \tilde{u} = f +\left( \mathcal{L}_t u - \mathcal{L}_t \tilde{u} \right) + \left( \mathcal{L}_t \tilde{u} - \tilde{\mathcal{L}}_t \tilde{u} \right) =: f+f_1 +f_2 \quad \text{in } Q^\Omega _{6/9} .
    \end{equation*}
    Note that we have
    \begin{equation}\label{eq:Boundsf1f2}
    \begin{aligned}
        \| f_1 \| _{L^p_t L^\infty _x(Q^\Omega _{6/9} )} &\leq c  \| \Tail (u;1/2,0) \| _{L^p_t (I^\ominus _{6/9})} , \\
         \text{and} \quad \| f_2 \| _{L^p_t L^\infty _x(Q^\Omega _{6/9} )} & \leq  c \| u \| _{L^\infty (Q_{7/9}^\Omega)}.
        \end{aligned}
    \end{equation}
    For $i=1,2$, let $v_i$ be the unique solution to
    \begin{equation} \label{eq:EquationForViInCutOff}
        \left\{ \begin{aligned}
          \partial _t v_i -\tilde{\mathcal{L}}_tv_i & =f_i & &\text{in} \quad Q^\Omega _{6/9} , \\
           v_i & =0 & & \text{in} \quad  ( I^\ominus_1 \times \R^d ) \setminus Q^\Omega _{6/9} .
        \end{aligned} \right.
    \end{equation}
    Then, $(\partial_t -\tilde{\mathcal{L}}_t) (\tilde{u}-v_1 -v_2) = f$ in $Q^\Omega _{5/9}$ and by using \autoref{Thm:CsBoundaryRegularityGlobalAAssumption} for $\tilde{u}-v_1-v_2$, $v_1$, and $v_2$, we obtain
    \begin{equation} \label{eq:ApplicationCsBoundForCutOff}
        \begin{aligned}
       &  [u ] _{C^{s}_{p} (\overline{Q^\Omega_{1/2}})}  \leq [\tilde{u} -v_1 -v_2 ] _{C^{s}_{p} (\overline{Q^\Omega_{1/2}})} + [v_1 ] _{C^{s}_{p} (\overline{Q^\Omega_{1/2}})} +[v_2 ] _{C^{s}_{p} (\overline{Q^\Omega_{1/2}})} \\
         &\qquad \leq c \left( \| \tilde{u} \| _{L^\infty (I^\ominus _{5/9} \times \R^d )} +\| f\| _{L^q_t L^r_x (Q^\Omega _{5/9})} \right)+c \sum_{i=1}^2 \left( \| v_i \| _{L^\infty (I ^\ominus _{5/9} \times \R^d )} +   \| f_i \| _{L^p_t L^\infty _x(Q^\Omega_{5/9})} \right) .
    \end{aligned}
    \end{equation}
    Furthermore, note that
    \begin{equation}\label{eq:BoundednessVi}
       \| v_i \| _{L^\infty (I ^\ominus _{5/9} \times \R^d )}= \| v_i \| _{L^\infty (I ^\ominus _{5/9} \times B^\Omega_{6/9})} \leq c \left(  \|v_i\| _{L^2 (Q^\Omega _{6/9})} + \|f_i\|_{L^p_t L^\infty _x(Q^\Omega _{6/9})} \right) \leq c \|f_i\|_{L^p_tL^\infty_x(Q^\Omega _{6/9})},
    \end{equation}
    where the first inequality follows by applying local boundedness (\autoref{Prop:LocalBoundedness}) in suitable subdomains of $I^{\ominus}_{5/9} \times B_{6/9}^{\Omega}$ and a straightforward covering argument, and the second inequality is consequence of \cite[Theorem 5.3]{Felsinger2014}.
    Combining \eqref{eq:Boundsf1f2}, \eqref{eq:ApplicationCsBoundForCutOff}, and \eqref{eq:BoundednessVi} yields
    \begin{equation*}
        [u ] _{C^{s}_{p} (\overline{Q^\Omega_{1/2}})}   \leq c \left( \| u\| _{L^\infty (Q_{7/9}^\Omega)} + \| \Tail (u;1/2,0) \|_{L^p_t (I^\ominus_1)} + \| f\| _{L_{t}^q L_x^{r} (Q^\Omega_1)} \right) .
    \end{equation*}
    Applying local boundedness (\autoref{Prop:LocalBoundedness}) for $u$ (here one has to truncate $u$ again to treat the tail as a right-hand side) proves \autoref{Thm:CsBoundaryRegAndHopfLemma}(i).
\end{proof}

Given a translation invariant operator $L$ of the form \eqref{eq:TranslationalInvariantOperator}-\eqref{eq:TranslationalInvariantKernel}, we define $\psi$ as the solution of \eqref{eq:EllipticBarrierUBar}.
Furthermore, for $z_0 =(t_0,x_0)\in Q^\Omega _1$, we set
\begin{equation}\label{eq:ExcessFunctional-higher-Bar}
    \begin{aligned}
        \overline{\Psi}(v;R,z_0) &= \int _{Q^\Omega _R (z_0)} \left| \frac{v}{\psi} - \left( \frac{v}{\psi} \right) _{Q_R ^\Omega (z_0) } \right| \intd z\\ 
            & \qquad+ \max \{ R,d_\Omega (x_0) \} ^{-s} R^{d+2s} \sup _{t \in I^\ominus _R(t_0)}\Tail (v-\psi (v/\psi )_{Q^\Omega _R (z_0)} ; R/4,x_0) .
    \end{aligned}
\end{equation}
Note that for technical reasons later (see proof of \autoref{Lem:CampanatoEstimate}), we introduce here another excess functional $\overline{\Psi}$ which (compared to $\Psi$ from \eqref{eq:ExcessFunctional-higher}) has a tail with a different constant ($R/4$ compared to $\theta ^{-1}R$).

\begin{Lemma}\label{Lem:CampanatoEstimateTranslInvariant}
    Let $\alpha \in (0, s)$ and $\eps \in (0, s)$.
     Let $\Omega \subset \R^d$ be a $C^{1,\alpha }$ domain with $0\in \partial \Omega$.
    Furthermore, let $L$ be a translation invariant operator of the form \eqref{eq:TranslationalInvariantOperator}-\eqref{eq:TranslationalInvariantKernel}.
    By $\psi$ we denote the solution to \eqref{eq:EllipticBarrierUBar} with respect to $L$.
    Given some $z_0=(t_0,x_0) \in Q^\Omega_{1/8}$ and $R \in (0,1/8)$, assume that $v$ solves
    \begin{equation*}
        \left\{ \begin{aligned}
          \partial _t v -Lv & =0 & &\text{in} \quad I^\ominus _R(t_0) \times (B_R(x_0)\cap \Omega ) , \\
           v & =0 & & \text{in} \quad  I^\ominus _R(t_0) \times ( B_R(x_0) \setminus \Omega ) .
        \end{aligned} \right.
    \end{equation*}
    Then, there exists some constant $c=c(d,s, \lambda , \Lambda ,\alpha , \eps , \Omega )>0$ such that
    \begin{equation} \label{eq:CampanatoEstimateTranslInvariantSupT}
        \begin{aligned}
             \sup _{t \in I_\rho ^\ominus(t_0)} \int _{B^\Omega_\rho (x_0)} & \left| \frac{v}{\psi} - \left( \frac{v}{\psi } \right) _{Q_\rho ^\Omega (z_0) } \right| \intd x  
            \leq c  \rho ^{-2s}\left( \frac{\rho}{R} \right) ^{d+2s+\eps}  \overline{\Psi}(v;R,z_0)
        \end{aligned}
    \end{equation}
    holds for all $0<\rho \leq R /16$, where $\overline{\Psi}(v;R,z_0)$ is defined in \eqref{eq:ExcessFunctional-higher-Bar}.
\end{Lemma}
\begin{proof}
    First, assume $d_\Omega (x_0) \leq \rho \leq R/4 $.
    Let $z\in \partial \Omega \cap B_{\rho }  (x_0)$ such that $d_\Omega(x_0) =|z-x_0|$ and let $c_0 , q\in \R$ be constants to be chosen later.
    Then,
    \begin{align*}
        \sup _{t \in I^\ominus _\rho (t_0)} & \int _{B ^\Omega_\rho (x_0)} \left| \frac{v}{\psi} - \left( \frac{v}{\psi} \right) _{Q_{\rho}^\Omega (z_0)} \right| \intd x  =
        \sup _{t \in I^\ominus _\rho (t_0)} \int _{B^\Omega _\rho (x_0)} \left| \frac{v-c_0\psi }{\psi} - \left( \frac{v-c_0 \psi}{\psi} \right) _{Q_{\rho}^\Omega (z_0)} \right| \intd x \\
        & \leq  2\sup _{t \in I^\ominus _\rho (t_0)} \int _{B ^\Omega _\rho (x_0)} \left| \frac{v-c_0\psi }{\psi} - q \right| \intd x   \\
        & \leq c \left( \int _{B_\rho ^\Omega (x_0)} d_\Omega ^{-s} \intd x \right) \rho ^{s+\eps } \sup _{(t,x) \in Q_\rho ^\Omega (z_0)} \left| \frac{(v-c_0\psi ) (t,x)-q \psi (x)}{|x-z|^{s +\eps }} \right| \\
         & \leq c  \rho ^{d+\eps } \sup _{(t,x) \in Q_{R/4} ^\Omega (z_0)} \left| \frac{(v-c_0\psi ) (t,x)-q \psi (x)}{|x-z|^{s +\eps }} \right| ,
    \end{align*}
    where we have used $\psi \asymp d_\Omega^s$ (\autoref{Lem:psi-properties}(ii)) and \cite[Lemma 2.5]{Kim2024}.
    Next, we will use the expansion from \autoref{Prop:ExpansionTranslationalInvariant} for the function $\tilde{v} := v-c_0\psi $.
    By scaling, shifting, and truncating \autoref{Prop:ExpansionTranslationalInvariant}, we can choose $q\in \R$ such that
    \begin{align*}
       |\tilde v(t,x) -q \psi (x)| \leq c R^{-s-\eps} |x-z| ^{s+\eps } \left( \| \tilde v\| _{L^\infty (Q^\Omega _{R/2}(z_0))} + \sup _{t\in I^\ominus _{R/2}(t_0)} \Tail (\tilde v;R/2,x_0) \right)  
    \end{align*}
    for all $(t,x) \in Q_{R/4}(z_0)$.
    Hence, also using local boundedness (\autoref{Prop:LocalBoundedness}), we obtain
    \begin{align*}
        \sup _{t \in I^\ominus _\rho (t_0)} & \int _{B ^\Omega_\rho (x_0)} \left| \frac{v}{\psi} - \left( \frac{v}{\psi} \right) _{Q_{\rho}^\Omega (z_0)} \right| \intd x \\
        & \leq c \rho ^{d+\eps} R^{-s-\eps } \left( \| v-c_0 \psi \| _{L^\infty (Q^\Omega _{R/2}(z_0))} + \sup _{t\in I^\ominus _{R/2}(t_0)} \Tail ( v-c_0 \psi ;R/2,x_0) \right) \\
        & \leq c \rho ^{d+\eps} R^{-s-\eps } \left( R^{-d-2s} \| v-c_0 \psi \| _{L^1 (Q^\Omega _{R}(z_0))} + \sup _{t\in I^\ominus _{R}(t_0)} \Tail ( v-c_0 \psi ;R/2,x_0) \right) \\
        & \leq c \rho ^{d+\eps} R^{-s-\eps } \left( R^{-d-s} \| v/\psi -c_0  \| _{L^1 (Q^\Omega _{R}(z_0))} + \sup _{t\in I^\ominus _{R}(t_0)} \Tail ( v-c_0 \psi ;R/2,x_0) \right) ,
    \end{align*}
    where in the last step we have used that $\psi \leq c d_\Omega ^s \leq c R^s$ in $B^\Omega_R(x_0)$ by \autoref{Lem:psi-properties}(ii).
    Consequently, \eqref{eq:CampanatoEstimateTranslInvariantSupT} follows in the case $d_\Omega (x_0) \leq \rho\leq R/4$ by setting $c_0= (v/ \psi) _{Q_R(z_0)}$.

    Now, we assume $d_\Omega(x_0)\geq R \geq 4\rho$.
    Then,
    \begin{equation}\label{eq:CampanatoCaseD0geqR}
    \begin{aligned}
        \sup _{t \in I^\ominus _\rho (t_0)}  \int _{B ^\Omega_\rho (x_0)} \left| \frac{v}{\psi} - \left( \frac{v}{\psi} \right) _{Q_{\rho}^\Omega (z_0)} \right| \intd x
        & \leq 2 \sup _{t \in I^\ominus _\rho (t_0)}  \int _{B ^\Omega_\rho (x_0)} \left| \frac{v}{\psi} -  \frac{v}{\psi}(t,x_0)  \right| \intd x \\
         & \leq c \rho ^{d+\eps } \left[ \frac{v-c_0 \psi}{\psi}  \right] _{C ^\eps_{x} (Q ^\Omega _{R/4} (z_0))} \\
         & \leq c \rho ^{d+ \eps} [v-c_0 \psi ]_{C ^\eps_{x} (Q^\Omega _{R/4} (z_0))} \| \psi ^{-1} \| _{L^\infty (B^\Omega_{R/4}(x_0))} \\
         & \qquad + c \rho ^{d+\eps } \| v-c_0 \psi \| _{L^\infty (Q^\Omega_{R/4}(z_0))} \| \psi ^{-1} \| _{C ^\eps_x (B ^\Omega_{R/4} (x_0))} .
    \end{aligned}
    \end{equation}
    Using the interior regularity estimates for translation invariant operators (\autoref{Prop:InteriorRegularity}) and local boundedness (\autoref{Prop:LocalBoundedness}), we obtain
    \begin{align*}
        &[v-c_0 \psi ]_{C ^\eps_{x} (Q^\Omega _{R/4} (z_0))}  \leq c R^{-\eps} \left( \| v-c_0 \psi \| _{L^\infty (Q ^\Omega _{R/2} (z_0))} + \sup _{t\in I ^\ominus _{R/2}(t_0)} \Tail (v-c_0\psi ; R/2 ,x_0 ) \right) \\
        & \qquad  \leq c R^{-\eps}\left(  R^{-d-2s}\| v-c_0 \psi \| _{L^1 (Q ^\Omega _{R} (z_0))} + \sup _{t\in I ^\ominus _{R}(t_0)} \Tail (v-c_0\psi ; R/2 ,x_0 )\right) \\
        &\qquad  \leq c R^{-d-2s-\eps} \left(d^s_\Omega (x_0) \int _{Q^\Omega _R(z_0)}\left| \frac{v}{\psi} - c_0 \right| \intd z + \sup _{t\in I ^\ominus _{R}(t_0)} \Tail (v-c_0\psi ; R/2 ,x_0 ) \right),
    \end{align*}
    where we have used $|\psi | \leq cd_\Omega ^s \leq c d_\Omega (x_0 )^s$ in $B_R(x_0)$ (\autoref{Lem:psi-properties}(ii)).
    Furthermore, we get
    \begin{align*}
        [\psi ^{-1}]_{C^\eps_x (B_{R/4}(x_0))} & \leq c \| \psi ^{-1} \|_{L^\infty (B_{R/4}(x_0))} ^2 [\psi ]_{C^\eps_x (B_{R/4}(x_0))} \\
        & \leq c d_\Omega ^{-2s} (x_0) R^{-\eps} \left( \|\psi \| _{L^\infty (B_R(x_0))} + \Tail (\psi ;R,x_0) \right) \leq c R^{-\eps}d^{-s}_\Omega (x_0),
    \end{align*}
    where we have also applied \autoref{Lem:psi-properties}(iii).
    Combining the above estimates yields \eqref{eq:CampanatoEstimateTranslInvariantSupT} also in the case $d_\Omega(x_0) \geq R\geq 4\rho$.

    Next, let us assume $4\rho \leq d_\Omega(x_0) \leq R/4$. This case is essentially a combination of the two cases that have been treated before, but it requires a careful treatment of the nonlocal tail term. 
    Set $d_0:= d_\Omega (x_0)$.
    Then, we first use \eqref{eq:CampanatoEstimateTranslInvariantSupT} in the already proven case for $\rho =\rho$ and $R:= d_0$.
    \begin{align*}
        \sup _{t \in I_\rho ^\ominus(t_0)} \int _{B^\Omega_\rho (x_0)} & \left| \frac{v}{\psi} - \left( \frac{v}{\psi } \right) _{Q_\rho ^\Omega (z_0) } \right| \intd x  
            \leq c  \rho ^{-2s}\left( \frac{\rho}{d_0} \right) ^{d+2s+\eps}  \Bigg[ \int _{Q^\Omega _{d_0} (z_0)} \left| \frac{v}{\psi} - \left( \frac{v}{\psi} \right) _{Q_{d_0} ^\Omega (z_0) } \right| \intd z \\
            & + d_0^{d+s} \sup _{t \in I^\ominus _{d_0}(t_0)}\Tail (v-\psi (v/\psi )_{Q^\Omega _{d_0} (z_0)} ; d_0/4,x_0) \Bigg] =:I_1 +I_2 .
    \end{align*}
    Note that for $I_1$, we can use \eqref{eq:CampanatoEstimateTranslInvariantSupT} for $\rho =d_0$ and $R:=R$.
    Hence, it remains to show that
    \begin{equation} \label{eq:TailEstimated0ToR}
       \tilde{I}_2:= d_0^{d+s} \sup _{t \in I^\ominus _{d_0}(t_0)}\Tail (v-\psi (v/\psi )_{Q^\Omega _{d_0} (z_0)} ; d_0/4,x_0) \leq c\left( \frac{d_0}{R} \right) ^{d+2s+\eps} \overline{\Psi} (v;R) .
    \end{equation}
    Let $m\in \N$ be such that $4^{-m-1}R\leq d_0 \leq 4^{-m}R$.
    Then, by decomposing the tail dyadically,
    \begin{align*}
        \tilde{I}_2 & \leq c \frac{d_0^{d+2s+\eps}}{R^{s+\eps}} \left( \frac{d_0}{R} \right) ^{s-\eps}\sup _{t\in I^\ominus _R (t_0)} \Tail (v-\psi (v/\psi )_{Q^\Omega _R (z_0)} ; R/4,x_0) \\
        & \qquad + c \frac{d_0^{d+2s+\eps}}{R^{s+\eps}} \left( \frac{d_0}{R} \right) ^{s-\eps} \left| (v/\psi )_{Q^\Omega _{d_0} (z_0)} - (v/\psi )_{Q^\Omega _R (z_0)}\right| \Tail (\psi  ; R/4,x_0) \\
        & \qquad + c d_0 ^{d+2s+\eps}d_0^{s-\eps} \sum _{i=1}^{m} \left(  4^{-i}R\right) ^{-d-s} \sup _{t\in I^\ominus _{4^{-i}R}(t_0)}  \int _{B^\Omega_{4^{-i}R}(x_0) } | v/\psi - (v/\psi )_{Q^\Omega _{4^{-i}R} (z_0)}   | \intd y \\
        & \qquad + c d_0 ^{d+2s+\eps}d_0^{s-\eps} \sum _{i=1}^{m} \left(  4^{-i}R\right) ^{-s}  \left| (v/\psi )_{Q^\Omega _{4^{-i}R} (z_0)} - (v/\psi )_{Q^\Omega _{d_0} (z_0)} \right|  \\
        & =: J_1 +J_2+J_3+J_4 .
    \end{align*}
    Note that $J_1$ is bounded by $cd_0 ^{d+2s+\eps} R^{-d-2s-\eps}\overline{\Psi} (v, R)$ since $\eps <s$.
    To bound $J_2$ and $J_4$, we first observe that by applying \eqref{eq:CampanatoEstimateTranslInvariantSupT} on intermediate scales, we get for every $i\in \{ 0, 1, \dots ,m \}$
    \begin{equation}\label{eq:DifferencesAverages}
    \begin{aligned}
         \left| (v/\psi )_{Q^\Omega _{4^{-i}R} (z_0)} - (v/\psi )_{Q^\Omega _{d_0} (z_0)} \right| &  \leq \left| (v/\psi )_{Q^\Omega _{4^{-i}R} (z_0)} - (v/\psi )_{Q^\Omega _{4^{-m}R} (z_0)} \right| \\
         & \qquad + \left| (v/\psi )_{Q^\Omega _{4^{-m}R} (z_0)} - (v/\psi )_{Q^\Omega _{d_0} (z_0)} \right| \\
         &\leq c\sum _{l=0}^{m-i} \dashint  _{Q_{4^{l-m}R}(z_0)} \left| \frac{v}{\psi} - \left( \frac{v}{\psi}\right) _{Q_{4^{l-m}R}(z_0)} \right| \intd z \\
         & \leq c\sum _{l=0}^{m-i} \left( 4^{l-m}R \right) ^{-d-2s} \left( 4^{l-m} \right) ^{d+2s+\eps}  \overline{\Psi} (v;R) \\
         & \leq c R^{-d-2s} 4^{-i \eps} \overline{\Psi} (v;R) .
    \end{aligned}
    \end{equation}
    Hence, we obtain $J_2 \leq cd_0 ^{d+2s+\eps} R^{-d-2s-\eps}\overline{\Psi} (v; R)$ using $\eps <s$, \eqref{eq:DifferencesAverages}, and \autoref{Lem:psi-properties}(iii).
    For $J_4$, we get
    \begin{equation*}
        J_4 \leq c d_0 ^{d+2s+\eps} \left( \frac{d_0}{R}\right) ^{s-\eps} R^{-d-2s-\eps}\sum _{i=1}^m 4^{i(s-\eps)} \overline{\Psi} (v;R) \leq c d_0 ^{d+2s+\eps} R^{-d-2s-\eps} \overline{\Psi} (v;R) ,
    \end{equation*}
    where we have used that $d_0 / R \leq  4^{-m}$.
    For $J_3$, we apply again \eqref{eq:CampanatoEstimateTranslInvariantSupT} on intermediate scales, to obtain
    \begin{align*}
        J_3 \leq c d_0 ^{d+2s+\eps} d_0^{s-\eps} R^{-d-3s}\sum _{i=1}^m 4^{i(s-\eps)} \overline{\Psi} (v;R) \leq c d_0 ^{d+2s+\eps} R^{-d-2s-\eps} \overline{\Psi} (v;R) .
    \end{align*}
    This finishes the proof of \eqref{eq:TailEstimated0ToR}.

    Now, let us assume that $\rho \leq d_\Omega(x_0) \leq 4\rho \leq R/4$, then \eqref{eq:CampanatoEstimateTranslInvariantSupT} follows by first observing
    \begin{equation*}
        \sup _{t \in I_\rho ^\ominus(t_0)} \int _{B^\Omega _\rho (x_0)}  \left| \frac{v}{\psi} - \left( \frac{v}{\psi } \right) _{Q_\rho ^\Omega (z_0) } \right| \intd x  
             \leq c \sup _{t \in I_{4\rho} ^\ominus(t_0)} \int _{B^\Omega _{4\rho} (x_0)}  \left| \frac{v}{\psi} - \left( \frac{v}{\psi } \right) _{Q_{4\rho} ^\Omega (z_0) } \right| \intd x  
    \end{equation*}
    and then using \eqref{eq:CampanatoEstimateTranslInvariantSupT} for $\rho := 4\rho$ and $R:=R$.

    Finally, assume $4\rho \leq R/4 \leq d_\Omega (x_0) \leq R$.
    Then, following the proof of case 2 above (see \eqref{eq:CampanatoCaseD0geqR}) with $\rho := \rho$ and $R:= R/4$, we get
    \begin{align*}
        &\sup _{t \in I_\rho ^\ominus(t_0)} \int _{B_\rho (x_0)}  \left| \frac{v}{\psi} - \left( \frac{v}{\psi } \right) _{Q_\rho ^\Omega (z_0) } \right| \intd x   \\
        & \qquad \leq c\rho ^{d+\eps} R^{-\eps} d_\Omega^{-s} (x_0) \left( \| v-c_0 \psi \| _{L^\infty (Q ^\Omega _{R/8} (z_0))} + \sup _{t\in I ^\ominus _{R/8}(t_0)} \Tail (v-c_0\psi ; R/8 ,x_0 ) \right) \\
        & \qquad \leq c \rho ^{d+\eps} R^{-\eps-s}  \left( \| v-c_0 \psi \| _{L^\infty (Q ^\Omega _{R/2} (z_0))} + \sup _{t\in I ^\ominus _{R}(t_0)} \Tail (v-c_0\psi ; R/4 ,x_0 ) \right)  .
    \end{align*}
    Now, the proof follows exactly as above in case 2, by applying local boundedness and using $\| \psi\| _{L^\infty (Q^\Omega _R (z_0))} \leq c R^s$.
\end{proof}

Next, we prove a Campanato-type estimate for solutions to translation invariant equations involving the parabolic nonlocal tail terms. Note that this procedure is slightly different from the one in \cite{Kim2024}, where the Campanato-type estimate was directly established for solutions to equations in divergence form. Here, we take an alternative route, since one has to use the equation once again in order to control the parabolic tail terms (see the proof of \autoref{Lem:CampanatoEstimateTranslInvariant}).
Furthermore, note that in comparison to \cite{Kim2024}, the excess functional $\overline{\Psi}$ (and also $\Psi$ in \eqref{eq:ExcessFunctional-higher}) is not equipped with the modified $\Tail _{\sigma ,B_1}$ but only takes into account the usual $\Tail$.
This is sufficient since in \eqref{eq:ExcessFunctionalEstimate-higher}, the $C^\sigma$ regularity of $K$ is still captured in $\Phi _{\sigma}(u;R)$.

\begin{Lemma}
\label{Lem:CampanatoEstimateTranslInvariant-2}
    Let $\alpha \in (0, s)$ and $\eps \in (0, s)$.
     Let $\Omega \subset \R^d$ be a $C^{1,\alpha }$ domain with $0\in \partial \Omega$.
    Furthermore, let $L$ be a translation invariant operator of the form \eqref{eq:TranslationalInvariantOperator}-\eqref{eq:TranslationalInvariantKernel}.
    Given some $z_0=(t_0,x_0) \in Q^\Omega_{1/8}$ and $R \in (0,1/8)$, assume that $v$ solves
    \begin{equation*}
        \left\{ \begin{aligned}
          \partial _t v -Lv & =0 & &\text{in} \quad I^\ominus _R(t_0) \times (B_R(x_0)\cap \Omega ) , \\
           v & =0 & & \text{in} \quad  I^\ominus _R(t_0) \times ( B_R(x_0) \setminus \Omega ) .
        \end{aligned} \right.
    \end{equation*}
    Then, there exists some constant $c=c(d,s, \lambda , \Lambda ,\alpha , \eps , \Omega )>0$ such that
    \begin{equation} \label{eq:CampanatoEstimateTranslInvariant}
        \begin{aligned}
             \overline{\Psi}(v;\rho ,z_0) \le c \left( \frac{\rho}{R} \right)^{d + 2s + \eps}  \overline{\Psi } (v;R,z_0)  
        \end{aligned}
    \end{equation}
    holds for all $0<\rho \leq R/16$, where $\overline{\Psi} (v;R,z_0)$ is defined in \eqref{eq:ExcessFunctional-higher-Bar}.
\end{Lemma}
\begin{proof}
    First, note that \autoref{Lem:CampanatoEstimateTranslInvariant} implies that we can bound the first summand of $\overline{\Psi}(v;\rho)$ by $c (\rho /R)^{d + 2s + \eps}  \overline{\Psi} (v;R)$.
    Hence, it remains to show that
    \begin{equation}\label{eq:TailEstimateCampanato}
        \begin{aligned}
        &\max \{ \rho ,d_\Omega (x_0)\} ^{-s} \rho ^{d+2s} \sup_{t\in I^\ominus _{\rho} (t_0)} \Tail  (v - \psi(v/\psi)_{Q_{\rho}^{\Omega}(z_0)};\rho /4  ,x_0) \leq c (\rho /R)^{d + 2s + \eps}  \overline{\Psi} (v;R,z_0) .
    \end{aligned}
    \end{equation}
    If we assume that $d_\Omega(x_0) \leq \rho$, then \eqref{eq:TailEstimateCampanato} follows exactly as in the proof of \eqref{eq:TailEstimated0ToR}.
    The other cases follow by the same modifications as in the proof of \autoref{Lem:ExcessFunctionalEstimate}.
\end{proof}

Let $\theta = 128$.
Assume that $-\mathcal{L}_t$ is an operator of the form \eqref{eq:OperatorDivergenceForm}-\eqref{eq:KernelDivergenceForm}.
Then, for some point $z_0 = (t_0,x_0) \in Q_1^\Omega$, we define $\psi _{z_0}$ as the function $\psi$ from \eqref{eq:EllipticBarrierUBar} with respect to the frozen operator $L_{z_0}$ (see \eqref{eq:FrozenKernel}).
Furthermore, for a given $\rho > 0$, we define the higher order nonlocal excess functional by
\begin{align}
\label{eq:ExcessFunctional-higher}
\begin{split}
    \Psi(u;\rho,z_0) &:= \iint_{Q_{\rho}^{\Omega}(z_0)} \left| \frac{u}{\psi_{z_0}} - \left( \frac{u}{\psi_{z_0}} \right)_{Q_{\rho}^{\Omega}(z_0)} \right| \d x \d t \\
    &\quad + \max\{ \rho , d_{\Omega}(x_0) \}^{-s} \rho^{d+2s} \sup_{t \in I_{\rho}^{\ominus}(t_0)} \Tail \left(u - \psi_{z_0} \left( \frac{u}{\psi_{z_0}} \right)_{Q_{\rho}^{\Omega}(z_0)} ; \theta^{-1} \rho , x_0 \right) .
    \end{split}
\end{align}
\begin{Lemma}
\label{Lem:CampanatoEstimate}
    Assume that we are in the setting of \autoref{Thm:CsBoundaryRegularityGlobalAAssumption}.
    Let $\eps \in (0,s )$.
    Then, for every $m\in \N$, $0<R \leq 1/16$, and $z_0 = (t_0,x_0) \in Q^\Omega_{1/8}$, we have
    \begin{equation}\label{eq:ExcessFunctionalEstimate-higher}
        \begin{aligned}
            \Psi(u;\theta^{-m} R,z_0) &\le c (\theta^{-m})^{d + 2s + \eps} \Psi (u; R,z_0) + c R^{\sigma} \Phi _{\sigma}(u;R,z_0) \\
        &\quad + c R^{d+3s - \frac{d}{r} - \frac{2s}{q}} \left( \| u \| _{L^\infty ((-1,0)\times \R ^d)} + \| f\| _{L_{t}^q L_x^{r} (Q^\Omega _1)} \right) 
        \end{aligned}
    \end{equation}
    for some constant $c=c(d,s,\lambda,\Lambda ,\alpha,q,r,\sigma ,\Omega ,\eps )>0$ where $\Psi(u; R,z_0)$ is defined in \eqref{eq:ExcessFunctional-higher}.
    \end{Lemma}
    \begin{proof}
         We write $\psi := \psi_{z_0}$ and set
\begin{equation*}
    B= \| u \| _{L^\infty ((-1,0)\times \R ^d)} + \| f\| _{L_{t}^q L_x^{r} (Q^\Omega _1)} .
\end{equation*}
We fix $0<R\leq 1/16$.
Let $v:=v_{R/8,z_0}$ and $w:=w_{R/8,z_0}$ be defined as in \eqref{eq:DefinitionV} and \eqref{eq:DefinitionW}.

First, note that we have
\begin{align*}
     &\Psi (u;\theta^{-m}R)  \leq \Psi(w;\theta ^{-m}R)+ c\overline{\Psi}(v;\theta ^{-m}R) \\
     & \qquad\qquad + c\max \{ d_\Omega(x_0),\theta ^{-m}R  \}^{-s} (\theta ^{-m}R) ^{2s} \sup _{t\in I^\ominus _{\theta^{-m}R}(t_0)} \int _{B ^\Omega_{\theta ^{-m}R}(x_0)} | v - \psi (v/\psi )_{Q^\Omega _{\theta ^{-m}R}(z_0)} | \intd x \\
     & \qquad\leq \Psi(w;\theta ^{-m}R)+ \overline{\Psi}(v;\theta ^{-m}R) \\
     & \qquad \qquad + c (\theta ^{-m}R) ^{2s} \sup _{t\in I^\ominus _{\theta^{-m}R}(t_0)} \int _{B ^\Omega_{\theta ^{-m}R}(x_0)} \left| \frac{v}{\psi} - \left(\frac{v}{\psi } \right) _{Q^\Omega _{\theta ^{-m}R}(z_0)} \right| \intd x .
\end{align*}
Hence, by \autoref{Lem:CampanatoEstimateTranslInvariant} and \autoref{Lem:CampanatoEstimateTranslInvariant-2}, it holds
\begin{align}
\label{eq:Campanato-help-1}
\begin{split}
    & \Psi (u;\theta^{-m}R)  \le  c (\theta^{-m})^{d + 2s + \eps} \Bigg[ \int _{Q^\Omega _{R/8}} \left| \frac{v}{\psi} - \left( \frac{v}{\psi} \right) _{Q^\Omega _{R/8}} \right| \intd z \\
    & \qquad \qquad\qquad+ \max\{ R , d_{\Omega}(x_0) \}^{-s} R^{d+2s} \sup_{t \in I_{R/8}^{\ominus}(t_0)} \Tail \left(v - \psi \left( v /\psi \right)_{Q_{R/8}^{\Omega}(z_0)} ;  R/32, x_0 \right)  \Bigg]  \\
             & \quad + 2 \| w/ \psi \| _{L^1(Q_{\theta ^{-m}R}^{\Omega}(z_0))}  \\
            &\quad + \max\{ \theta^{-m}R , d_{\Omega}(x_0) \}^{-s} (\theta^{-m}R)^{d+2s} \sup_{t \in I_{\theta ^{-m}R}^{\ominus}(t_0)} \Tail \left(w - \psi \left( w /\psi \right)_{Q_{\theta ^{-m}R}^{\Omega}(z_0)} ; \theta^{-m-1} R , x_0 \right).
            \end{split}
\end{align}
Now, we observe that since $w = 0$ in $I^\ominus_{R/8}(t_0) \times (\R^d \setminus B_{R/8}(x_0))$ and by \autoref{Lem:psi-properties}(ii) and (iii), 
\begin{align}
\label{eq:Campanato-help-2}
\begin{split}
    &\max\{ \theta^{-m}R , d_{\Omega}(x_0) \}^{-s} (\theta^{-m}R)^{d+2s} \sup_{t \in I_{\theta ^{-m}R}^{\ominus}(t_0)} \Tail \left(w - \psi \left( w /\psi \right)_{Q_{\theta ^{-m}R}^{\Omega}(z_0)} ; \theta^{-m-1} R , x_0 \right) \\
    &\quad\le c \max\{ \theta^{-m}R , d_{\Omega}(x_0) \}^{-s} (\theta^{-m}R)^{2s} \sup_{t \in I_{\theta^{-m}R}^{\ominus}(t_0)} \int_{B_{R/8}(x_0)} |w| \d x + c \left\Vert \frac{w}{\psi} \right\Vert_{L^1(Q_{\theta^{-m}R}^{\Omega}(z_0))} \\
    &\quad\le c \max\{ R , d_{\Omega}(x_0) \}^{-s} R^{2s + \frac{d}{2}} \Vert w \Vert_{L^{\infty}_t L^2_x(Q_{R/8}(z_0))} + c \left\Vert \frac{w}{d_\Omega ^s} \right\Vert_{L^1(Q_{R/8}^{\Omega}(z_0))}.
    \end{split}
\end{align}

Moreover, we recall that $u=v+w$ and $\theta ^{-1} R \leq R/32$, and hence,
\begin{align*}
    &\max \{ R,d_\Omega (x_0) \} ^{-s} R^{d+2s} \sup _{t \in I^\ominus _{R/8}(t_0)}\Tail (v-\psi (v/\psi)_{Q^\Omega _{R/8} (z_0)} ; R/32,x_0) \\
    &\quad \le c \max \{ R,d_\Omega (x_0)\} ^{-s}  R^{d+2s} \Bigg( \sup _{t \in I^\ominus _{R/8}(t_0)} \Tail (u - \psi(u/\psi)_{Q_{R/8}^{\Omega}(z_0)};\theta^{-1} R,x_0) \\
    &\qquad\qquad\qquad\qquad\qquad\qquad\qquad\qquad\quad +\sup _{t \in I^\ominus _{R/8}(t_0)} \Tail (w - \psi(w/\psi)_{Q_{R/8}^{\Omega}(z_0)};\theta^{-1} R,x_0) \Bigg) \\
    &= I_1 + I_2.
\end{align*}
We estimate $I_1$ and $I_2$ separately. For $I_2$, we use that $w = 0$ in $I_{R/8}^{\ominus}(t_0) \times (\R^d \setminus B_{R/8}(x_0))$ and \autoref{Lem:psi-properties}(ii) and (iii) to obtain
\begin{align*}
    I_2 &\le c\max \{ R,d_\Omega (x_0) \} ^{-s}  \left( R^{2s} \sup _{t \in I^\ominus _{R/8}(t_0)} \int_{B_{R/8}(x_0)} |w| \d x +   \left\| \frac{w}{d_\Omega ^s} \right\| _{L^1 (Q^\Omega _{R/8}(z_0))}  \Tail(\psi; \theta^{-1}R,x_0) \right) \\
    &\le c \max \{ R,d_\Omega (x_0)\} ^{-s} R^{2s+\frac{d}{2}} \| w \| _{L^\infty_t L^2_x (Q_{R/8}(z_0))} + c\left\| \frac{w}{d_\Omega ^s} \right\| _{L^1 (Q^\Omega_{R/8}(z_0))}.
\end{align*}

Moreover, for $I_1$ we have by \autoref{Lem:psi-properties}(iii), 
\begin{align*}
    I_1 &\le c \max \{ R,d_\Omega (x_0) \} ^{-s} R^{d+2s} \Bigg( \sup _{t \in I^\ominus _{R/8}(t_0)}\Tail (u-\psi (u/\psi)_{Q^\Omega _{R} (z_0)} ; \theta^{-1} R,x_0) \\
    & \qquad\qquad\qquad\qquad\qquad\qquad\qquad\qquad + \Big| (u/\psi)_{Q^\Omega _{R/8} (z_0)} - (u/\psi)_{Q^\Omega _{R} (z_0)}\Big| \Tail(\psi;\theta^{-1}R,x_0) \Bigg) \\
    & \leq c \Psi (u;R) + R^{d+2s }\Big| (u/\psi)_{Q^\Omega _{R/8} (z_0)} - (u/\psi)_{Q^\Omega _{R} (z_0)}\Big|\\
    & \leq c \Psi (u;R) .
\end{align*}

This implies
\begin{align}
\label{eq:Campanato-help-3}
\begin{split}
    &\max \{ R,d_\Omega (x_0) \} ^{-s} R^{d+2s} \sup _{t \in I^\ominus _{R/8}(t_0)}\Tail (v-\psi (v/\psi)_{Q^\Omega _{R/8} (z_0)} ; R/16,x_0) \\
    &\qquad \le  c\Psi(u,R) + c\max \{ R,d_\Omega (x_0)\} ^{-s} R^{2s+\frac{d}{2}} \| w \| _{L^\infty_t L^2_x (Q_{R/8}(z_0))} + c\left\| \frac{w}{d_\Omega ^s} \right\| _{L^1 (Q_{R/8}(z_0))}.
    \end{split}
\end{align}

Finally, we have
\begin{equation}\label{eq:Campanato-help-4}
\begin{aligned}
    &\int _{Q^\Omega _{R/8}} \left| \frac{v}{\psi} - \left( \frac{v}{\psi} \right) _{Q^\Omega _{R/8}} \right| \intd z  \leq \int _{Q^\Omega _{R/8}} \left| \frac{u}{\psi} - \left( \frac{u}{\psi} \right) _{Q^\Omega _{R/8}} \right| \intd z +\int _{Q^\Omega _{R/8}} \left| \frac{w}{\psi} - \left( \frac{w}{\psi} \right) _{Q^\Omega _{R/8}} \right| \intd z\\
    &\qquad \leq c \Psi (u;R) + c R^{d+2s} \Big| (u/\psi)_{Q^\Omega _{R/8} (z_0)} - (u/\psi)_{Q^\Omega _{R} (z_0)}\Big|
   +c \left\| \frac{w}{d_\Omega ^s} \right\| _{L^1(Q^\Omega_{R/8}(z_0))} \\
   & \qquad\leq c \Psi (u;R) +c \left\| \frac{w}{d_\Omega ^s} \right\| _{L^1(Q^\Omega_{R/8}(z_0))}.
\end{aligned}
\end{equation}
Altogether, combining \eqref{eq:Campanato-help-1}, \eqref{eq:Campanato-help-2}, \eqref{eq:Campanato-help-3}, and \eqref{eq:Campanato-help-4}, we deduce
\begin{align*}
    \Psi(u;\theta^{-m}R) &\le c (\theta^{-m})^{d + 2s + \eps} \Psi(u; R) \\
    &\quad + c \max \{ R,d_\Omega (x_0)\} ^{-s} R^{2s+\frac{d}{2}} \| w \| _{L^\infty_t L^2_x (Q_{R/8}(z_0))} + c\left\| \frac{w}{d_\Omega ^s} \right\| _{L^1 (Q_{R/8}(z_0))}.
\end{align*}
By applying \autoref{Lem:FreezingEstimateIntW/ds}, this concludes the proof.
    \end{proof}

\begin{Lemma}\label{Lem:uOverPsiHoelderContinuous}
    Assume that we are in the same setting as in \autoref{Thm:CsBoundaryRegularityGlobalAAssumption}.
    Let $\beta$ be a positive constant satisfying $\beta < \sigma <s$ and $\beta \leq s-\frac{2s}{q}-\frac{d}{r}$.
    Then,
    \begin{equation}\label{eq:uOverPsiHoelderContinuous}
        \left[ \frac{u(t,x)}{\psi _{(t,x)}(x)}\right] _{C^{ \beta}_{p}( \overline{Q^\Omega_{1/8}})} \leq c  \left( \| u \| _{L^\infty ((-1,0)\times \R ^d)} + \| f\| _{L_{t}^q L_x^{r} (Q^\Omega _1)} \right)
    \end{equation}
    for some constant $c=c(d,s,\lambda , \Lambda ,\alpha ,\sigma , q,r,\beta , \Omega)>0$.
\end{Lemma}
\begin{proof}
    Assume, without loss of generality, that $\| u \| _{L^\infty ((-1,0)\times \R ^d)} + \| f\| _{L_{t}^q L_x^{r} (Q^\Omega _1)} \leq 1$.

    \textbf{Step 1:}  
    We set $\overline{\eps}:= \sigma - \beta >0$.
    Let $R_0>0$ be the constant from \autoref{Lem:ExcessFunctionalEstimate} (with $\eps := \overline{\eps}  $ in \autoref{Lem:ExcessFunctionalEstimate}).
    First, we claim that for every $z_0\in Q^\Omega _{1/8}$ and for every $m\in \N$, it holds that
    \begin{equation} \label{eq:DecayEstimateForPsiU}
        \Psi (u;\theta ^{-m}R_0,z_0) \leq c (\theta^{-m}) ^{d+2s+\beta  } .
    \end{equation}
    
    Let $0<  R\leq R_0$ and $m\in \N$.
    Choose $k\in \N$ and $\mu \in (\theta ^{-1},1]$ such that $R=\theta ^{-k}\mu R_0$.
    Then, by applying \autoref{Lem:CampanatoEstimate} (with $\eps :=(\sigma +s)/2$) and \autoref{Lem:ExcessFunctionalEstimate} (with $\overline{\eps}=\sigma -\beta$, $\theta^{-m}R:=\theta^{-k}\mu R_0$, and $R=\mu R_0$), we obtain
    \begin{align*}
        & \Psi (u;\theta ^{-m}R) \leq  c (\theta^{-m})^{d + 2s + \eps} \Psi(u; R) + c R^{\sigma} \Phi_{\sigma}(u;R) + c R^{d+3s - \frac{d}{r} - \frac{2s}{q}} \\
        & \qquad \leq c (\theta^{-m})^{d + 2s + \eps} \Psi(u; R) + c R^{\sigma }\left( \theta ^{-k} \right) ^{d+2s -\overline{\eps}} \Phi_\sigma (u;\mu R_0) 
         + c  R^{\sigma+d+2s-\overline{\eps}}+ c R^{d+3s - \frac{d}{r} - \frac{2s}{q}} .
    \end{align*}
    Note that as in the proof of \autoref{Thm:Cs-epsRegWithCoefficients}, we have $\Phi_\sigma(u; \mu R_0) \leq c$ for some constant $c$ independent of $\mu$.
    Hence, we have
    \begin{equation*}
        \Psi (u;\theta ^{-m}R) \leq c (\theta^{-m})^{d + 2s + \frac{\sigma +s}{2}} \Psi(u; R) + c R^{d+2s+ \beta }  .
    \end{equation*}
     By applying the iteration lemma (\autoref{Lem:IterationLemmaOnDyadicScale}), we get    
    \begin{equation}\label{eq:GeneralDecayEstimateForPsiU}
            \Psi (u;\theta ^{-m}R) \leq c (\theta^{-m}) ^{d+2s+ \beta } \Psi (u;R) + c(\theta ^{-m}R)^{d+2s+ \beta }  
    \end{equation}
    for every $m\in \N$ and $0<R\leq R_0$.
    Note that also $\Psi (u;R_0) \leq c$.
    Consequently, \eqref{eq:DecayEstimateForPsiU} follows by setting $R=R_0$ in \eqref{eq:GeneralDecayEstimateForPsiU}.

    \textbf{Step 2:}
    In a second step we will show how \eqref{eq:DecayEstimateForPsiU} implies \eqref{eq:uOverPsiHoelderContinuous}.
    Let $0<\rho \leq R_0$ and $z_0 \in Q^\Omega_{1/8}$.
    Then, using the identity $(ab-cd)=\frac{1}{2}(a-c)(b+d)+\frac{1}{2}(a+c)(b-d)$, we get
    \begin{align*}
        &\iint _{Q^\Omega _\rho (z_0)} \left| \frac{u(t,x)}{\psi_{(t,x)}(x)} - \left( \frac{u}{\psi} \right) _{Q^\Omega_\rho (z_0)} \right| \intd x \intd t \leq
         2 \iint _{Q^\Omega _\rho (z_0)} \left| \frac{u(t,x)}{\psi_{z_0}(x)} \frac{\psi _{z_0}(x)}{\psi _{(t,x)}(x)} - \left( \frac{u}{\psi _{z_0}} \right) _{Q^\Omega_\rho (z_0)} \right| \intd x \intd t \\
         & \qquad \leq \iint _{Q^\Omega _\rho (z_0)} \left| \frac{u}{\psi_{z_0}} - \left( \frac{u}{\psi _{z_0}} \right) _{Q^\Omega_\rho (z_0)} \right| \left| \frac{\psi _{z_0}}{\psi} +1\right| \intd x \intd t \\
         & \qquad \qquad \qquad \qquad+
         \iint _{Q^\Omega _\rho (z_0)} \left| \frac{u}{\psi_{z_0}} + \left( \frac{u}{\psi _{z_0}} \right) _{Q^\Omega_\rho (z_0)} \right| \left| \frac{\psi _{z_0}}{\psi} -1\right| \intd x\intd t \\
         & \qquad = I_1 +I_2 .
    \end{align*}
    Choose $m\in \N$ such that $\theta ^{-m-1}R_0 \leq \rho \leq \theta ^{-m}R_0$.
    For $I_1$, we use that $\psi _{z_0}(x) \asymp \psi _{(t,x)}(x) \asymp d^s_\Omega (x)$ in $Q^\Omega _\rho (x_0)$, and \eqref{eq:DecayEstimateForPsiU}, to obtain
    \begin{align*}
        I_1 & \leq c \left\| \frac{\psi _{z_0}}{\psi} +1 \right\| _{L^\infty (Q^\Omega _\rho (z_0))} \iint _{Q^\Omega _{\theta ^{-m}R_0} (z_0)} \left| \frac{u}{\psi_{z_0}} - \left( \frac{u}{\psi _{z_0}} \right) _{Q^\Omega_{\theta ^{-m}R_0} (z_0)} \right|\intd x \intd t \\
        &\leq c (\theta ^{-m} )^{d+2s+\beta} 
        \leq c \rho ^{d+2s+\beta} .
    \end{align*}
    Next, we choose some $\tilde{\eps} >0 $ with $\tilde{\eps}<1-s$ and $\tilde{\eps} \leq (\sigma - \beta )/2$.
    Using that $|u|/\psi _{z_0} \leq c|u|/d_\Omega ^s \leq cd_\Omega ^{-\tilde{\eps}}$ in $Q^\Omega _{1/2}$ by \autoref{Lem:psi-properties}(ii) and \autoref{Thm:Cs-epsRegWithCoefficients}, we obtain
    \begin{align*}
        I_2 &\leq c \iint _{Q^\Omega_\rho (z_0)} \left[ \psi ^{-1} \left| \frac{u}{\psi _{z_0}} \right| + \psi ^{-1} \left( \frac{|u|}{\psi _{z_0}} \right) _{Q^\Omega_\rho (z_0)} \right] \left| \psi _{z_0} -\psi \right|\intd x\intd t  \\
        & \leq  c \rho^{2s} \left[ \int_{B^\Omega_\rho (x_0)}  d_\Omega ^{-s-\tilde{\eps}}\intd x + \rho ^{-d}\int_{B^\Omega_\rho (x_0)}  d_\Omega ^{-s} \intd x  \int_{B^\Omega_\rho (x_0)}  d_\Omega ^{-\tilde{\eps}}\intd x \right]  \sup _{ (t,x) \in Q^\Omega _\rho (x_0)} | \psi _{z_0}(x) - \psi _{(t,x)}(x) | \\
        & \leq c  \rho ^{d+2s}\max \{  \rho ,d_\Omega (x_0) \} ^{-s-\tilde{\eps} }\sup _{w_0\in Q_\rho ^\Omega(z_0)}\sup _{ x \in B^\Omega _\rho (x_0)} | \psi _{z_0}(x) - \psi _{w_0}(x) |,
    \end{align*}
    where in the last step we have distinguished between the cases $d_\Omega (x_0)\geq 2 \rho$ and $d_\Omega (x_0)\leq 2 \rho$ and used \cite[Lemma 2.5]{Kim2024} in the latter case.
    Using \autoref{Lem:ClosenessOfBarriers} and the assumption \eqref{eq:KernelHoelderCont} with $\mathcal{A}=(-1,0)\times \R ^d$, we conclude
    \begin{align*}
        I_2 & \leq c \rho ^{d+2s}\max \{  \rho ,d_\Omega (x_0) \} ^{-s-\tilde{\eps} } \max \{  \rho ,d_\Omega (x_0) \} ^{s-\tilde{\eps} } \sup _{w_0\in Q_\rho ^\Omega (x_0)} [\psi _{z_0} - \psi _{w_0}] _{C^{s-\tilde{\eps}_x} (B^\Omega _{1/2})} \\
        &\leq  c \rho ^{d+2s+\sigma} \max \{ \rho, d_\Omega (x_0)\} ^{-2\tilde{\eps}} \leq c \rho ^{d+2s+\sigma -2\tilde{\eps}}  \leq c \rho ^{d+2s+\beta }.
    \end{align*}
    Altogether, we obtain
    \begin{align*}
        \iint _{Q^\Omega _\rho (z_0)} \left| \frac{u}{\psi} - \left( \frac{u}{\psi} \right) _{Q^\Omega_\rho (z_0)} \right| \intd x \intd t \leq c\rho ^{d+2s+\beta} .
    \end{align*}
    By Campanato's embedding, we conclude \eqref{eq:uOverPsiHoelderContinuous}.
\end{proof}

\begin{proof}[Proof of \autoref{Thm:CsBoundaryRegularityGlobalAAssumption}]
    Without loss of generality, assume $\| u \| _{L^\infty ((-1,0)\times \R ^d)} + \| f\| _{L_{t}^q L_x^{r} (Q^\Omega _1)} \leq 1$.
    By \autoref{Lem:uOverPsiHoelderContinuous}, we have $[u/\psi] _{C^{ \beta}_{p} (Q^\Omega _{1/8})}\leq c$ for some small $\beta>0$.
    Fix some $z_1=(t_1,x_1) \in Q^\Omega _{1/8}$ with $d_\Omega (x_1)\geq c>0$.
    Then, for every $z_0=(t_0,x_0) \in Q^\Omega _{1/8} $, using \autoref{Lem:psi-properties}(ii) and local boundedness (\autoref{Prop:LocalBoundedness}), we have
    \begin{align*}
        \left| \frac{u(z_0)}{\psi _{z_0}(x_0)} \right| \leq  \left| \frac{u(z_0)}{\psi _{z_0}(x_0)} - \frac{u(z_1)}{\psi_{z_1}(x_1)} \right| + \left| \frac{u(z_1)}{\psi_{z_1}(x_1)} \right| \leq c (|x_0-x_1|^\beta +|t_0-t_1|^\frac{\beta}{2s}) +c \frac{|u(z_1)|}{d_\Omega^s(x_1)} \leq c .
    \end{align*}
    Hence, we have shown that $|u| \leq c d_\Omega^s$ in $Q^\Omega _{1/8}$.
    Combining this estimate with interior regularity (see step 2 in the proof of \autoref{Lem:BoundaryRegFromBoundOnU}) yields the $C^{s}_{p}$ regularity of $u$.
\end{proof}
We end this section with the following counterexamples.

\begin{Example}
\label{example:counterex-1}
Let $f\colon (0,1) \times B_1 \to \R $ be such that
\begin{equation} \label{eq:f-counterex-1}
    f(t,x) \asymp(1/2 - t)^{-1/2} |\log ( 1/2 - t)|^{-\alpha} 
\end{equation}
for some $\alpha \in (1/2,1)$.
Note that $f \in L^2_tL^\infty_x$ since $\alpha>1/2$.
Now, let $u$ be the solution to
\begin{equation} \label{eq:Equation-counterex-1}
    \left\{ \begin{aligned}
           \partial_t u + (-\Delta)^s u & =f & &\text{in} \quad (0,1) \times B_1 , \\
           u & =0 & & \text{in} \quad  (0,1) \times B^c_1 ,\\
           u & =0 & & \text{in} \quad  \{ 0 \} \times B_1 .
        \end{aligned} \right.
\end{equation}
We will show that $u \not \in L^{\infty}_tC^s_x((1/4,3/4) \times \overline{B_1})$.

 Let us fix $t = 1/2$ and consider $x \in B_1$. We denote $d(x):= d_{B_1}(x)$.

We claim that for every $\tau \in (0,t)$ with $d(x) < (t-\tau)^{1/2s}/2$, it holds
\begin{align}
\label{eq:heat-kernel-integral}
    \int_{B_1} p_{B_1}(t-\tau,x,y) \d y \ge c (1 \wedge d^s(x) (t-\tau)^{-\frac{1}{2}}) = c d^s(x) (t-\tau)^{-\frac{1}{2}}.
\end{align}

To prove \eqref{eq:heat-kernel-integral}, let $z = x - \frac{(t-\tau)^{1/2s}}{2} \frac{x}{|x|}$. Then, it holds $d(z) = d(x) + \frac{(t-\tau)^{1/2s}}{2}$. We define $B = B_{(t-\tau)^{1/2s}/4}(z)$ and observe that for any $y \in B$ it holds
\begin{align*}
    d(y) (t-\tau)^{-1/2s} \ge c, \qquad |x-y|(t-\tau)^{-1/2s} \le (|x-z| + |z-y|) (t-\tau)^{-1/2s} \le c .
\end{align*}
Hence,
\begin{align*}
    \int_{B_1} & p_{B_1}(t-\tau,x,y) \d y \\
    &\ge c d^s(x)(t-\tau)^{-\frac{1}{2}} \int_{B_1} (1 \wedge d^s(y) (t-\tau)^{-\frac{1}{2}}) \left( (t-\tau)^{-\frac{d}{2s}}  \wedge \frac{t-\tau}{|x-y|^{d+2s}} \right) \d y \\
    &\ge c d^s(x)(t-\tau)^{-\frac{1}{2}} \int_B (t-\tau)^{-d/2s} \d y \ge c d^s(x)(t-\tau)^{-\frac{1}{2}} ,
\end{align*}
as desired.

Then, by the representation formula, and using \eqref{eq:heat-kernel-integral}, which holds for $\tau < t - (2d(x))^{2s}$,
\begin{align*}
    u(t,x)/d^s(x) &= d^{-s}(x) \int_0^t  \int_{B_1} p_{B_1}(t-\tau,x,y) f(\tau ,y ) \d y \d \tau\\
    &\ge c \int_0^{1/2 - (2d(x))^{2s}} (1/2-\tau)^{-1} |\log (1/2 - \tau)|^{-\alpha} \d \tau \\
    &\ge c  |\log [(2d(x))^{2s}]|^{1-\alpha} \ge c |\log d(x)|^{1-\alpha} \to +\infty,
\end{align*}
as $x \to \partial B_1$, since $\alpha < 1$. Hence, $u \not \in L^{\infty}_tC^s_x((1/4,3/4) \times \overline{B_1})$.

Note that by a standard Campanato iteration one can easily show that $u \in C^{s}_{p}((1/4,3/4) \times B_{1/2})$ (see \autoref{Prop:InteriorRegularity}), so the boundary regularity differs from the interior regularity when considering source terms in the borderline space $L^{2}_t L^{\infty}_{x}$.

Furthermore, note that by following the arguments in \cite{Byun2025}, one could prove $u\in C^{s}_{p}((1/4,3/4) \times\overline{ B_1})$ if $f \in L^{2,1}_t L^{\infty}_x$.
\end{Example}

\begin{Example} \label{example:counterex-2}
    Let $u$ be the solution to 
    \begin{equation*}
    \left\{ \begin{aligned}
           \partial_t u + (-\Delta)^s u & =0 & &\text{in} \quad (0,1) \times B_1 , \\
           u (t,x)& = (1/2 - t)^{-1/2} |\log ( 1/2 - t)|^{-\alpha}\ind _{B} (x)  & & \text{in} \quad  (0,1) \times B^c_1 ,\\
           u & =0 & & \text{in} \quad  \{ 0 \} \times B_1 .
        \end{aligned} \right.
\end{equation*}
where $B:= B_1(3e_d)$ and $\alpha \in (1/2,1)$ is a constant.
Note that $u$ has a tail which is merely in $L^2_t$, i.e. $u \in L^2((0,1) ; L^1_{2s}(\R^d)) \setminus L^{2+\eps}((0,1) ; L^1_{2s}(\R^d))$ for every $\eps >0$.

Let $\overline{u}=u \ind _{(0,1)\times B_1}$. Then, $\overline{u}$ solves \eqref{eq:Equation-counterex-1} with some right-hand side $f$ which satisfies \eqref{eq:f-counterex-1}.
Hence, by \autoref{example:counterex-1}, we get that $u \not \in L^{\infty}_tC^s_x((1/4,3/4) \times \overline{B_1})$.
\end{Example}

\section{Hopf lemma}\label{sec:HopfLemma}
The goal of this section is to prove the Hopf lemma from \autoref{Thm:CsBoundaryRegAndHopfLemma}(ii).
Before we start with the proof of \autoref{Thm:CsBoundaryRegAndHopfLemma}(ii) in \autoref{sec:SubsectionProofHopf}, we show in \autoref{sec:SubsectionClosenessSolForHopf} that certain solutions for the divergence form operator $\mathcal{L}_t$ are close to solutions with respect to the frozen translation invariant operator $L_0$ (see \autoref{Prop:ClosenessOfSolutionsForHopf}).
\subsection{Closeness of solutions}\label{sec:SubsectionClosenessSolForHopf}
\begin{Proposition} \label{Prop:ClosenessOfSolutionsForHopf}
    Let $\alpha, \sigma \in (0,s)$, $\eps \in (0,1)$, and let $\Omega \subset \R^d$ be a $C^{1,\alpha }$ domain with $0\in \partial \Omega$.
    Furthermore, let $\mathcal{L}_t$ be an operator of the form \eqref{eq:OperatorDivergenceForm}-\eqref{eq:KernelDivergenceForm}, satisfying \eqref{eq:KernelHoelderCont} with $\mathcal{A}= Q_1$.
    Let $\phi$ be a solution to 
    \begin{equation*}
        \left\{ \begin{aligned}
          \partial _t \phi -\mathcal{L}_t\phi & = 0 & &\text{in} \quad Q^\Omega _\eps , \\
           \phi & =0 & & \text{in} \quad  \left( I^\ominus _\eps \times ( B_{2\eps} \setminus B^\Omega_\eps ) \right) \cup \left( \{ -\eps ^{2s} \} \times \R^d \right)  , 
        \end{aligned} \right.
    \end{equation*}
    with $0\leq \phi \leq 1$, and let $\phi _0$ be a solution to
    \begin{equation*}
        \left\{ \begin{aligned}
          \partial _t \phi _0 -L_0\phi_0 & = 0 & &\text{in} \quad Q^\Omega _\eps , \\
           \phi_0 & =\phi & & \text{in} \quad  \left( I^\ominus _\eps \times ( \R^d \setminus B^\Omega_\eps ) \right) \cup \left( \{ -\eps ^{2s} \} \times \R^d \right) , 
        \end{aligned} \right.
    \end{equation*}
    where $L_0$ denotes the frozen operator at $0$ (see \eqref{eq:FrozenKernel}).
    Then,
    \begin{equation} \label{eq:CloseHopfEstimate}
        \dashint _{Q^\Omega _\rho} |\phi -\phi _0 | \intd z \leq c \eps ^{-s+\sigma} \rho ^s \quad \text{for all}\quad 0<\rho \leq \eps ,
    \end{equation}
    for some constant $c=c(d,s,\lambda , \Lambda ,\alpha ,\sigma  ,\Omega)>0$.
\end{Proposition}
To prove \autoref{Prop:ClosenessOfSolutionsForHopf}, we start with a freezing estimate (\autoref{Lem:FreezingEstimateForCloSolHopf}) which is an analogue of \autoref{Lem:FreezingEstimateIntW/ds}.
Note that in \autoref{Lem:FreezingEstimateForCloSolHopf}, we prove a stronger estimate compared to the one in \cite[Lemma 7.3]{Kim2024}, which does not include a factor $R^{-\delta}$ on the right-hand side.
\begin{Lemma}\label{Lem:FreezingEstimateForCloSolHopf}
    Assume that we are in the setting of \autoref{Prop:ClosenessOfSolutionsForHopf}.
    Set $u:=\phi -\phi _0$ and let $0<R\leq \eps / 8$.
    Furthermore, let $v_{R}$ be a solution to 
    \begin{equation}\label{eq:DefinitionVForHopf}
        \left\{ \begin{aligned}
          \partial _t v_{R} -L_0v_{R} & = 0 & &\text{in} \quad Q^\Omega _R , \\
           v_{R} & =u & & \text{in} \quad  \left( I^\ominus _R\times ( \R^d \setminus B^\Omega_R ) \right) \cup \left( \{ -R ^{2s} \} \times \R^d \right) , 
        \end{aligned} \right.
    \end{equation}
    and set $w_{R}:=u-v_{R}$.
    Then,
    \begin{equation*}
         \| w_{R}\| _{L^2_t \dot{H}^s_x (I^\ominus _R \times \R^d)} +  \sup _{t\in I^\ominus _R} \| w_{R}(t,\cdot )\| _{L^2 _x (B_R)} \leq c \eps ^{-s} R^{\frac{d}{2}+s+\sigma }
    \end{equation*}
    for some constant $c=c(d,s,\lambda ,\Lambda ,\alpha, \sigma  ,\Omega )>0$.
\end{Lemma}
\begin{proof}
Set $w:=w_{R}$.
    Note that $w$ solves
    \begin{equation} \label{eq:DefinitionWForHopf}
        \left\{ \begin{aligned}
          \partial _t w -L_0w & = (\mathcal{L}_t -L_0) \phi & &\text{in} \quad Q^\Omega _R , \\
           w & =0 & & \text{in} \quad  \left( I^\ominus _R\times ( \R^d \setminus B^\Omega_R ) \right) \cup \left( \{ -R ^{2s} \} \times \R^d \right) . 
        \end{aligned} \right.
    \end{equation}
    Testing this equation with approximations of $\ind _{[-R^{2s},t]} w$, as in \autoref{Lem:EstimateOnHsw}, yields
    \begin{align*}
        \| w\| ^2 _{L^2_t \dot{H}^s_x (I^\ominus _R \times \R^d)} &+  \sup _{t\in I^\ominus _R} \| w(t,\cdot )\| ^2_{L^2 _x (B_R)}
         \leq  c \int _{I^\ominus_R} |\mathcal{E}^{K_0-K} (w, \phi) |\intd t \\
        & \leq  c \int _{I^\ominus_R}R^\sigma \left( \int _{B^\Omega_R} \int _{B_{2R}} \frac{|\phi (t,x)-\phi (t,y)|^2}{|x-y|^{d+2s}} \intd y \intd x\right)^\frac{1}{2} [w] _{H^s_x(\R^d )} \intd t \\
        & \qquad + c  \int _{I^\ominus _R} \int _{B^\Omega _R} \int _{B_{\eps /2} \setminus B_{2R}} |\phi (t,x)- \phi (t,y)| |w(t,x)| \frac{|x|^\sigma + |y|^\sigma + |t| ^\frac{\sigma}{2s}}{|x-y|^{d+2s}} \intd y \intd x\intd t \\
        & \qquad + c  \int _{I^\ominus _R} \int _{B^\Omega _R} \int _{B_1 \setminus B_{\eps /2}}  |\phi (t,x)- \phi (t,y)| |w(t,x)| \frac{|x|^\sigma + |y|^\sigma + |t| ^\frac{\sigma}{2s}}{|x-y|^{d+2s}} \intd y \intd x\intd t\\
        & \qquad + c \int _{I _R^\ominus} \int _{B^\Omega _R} \int _{\R^d \setminus B_{1}} \frac{|\phi (t,x)-\phi (t,y)| |w(t,x)|}{|x-y|^{d+2s}} \intd y \intd x\intd t \\
        & = cI_1 +cI_2+cI_3 +cI_4,
    \end{align*}
    where we have used \eqref{eq:DifferenceFrozenKernelToKernel}.
    For $I_1$, we use Young's inequality to obtain
    \begin{align*}
        I_1 & \leq c (\eta) R^{2\sigma }  \| \phi \| ^2 _{L^2_t \dot{H}^s_x (I_{2R}^\ominus\times B _{2R})} + c \eta \| w\| ^2 _{L^2_t \dot{H}^s_x (I^\ominus _R \times \R^d)} ,
    \end{align*}
    where we choose $\eta>0$ so small that we can absorb the second summand on the left-hand side.
    For the first summand, we use \autoref{Prop:CaccioppoliEstimate}, as well as $\| \phi \| _{L^\infty (I^\ominus_\eps \times \R^d)} \leq 1$ and $[\phi ]_{C_x^s (Q^\Omega _{\eps /2})} \leq c \eps ^{-s}$, and get
    \begin{align*}
       & \| \phi \| ^2 _{L^2_t \dot{H}^s_x (I_{2R}^\ominus\times B _{2R})}  \leq
        c R^{-2s} \| \phi \| ^2_{L^2(Q^\Omega _{4R})} + c R^d \sup _{t\in I^\ominus_{4R}} \left( \Tail (\phi ,2R,0) \right) ^2 \\
        & \leq c R^{d+2s} \eps ^{-2s} + c R^{d+4s} \sup _{t\in I^\ominus_{4R}} \left(  \int _{B_{\eps /2}\setminus B_{2R}} \frac{| \phi(t,y)|}{|y|^{d+2s}} \intd y \right) ^2 +c R^{d+4s} \sup _{t\in I^\ominus_{4R}} \left(  \int _{ \R^d \setminus B_{\eps /2}} \frac{| \phi (t,y)|}{|y|^{d+2s}} \intd y \right) ^2 \\
        & \leq c R^{d+2s} \eps ^{-2s} + c R^{d+4s}  \left(  \int _{B_{\eps /2}\setminus B_{2R}} \frac{\eps^{-s}}{|y|^{d+s}} \intd y \right) ^2 +c R^{d+4s}  \left(  \int _{ \R^d \setminus B_{\eps /2}} \frac{1}{|y|^{d+2s}} \intd y \right) ^2 \\
        & \leq c R^{d+2s} \eps ^{-2s}.
    \end{align*}
    
    For $I_2$, note that $|x|\leq|y| \leq 2|x-y|$ as well as $|t|^{1/(2s)} \leq |y|$. Together with $|\phi (t,x)| \leq c \eps ^{-s}|x|^{s} \leq c\eps ^{-s}|y|^{s}$ and $|\phi (t,y)| \leq c\eps ^{-s}|y|^{s}$, we obtain
    \begin{align*}
        I_2 & \leq c\int _{I^\ominus _R} \int _{B^\Omega _R} \int _{B_{\eps /2} \setminus B_{2R}}  |w(t,x)| \frac{\eps ^{-s}|y|^s}{|y|^{d+2s-\sigma}} \intd y \intd x\intd t \\
        & \leq c \int _{I^\ominus _R}\eps^{-s}R^{\sigma -s} R^{\frac{d}{2}+s} [w(t,\cdot )]_{H^s(\R^d )} \intd t 
        \leq c(\eta ) \eps ^{-2s} R^{d+2s+2\sigma} + c \eta \| w\| ^2 _{L^2_t \dot{H}^s_x (I^\ominus _R \times \R^d)} ,
    \end{align*}
    where we have also used \eqref{eq:EstimateL1toL2toHs}.
    For $I_3$, we use again $|x|\leq|y| \leq 2|x-y|$ as well as $|t|^{1/(2s)} \leq |y|$, together with $|\phi| \leq 1$, \eqref{eq:EstimateL1toL2toHs}, and $\sigma <s$, to obtain
    \begin{align*}
        I_3 &\leq c \int _{I _R^\ominus}  \left(\int _{B_1 \setminus B_{\eps /2}} \frac{ |y|^\sigma}{|y|^{d+2s}} \intd y \right) \left( \int _{B^\Omega _R} |w(t,x)|\intd x\right) \intd t  
        \leq c \int _{I _R^\ominus}\eps ^{\sigma-2s} R^{\frac{d}{2}+s} [w(t,\cdot )]_{H^s (\R^d)}\intd t \\
        & \leq  c \int _{I _R^\ominus}\eps ^{-s} R^{\frac{d}{2}+\sigma} [w(t,\cdot )]_{H^s (\R^d)}\intd t  \leq c(\eta ) \eps ^{-2s} R^{d+2s+2\sigma} + c \eta \| w\| ^2 _{L^2_t \dot{H}^s_x (I^\ominus _R \times \R^d)} .
    \end{align*}
    For $I_4$, we use $|\phi|\leq 1$ and $\sigma <s$, to get
    \begin{align*}
        I_4 & \leq c \int _{I _R^\ominus}  \left(\int _{\R^d \setminus B_1} \frac{ 1}{|y|^{d+2s}} \intd y \right) \left( \int _{B^\Omega _R} |w(t,x)|\intd x\right) \intd t \leq c \int _{I _R^\ominus} R^{\frac{d}{2}+s} [w(t,\cdot )]_{H^s (\R^d)}\intd t \\
        & \leq c(\eta ) R^{d+4s}  +c \eta  \| w\| ^2 _{L^2_t \dot{H}^s_x (I^\ominus _R \times \R^d)} \leq c(\eta ) \eps ^{-2s} R^{d+2s+2\sigma}   +c \eta  \| w\| ^2 _{L^2_t \dot{H}^s_x (I^\ominus _R \times \R^d)} .
    \end{align*}
    By choosing $\eta>0$ small, we can absorb $\| w\| ^2 _{L^2_t \dot{H}^s_x (I^\ominus _R \times \R^d)}$ on the left-hand side.
\end{proof}
\begin{Lemma}\label{Lem:ClosSolHopfBoundOnU}
    Assume that we are in the setting of \autoref{Prop:ClosenessOfSolutionsForHopf}.
    Set $u:=\phi -\phi _0$.
    Then,
    \begin{equation*}
         \| u\| _{L^2_t \dot{H}^s_x (I^\ominus _{\eps } \times \R^d)} +  \sup _{t\in I^\ominus _{\eps }} \| u(t,\cdot )\| _{L^2 _x (B_{\eps })} \leq c  \eps^{\frac{d}{2}+\sigma }
    \end{equation*}
    for some constant $c=c(d,s,\lambda ,\Lambda ,\alpha, \sigma  ,\Omega )>0$.
\end{Lemma}
\begin{proof}
    Note that $u$ solves
    \begin{equation*}
        \left\{ \begin{aligned}
          \partial _t u -L_0u & = (\mathcal{L}_t -L_0) \phi & &\text{in} \quad Q^\Omega _\eps , \\
           u & =0 & & \text{in} \quad \left( I^\ominus _\eps \times ( \R^d \setminus B^\Omega_\eps ) \right) \cup \left( \{ -\eps ^{2s} \} \times \R^d \right) . 
        \end{aligned} \right.
    \end{equation*}
    Hence, this lemma follows by the same arguments as in the proof of \autoref{Lem:FreezingEstimateForCloSolHopf}.
    Note that the restriction $R\leq \eps /8$ in \autoref{Lem:FreezingEstimateForCloSolHopf} is not necessary.
    For $R\in (\eps /8 ,\eps ]$, \autoref{Lem:FreezingEstimateForCloSolHopf} remains true, by following exactly the same proof as above, only using $|\phi| \leq 1$ instead of the bound $[\phi ]_{C_x^s (Q^\Omega _{\eps /2})} \leq c \eps ^{-s}$.
    Furthermore, for the first term in $I_1$, we explicitly use the initial data $\phi =0$ in $\{ - \eps ^{2s}\} \times B_\eps $ to get a Caccioppoli estimate which holds on the full time interval $I^\ominus _{\eps}$.
\end{proof}
\begin{proof}[Proof of \autoref{Prop:ClosenessOfSolutionsForHopf}]
    Set $u:=\phi -\phi _0$ and $\theta =128$.
    Let $ 0< \rho \leq \theta ^{-1}R \leq \theta ^{-1} \eps$.
    Then, let $v:=v_{R/8}$ and $w:=w_{R/8}$ be solutions to \eqref{eq:DefinitionVForHopf} and \eqref{eq:DefinitionWForHopf}.
    Following the proof of \eqref{eq:estimate-iteration-for-v-w-seperated}, using \autoref{Lem:FreezingEstimateForCloSolHopf} (together with Hardy's inequality, see \cite[Lemma 2.4]{Kim2024}) instead of \autoref{Lem:FreezingEstimateIntW/ds}, we obtain for every $\delta \in (0,s)$
    \begin{equation*}
    \begin{aligned}
        \rho ^{2s}\left\| \frac{v}{d_\Omega ^s} \right\| _{L^\infty _t L^1_x (Q^\Omega_\rho )} + \left\| \frac{w}{d_\Omega ^s} \right\| _{L^1 (Q^\Omega_{R/8})} &+  R^{s+\frac{d}{2}} \| w \| _{L^\infty_t L^2_x (Q_{R/8})} \\
        &  \leq c \left(\frac{\rho}{R} \right) ^{d+2s-\frac{\delta}{2}} \Phi _\sigma (u;R,0) + c \eps ^{-s} R^{d+2s+\sigma}  .
    \end{aligned}
\end{equation*}
     Furthermore, following the proof of \eqref{eq:TailEstimateForIteration}, we get
    \begin{align*}
          (\theta^{-m}R) ^{d+s} \sup_{t\in I^\ominus _{\theta^{-m}R} } \Tail _{\sigma ,B_1} (u;\theta^{-m-1}R  ,0)  \leq c  \left( \theta^{-m} \right) ^{d+2s-\frac{\delta}{2}} \Phi _\sigma (u;R) + c \eps ^{-s}R^{d+2s+\sigma}  
    \end{align*}
    for all $m\in \N$.
    Hence, we have shown that
    \begin{align*}
        \Phi _\sigma (u;\theta ^{-m}R,0) \leq c  \left( \theta^{-m} \right) ^{d+2s-\frac{\delta}{2}} \Phi _\sigma (u;R,0) + c \eps ^{\sigma +\delta -s}R^{d+2s-\delta}.
    \end{align*}
    Applying \autoref{Lem:IterationLemmaOnDyadicScale}, we obtain
    \begin{align*}
        \Phi _\sigma (u;\theta ^{-m}R,0) \leq c  \left( \theta^{-m} \right) ^{d+2s-\delta} \Phi _\sigma (u;R,0) + c \eps ^{\sigma +\delta -s} (\theta ^{-m}R)^{d+2s- \delta } .
    \end{align*}
    In particular, by setting $R=\eps $, and using \autoref{Lem:ClosSolHopfBoundOnU} (together with Hardy's inequality) to bound $\Phi _\sigma (u;\eps ,0)$, we get for all $m\in \N$
    \begin{equation}\label{eq:CloseHopfAverageUOverDs}
        \dashint _{Q^\Omega _{\theta ^{-m} \eps }} \left| \frac{u}{d_\Omega ^s} \right| \intd z \leq  c\theta ^{m\delta }\eps ^{-d-2s}   \Phi _\sigma (u;\eps ,0) +   c\theta ^{m \delta }\eps ^{\sigma -s  }  \leq c \eps ^{\sigma -s  }\theta ^{m \delta }
    \end{equation}

    Next, we follow the proof of \autoref{Lem:CampanatoEstimate} to show a Campanato-type estimate.
    Let $\psi =\psi_0$ be the barrier from \eqref{eq:EllipticBarrierUBar} for the frozen operator $L=L_0$.
    Then, following the proof of \autoref{Lem:CampanatoEstimate} using again \autoref{Lem:FreezingEstimateForCloSolHopf} instead of \autoref{Lem:FreezingEstimateIntW/ds}, we obtain
    \begin{align*}
        \Psi (u;\theta^{-m} R,0) &\le c (\theta^{-m})^{d + 2s + \gamma} \Psi (u; R,0) + c\eps ^{-s} R^{d+2s+\sigma } 
    \end{align*}
    for every $\gamma \in (0, s)$.
    Choosing $\gamma \leq 2 \sigma$ (which implies $\eps ^{-s} R^{d+2s+\sigma } \leq \eps ^{\sigma -s -\frac{\gamma}{2}}R^{d+2s+\frac{\gamma}{2}}$) and applying \autoref{Lem:IterationLemmaOnDyadicScale}, we get
    \begin{align*}
        \Psi (u;\theta^{-m} R,0) \leq c (\theta^{-m})^{d + 2s + \frac{\gamma}{2} } \Psi (u; R,0) +c \eps ^{\sigma -s -\frac{\gamma}{2}} (\theta ^{-m}R)^{d+2s+\frac{\gamma}{2}}.
    \end{align*}
    In particular, setting $R=\eps$, we obtain
    \begin{equation}\label{eq:CloseHopfCampanatoEstDyadicScale}
        \dashint _{Q^\Omega _{\theta ^{-m} \eps }} \left| \frac{u}{\psi} - \left( \frac{u}{\psi} \right) _{Q^\Omega _{\theta ^{-m} \eps }} \right| \intd z \leq c \eps ^{\sigma -s- \frac{\gamma}{2}} (\theta ^{-m}\eps )^{\frac{\gamma}{2}} .
    \end{equation}
    By the same arguments as in the proof of \cite[(5.8), (5.9)]{Giaquinta2012}, we get that the limit $q:= \lim _{m\to \infty} (u/\psi)_{Q^\Omega _{\theta ^{-m} \eps }}$ exists and satisfies
    \begin{equation} \label{eq:CloseHopfBoundQMinusAverage}
        \left| q - \left( \frac{u}{\psi} \right) _{Q^\Omega _{\theta ^{-m} \eps }} \right| \leq c \eps ^{\sigma -s- \frac{\gamma}{2}} (\theta ^{-m}\eps )^{\frac{\gamma}{2}}
    \end{equation}
    for all $m\in \N$.
    Together with \eqref{eq:CloseHopfAverageUOverDs}, this implies
    \begin{equation}\label{eq:CloseHopfBoundOnQ}
        |q| \leq \left| q - \left( \frac{u}{\psi} \right) _{Q^\Omega _{\eps }} \right| + \dashint _{Q^\Omega _{ \eps }} \left| \frac{u}{d_\Omega ^s} \right| \intd z \leq c \eps ^{\sigma -s} .
    \end{equation}
    Combining \eqref{eq:CloseHopfCampanatoEstDyadicScale}, \eqref{eq:CloseHopfBoundQMinusAverage}, and \eqref{eq:CloseHopfBoundOnQ} yields
    \begin{equation*}
        \dashint _{Q^\Omega _{\theta ^{-m} \eps }} \left| \frac{u}{\psi} \right| \intd z \leq |q| + \dashint _{Q^\Omega _{\theta ^{-m} \eps }} \left| \frac{u}{\psi} - \left( \frac{u}{\psi} \right) _{Q^\Omega _{\theta ^{-m} \eps }} \right| \intd z + \left| q- \left( \frac{u}{\psi} \right) _{Q^\Omega _{\theta ^{-m} \eps }} \right| \leq c \eps ^{\sigma -s},
    \end{equation*}
    which, by \autoref{Lem:psi-properties}(ii), implies 
    \begin{equation*}
        \dashint  _{Q^\Omega _{\theta ^{-m} \eps }} |u| \intd z\leq c(\theta ^{-m}\eps )^{s}  \dashint _{Q^\Omega _{\theta ^{-m} \eps }} \left| \frac{u}{\psi} \right| \intd z \leq c \eps ^{\sigma -s} (\theta ^{-m}\eps )^{s} .
    \end{equation*}
    Now, \eqref{eq:CloseHopfEstimate} follows by choosing $m\in \N$ such that $ \theta ^{-m-1}\eps \leq \rho \leq \theta^{-m}\eps $.
\end{proof}

\subsection{Proof of Hopf lemma} \label{sec:SubsectionProofHopf}
\begin{proof}[Proof of \autoref{Thm:CsBoundaryRegAndHopfLemma}(ii)]
    By the strong maximum principle, which follows from the interior Harnack inequality \cite[Theorem 1.1]{Kassmann2024}, we may assume that $u>0$ in $Q_1^\Omega$.
    We will show that there exists some $\eps =\eps (d,s,\lambda ,\Lambda , \alpha ,\Omega)>0$ such that for all $(t_0,x_0)\in Q_{1/2}^\Omega$ with $d_\Omega (x_0)\leq \eps /4$, 
    \begin{equation} \label{eq:HopfEstimateCloseToBoundary}
        u(t_0,x_0) \geq c_1 d_\Omega ^s (x_0)
    \end{equation}
    for some constant $c_1 =c_1 (d,s,\lambda ,\Lambda, \alpha ,p,q,r, \sigma ,\Omega ,u, f)>0$.

    To prove \eqref{eq:HopfEstimateCloseToBoundary}, let $\eps \in (0,1)$ be a constant which we will determine later, and fix some $z_0=(t_0,x_0) \in Q_{1/2}^\Omega$ with $d_\Omega(x_0) \leq \eps /4$.
    Furthermore, choose $\tilde{x}\in \partial \Omega $ such that $d_\Omega (x_0) = |x_0 - \tilde{x}|$.
    After a rotation and translation, we may assume that $t_0=0$, $\tilde{x}=0$, and $x_0= d_\Omega (x_0 ) e_d$ where $e_d = (0, \dots ,0,1) \in \R^d$, so that the outer unit normal vector at $\tilde{x}=0 \in \partial \Omega$ is $-e_d$.
    Furthermore, by choosing $\eps >0$ small, we may assume that $B^\eps := B_\eps (4\eps e_d) \subset \Omega$ with $\mathrm{dist}(B^\eps, \Omega ^c)>\eps$.

    Let $\phi _\eps$ be the barrier defined by
    \begin{equation*}
    \left\{ \begin{aligned}
          \partial _t \phi_\eps -\mathcal{L}_t\phi _\eps & = 0 & &\text{in} \quad Q^\Omega _\eps , \\
          \phi _\eps & =\ind _{B^\eps} & & \text{in} \quad  I^\ominus _\eps \times ( \R^d \setminus B^\Omega _\eps ) , \\
           \phi _\eps & =\ind _{B^\eps} & & \text{in} \quad  \{ -\eps ^{2s}\} \times \R^d .
        \end{aligned} \right.
    \end{equation*}
    Set $c_\eps = \inf _{(t,x)\in I^\ominus _\eps \times \{ d_\Omega \geq \eps \}} u>0$.
    Using the parabolic maximum principle, we conclude that
    \begin{equation}\label{eq:BoundUGeqPhiEps}
        u \geq c_\eps \phi _\eps \quad \text{in} \quad I^\ominus_\eps \times \R^d .
    \end{equation}
    Let $L_0$ be the frozen operator at $0$ (see \eqref{eq:FrozenKernel}) and let $\bar{\phi}_\eps$ be the solution to
    \begin{equation*}
    \left\{ \begin{aligned}
          \partial _t \bar{\phi}_\eps -L_0\bar{\phi} _\eps & = 0 & &\text{in} \quad Q^\Omega _\eps , \\
          \bar{\phi} _\eps & =\ind _{B^\eps} & & \text{in} \quad  I^\ominus _\eps \times ( \R^d \setminus B^\Omega _\eps ) , \\
           \bar{\phi} _\eps & =\ind _{B^\eps} & & \text{in} \quad  \{ -\eps ^{2s}\} \times \R^d .
        \end{aligned} \right.
    \end{equation*}

    We claim that 
    \begin{equation} \label{eq:BoundPhiBarEpsGeqDs}
        \bar{\phi}_\eps \geq c_2 \eps ^{-s}d_\Omega^s \quad \text{in} \quad I_{\eps /2}^\ominus \times B^\Omega _{\eps /2} \cap \{ x=x_d e_d \} 
    \end{equation}
    for some constant $c_2=c_2(d,s,\lambda, \Lambda ,\alpha ,\Omega )>0$.

  We intend to deduce \eqref{eq:BoundPhiBarEpsGeqDs} from the Hopf lemma for elliptic equations together with a barrier argument. To do so, let $\bar{\psi}_\eps$ be the solution to the elliptic problem
    \begin{equation*}
    \left\{ \begin{aligned}
          -L_0\bar{\psi} _\eps & = 0 & &\text{in} \quad B^\Omega _\eps , \\
           \bar{\psi} _\eps & =\ind _{B^\eps} & & \text{in} \quad  \R^d \setminus B^\Omega _\eps .
        \end{aligned} \right.
    \end{equation*}
    By \cite[Lemma 7.5]{Kim2024}, $\bar{\psi}_\eps$ satisfies
    \begin{equation}\label{eq:KWLemma7.5}
        \bar{\psi}_\eps (me_d) \geq c_3 \eps ^{-s}m^s \quad \text{for any}\quad m\in (0, \eps /2 ]
    \end{equation}
    for some constant $c_3=c_3(d,s,\lambda, \Lambda ,\alpha ,\Omega )>0$. \\
    Note that $-L_0 \ind _{B^\eps} (x) \asymp -\eps ^{-2s}$ for all $x\in B^\Omega _\eps$.
    Hence, we can choose a small constant $c_4=c_4(d,s,\lambda )\in (0,1/2)$ such that
    \begin{equation*}
        \inf _{x\in B^\Omega _\eps} (-L_0 \ind _{B^\eps} ) (x) \leq -2c_4 \eps ^{-2s} <0.
    \end{equation*}
    Now, let us define 
    \begin{align*}
        p_\eps (t,x)=c_4\eps ^{-2s}(t+\eps ^{2s})\bar{\psi}_\eps (x) + 1/2\ind _{B^\eps}(x).
    \end{align*}
    Then, $p_\eps$ satisfies
    \begin{equation*}
    \left\{ \begin{aligned}
          \partial_t p_\eps - L_0 p_\eps  = c_4 \eps ^{-2s} \bar{\psi}_\eps (x)+ 1/2 (-L_0 \ind _{B^\eps})& \leq  0 & &\text{in} \quad Q^\Omega _\eps , \\
           p _\eps = c_4\eps ^{-2s}(t+\eps ^{2s})\ind _{B^\eps} + 1/2 \ind _{B^\eps} & \leq\ind _{B^\eps} & & \text{in} \quad   I^\ominus _\eps \times  (\R^d \setminus B^\Omega _\eps )  , \\
           p _\eps = 1/2 \ind _{B^\eps}& \leq\ind _{B^\eps} & & \text{in} \quad   \{ -\eps ^{2s} \} \times  \R^d   ,
        \end{aligned} \right.
    \end{equation*}
    where we have used $\bar{\psi}_\eps\leq 1$ (by the maximum principle) and $c_4\leq 1/2$.

    Using the maximum principle (on $\bar{\phi}_\eps -p_\eps$) and \eqref{eq:KWLemma7.5}, we conclude that $\bar{\phi}_{\eps} \ge p_{\eps}$, which yields \eqref{eq:BoundPhiBarEpsGeqDs}, as claimed.

    We want to finish the proof of \eqref{eq:HopfEstimateCloseToBoundary}, by using the bounds \eqref{eq:BoundUGeqPhiEps} and \eqref{eq:BoundPhiBarEpsGeqDs}, together with the closeness of $\phi_\eps$ and $\bar{\phi}_\eps$ from \autoref{Prop:ClosenessOfSolutionsForHopf}.
    For every $\rho \in (0,\eps /4]$ let us define
    \begin{equation*}
        I^\oplus = I^\ominus _{\rho /4^\frac{1}{2s}} \quad \text{and} \quad I^\ominus = I^\ominus _{\rho /4 ^{\frac{1}{2s}}}(3\rho ^{2s}/4) .
    \end{equation*}
    Using that $\partial \Omega \in C^{1,\alpha}$, we get that for every $\rho \in (0,\eps /4]$ there exists some $y\in B^\Omega _\rho$ and $\kappa >0$ with $\kappa \asymp \rho$ such that $B_{2\kappa} (y) \subset B_\rho ^\Omega$ and $I^\ominus _{\kappa} (-c\kappa ^{2s})  \subset I^\ominus $ for some $c>0$.
    Using \eqref{eq:BoundPhiBarEpsGeqDs} and the interior Harnack inequality (\cite[Theorem 1.1]{Kassmann2024}), we obtain $\bar{\phi}_\eps \geq c_5 \eps ^{-s}\rho ^s$ in $I^\ominus _{\kappa} (-c\kappa ^{2s}) \times B_\kappa (y)$.
    Hence,
    \begin{equation}\label{eq:AverageBoundPhiBarEps}
        \dashint _{I^\ominus \times B^\Omega _\rho } \bar{\phi} _\eps \intd z \geq  \frac{| I^\ominus _{\kappa} (-c\kappa ^{2s}) \times B_\kappa (y)|}{| I^\ominus \times B^\Omega _\rho  |} \dashint _{I^\ominus _{\kappa} (-c\kappa ^{2s}) \times B_\kappa (y)} \bar{\phi} _\eps \intd z \geq c_6 \eps ^{-s}\rho ^s .
    \end{equation}
    By \autoref{Prop:ClosenessOfSolutionsForHopf}, we have for every $\rho \in (0,\eps /4]$
    \begin{equation}\label{eq:ClosenessBarriersInHopfProof}
        \dashint _{I^\ominus \times B^\Omega _\rho }| \phi _\eps -\bar{\phi}_\eps | \intd z \leq c_7 \rho ^s \eps ^{-s+\sigma } .
    \end{equation}
    Now, we choose $\eps>0$ small enough such that $c_7 \eps ^{\sigma } \leq c_6 /2$. Then, by \eqref{eq:AverageBoundPhiBarEps} and \eqref{eq:ClosenessBarriersInHopfProof}, we get
    \begin{equation*}
        \dashint _{I^\ominus \times B^\Omega _\rho }\phi _\eps \intd z \geq \dashint _{I^\ominus \times B^\Omega _\rho }\bar{\phi} _\eps \intd z - \dashint _{I^\ominus \times B^\Omega _\rho }|\phi _\eps -\bar{\phi}_\eps | \intd z \geq c_6 \eps ^{-s}\rho ^s - c_7 \rho ^{s} \eps ^{-s+\sigma } \geq \frac{c_6}{2} \eps ^{-s} \rho ^s .
    \end{equation*}
    Together with \eqref{eq:BoundUGeqPhiEps}, we obtain
    \begin{equation*}
        c_\eps \frac{c_6}{2} \eps ^{-s} \rho ^{s}\leq \dashint _{I^\ominus \times B^\Omega _\rho} u \intd z \quad \text{for all}\quad \rho \in (0, \eps /4] .
    \end{equation*}
    Now, fix some small $\eps >0$ that satisfies the above restrictions.
    Assume, without loss of generality, that $ \| u\|_{L^2 (Q_1)} +\|\Tail (u ;1/2,0) \|_{ L^{p}_t (I^\ominus _1)} + \| f\| _{L^q_tL^r_x (Q_1^\Omega)} \leq 1$.
    Then, by \autoref{Thm:CsBoundaryRegAndHopfLemma}(i), we have $[u]_{C^s_x (Q^\Omega _{1/2})} \leq c$.
    Therefore, for every $\eta \in (0,1)$, we get
    \begin{equation*}
        | I^\ominus \times B^\Omega _\rho |^{-1} \int _{I^\ominus \times (B^\Omega _\rho \cap \{ d_\Omega \leq \eta \rho \})} u \intd z \leq  (\eta \rho )^s [u]_{C^s_x (Q^\Omega _{1/2})} \leq c_8 \eta ^s \rho ^s .
    \end{equation*}
    Next, we choose $\eta \in (0,1)$ such that $c_8 \eta ^s \leq c_\eps c_6 \eps ^{-s}/4$. Note that $\eta$ depends on $\eps$ but not on $\rho$.
    Then, we get for any $\rho \in (0, \eps /4]$
    \begin{align*}
        \frac{c_\eps c_6 \eps ^{-s}}{4} \rho ^s &\leq \dashint _{I^\ominus \times B^\Omega _\rho} u \intd z - \frac{c_\eps c_6 \eps ^{-s}}{4} \rho ^s \\
        &\leq | I^\ominus \times B^\Omega _\rho |^{-1} \int _{I^\ominus \times (B^\Omega _\rho \cap \{ d_\Omega > \eta \rho \})} u \intd z \leq c_9 \inf _{I^\oplus \times (B_\rho ^\Omega \cap \{ d_\Omega >\eta \rho \} ) } u,
    \end{align*}
    where in the last step we have used the weak Harnack inequality for the supersolution $u$ (see, e.g., \cite[Theorem 1.1]{KassmannWeidner2022.1}) on balls with radii comparable to $\eta \rho $.
    Note that the constant $c_9$ also depends on $d,s,\lambda, \Lambda  , \eta$, and therefore also on $\eps$, but not on $\rho$.

    Now, by choosing $\rho =d_\Omega (x_0)$, we deduce \eqref{eq:HopfEstimateCloseToBoundary}, which concludes the proof.
\end{proof}
\begin{Remark}\label{Rem:HopfDependenceOnUAndF}
    Following the above proof, we can make the dependence of the constant $c$ on $u$ and $f$ in the Hopf Lemma $u\geq c d_\Omega ^s$ (see \autoref{Thm:CsBoundaryRegAndHopfLemma}(ii)) more explicit.
    After dividing the equation by a constant, we may assume that 
    \begin{equation*}
         \| u\|_{L^2 (Q_1)} +\|\Tail (u ;1/2,0) \|_{ L^{p}_t (I^\ominus _1)} + \| f\| _{L^q_tL^r_x (Q_1^\Omega)} \leq 1 .
    \end{equation*}
    Then, there exists some $\eps =\eps (d,s,\lambda ,\Lambda ,\alpha , \Omega)>0$ such that the constant $c$ in the Hopf Lemma depends on $u$ and $f$ only through a lower bound $c_0$ of $u$ in the interior
    \begin{equation*}
        0<c_0\leq \inf _{Q^\Omega _{1/2+\eps}\cap \{ d_\Omega \geq \eps \}}u .
    \end{equation*}
    Furthermore, note that the constant $c$ in the Hopf Lemma and $\eps$ depend on the domain $\Omega$ only through an upper bound of the $C^{1,\alpha}$ norm of the boundary $\partial \Omega$.
\end{Remark}

\section{Boundary Harnack result}\label{sec:BoundaryHarnack}
\begin{proof}[Proof of \autoref{Thm:BoundaryHarnack}]
    First, note that
    \begin{equation*}
        \left\| u_1 /u_2 \right\| _{L^\infty (Q _{1/2}^\Omega)} \leq  \| u_1 /d_\Omega^s\| _{L^\infty (Q _{1/2}^\Omega)} \left\| d_\Omega ^s / u_2 \right\| _{L^\infty (Q _{1/2}^\Omega)}.
    \end{equation*}
     By the $C^s$ boundary regularity (see \autoref{Thm:CsBoundaryRegAndHopfLemma}(i)), we have $\| u_1 /d_\Omega ^s \| _{L^\infty (Q _{1/2}^\Omega)} \leq c$.
     Furthermore, by the Hopf lemma (see \autoref{Thm:CsBoundaryRegAndHopfLemma}(ii) together with \autoref{Rem:HopfDependenceOnUAndF}), we have
    $\left\| d_\Omega^s /u_2 \right\| _{L^\infty (Q _{1/2}^\Omega)} \leq c$ for some constant $c$ which also depends on a lower bound of $u_2$ in the interior which in turn can be bounded in terms of $0<c_0 \leq u_2 (-(3/4)^{2s},x_0)$ by a Harnack chain argument.
    Hence, we get that $\left\| u_1 /u_2 \right\| _{L^\infty (Q _{1/2}^\Omega)} \leq c$.

    Now, set $h_i(t,x):=\frac{u_i(t,x)}{\psi _{(t,x)}(x)}$ where $\psi _{(t,x)}$ is the function $\psi$ from \eqref{eq:EllipticBarrierUBar} with respect to the frozen operator $L_{(t,x)}$. Then, for every $z_1 ,z_2 \in Q^\Omega _{1/2}$, we get
    \begin{align*}
        \left| \frac{h_1(z_1)}{h_2(z_1)} - \frac{h_1(z_2)}{h_2(z_2)} \right| & = \frac{1}{|h_2(z_1)h_2(z_2)|} \left| h_1(z_1)h_2(z_2) -h_1(z_2)h_2(z_1) \right| \\
        & \leq \| h_2^{-2}\| _{L^\infty (Q^\Omega _{1/2})}\left( |h_2(z_2)| |h_1(z_1) -h_1(z_2)| + |h_1(z_2)| |h_2(z_2)-h_2(z_1)| \right) .
    \end{align*}
    Hence, we obtain
    \begin{align*}
        [h_1/h_2] _{C^{\beta}_{p} (\overline{Q^\Omega _{1/2}})} 
       \leq \| h_2^{-2}\| _{L^\infty (Q^\Omega _{1/2})}\left(  \| h_2\| _{L^\infty (Q^\Omega _{1/2})}  [h_1] _{C^{\beta}_{p} (\overline{Q^\Omega _{1/2}})} + \| h_1\| _{L^\infty (Q^\Omega _{1/2})}  [h_2] _{C^{\beta}_{p} (\overline{Q^\Omega _{1/2}})} \right) .
    \end{align*}
    By the $C^s_x$-boundary regularity together with $\psi \asymp d_\Omega ^s$ (see \autoref{Lem:psi-properties}(ii)) we get $ \| h_i\| _{L^\infty (Q^\Omega _{1/2})} \leq c$.
    Moreover, we have $\| h_2 ^{-2}\| _{L^\infty (Q^\Omega _{1/2})} \leq c$ by the Hopf lemma and \autoref{Lem:psi-properties}(ii).
    
    It remains to show that $ [h_i] _{C^{\beta}_{p} (\overline{Q^\Omega _{1/2}})} \leq c$ for both $i=1,2$.
    First, note that by a cut-off argument as in the proof of  \autoref{Thm:CsBoundaryRegularityGlobalAAssumption} $\Rightarrow$ \autoref{Thm:CsBoundaryRegAndHopfLemma}(i), we can assume that $u_i$ is globally bounded by treating the tail of $u_i$ as a right-hand side.
    Now, we can apply \autoref{Lem:uOverPsiHoelderContinuous} and get $ [h_i] _{C^{\beta}_{p} (\overline{Q^\Omega _{1/2}})} \leq c$.
    This concludes the proof since $u_1/u_2= h_1/h_2$.    
\end{proof}
\begin{Remark}\label{Rem:BetaEquSigmaPossible}
    Note that in the proof of \autoref{Thm:BoundaryHarnack}, we have used the assumption $\beta < \sigma$ only when applying \autoref{Lem:uOverPsiHoelderContinuous}.
    In fact, one can improve \autoref{Lem:uOverPsiHoelderContinuous} (and hence \autoref{Thm:BoundaryHarnack}) to hold for the entire range $ 0 < \beta  \leq \min \{\sigma , s-\frac{2s}{q}-\frac{d}{r} , s-\frac{2s}{p}\}$.

    To see this, we have to run the proof of \autoref{Lem:uOverPsiHoelderContinuous} again and use the fact that we have already established the $C^s_p$ boundary regularity (from \autoref{Thm:CsBoundaryRegAndHopfLemma}(i)).
    To be more precise, in the proof of \autoref{Lem:uOverPsiHoelderContinuous}, we can choose $\overline{\eps}=0$ since $\Phi _{\sigma} (u;R) \leq c R^{d+2s}$ by using $|u| \leq c d_\Omega^s$ in $Q^\Omega _{3/4}$ (the tail part of $\Phi _{\sigma} (u;R)$ can be bounded as in \autoref{Lem:psi-properties}(iii)).
    Furthermore, we can choose $\tilde{\eps}=0$ in the proof of \autoref{Lem:uOverPsiHoelderContinuous}.
    However, for this, we need that \autoref{Lem:ClosenessOfBarriers} remains true for $\eps =0$. This can be proved by running another Campanato iteration scheme similar to the one in \cite[Section 6]{Kim2024}.
\end{Remark}

\section{Dirichlet heat kernel bounds}\label{sec:DirichletHeatKernelBounds}
The main goal of this section is to prove the Dirichlet heat kernel estimates from \autoref{Thm:HeatKernelBounds}. Before proving the upper bound in \autoref{sec:UpperHeatKernelBounds} and the lower bound in \autoref{sec:LowerHeatKernelBounds}, we start with some basic properties for the Dirichlet heat kernel.
\begin{Proposition}
    \label{prop:representation}
    Let $\Omega \subset \R^d$ be a bounded domain and let $\mathcal{L}_t$ be an operator of the form \eqref{eq:OperatorDivergenceForm}-\eqref{eq:KernelDivergenceForm}. Let $\eta \in [0,T)$. Then, there exists a function $(t,x,y) \mapsto p_{\eta,t}^{\Omega}(x,y)$, such that for any $f \in L^2(\Omega)$,
    \begin{align}
    \label{eq:representation}
        (t,x) \mapsto P_{\eta,t}^{\Omega}f(x) := \int_{\Omega} p_{\eta,t}^{\Omega}(x,y) f(y) \d y
    \end{align}
    is the unique weak solution to 
    \begin{equation} \label{eq:DirichletSemiGroup}
    \left\{ \begin{aligned}
        \partial _t P_{\eta,t}^{\Omega}f -\mathcal{L}_t P_{\eta,t}^{\Omega}f & =0 & &\text{in} \quad (\eta , T ) \times \Omega , \\
           P_{\eta,t}^{\Omega}f & =0 & & \text{in} \quad  (\eta ,T ) \times \Omega ^c , \\
           P_{\eta,t}^{\Omega}f & = f & & \text{in} \quad  \{ \eta \} \times \R ^d
    \end{aligned} \right.
\end{equation}
    in the sense of \autoref{Def:WeakSolConcept}, which attains its initial datum in the sense of $L^2_{loc}(\Omega)$. 
    Moreover, it holds for any $x,y \in \R^d$ and $0 \le \eta < \tau < t < T$ that 
    \begin{align}
    \label{eq:p-properties}
        p_{\eta,t}^{\Omega}(x,y) \ge 0,  ~~ \int_{\Omega} p_{\eta,t}^{\Omega}(x,y) \d y \le 1, ~~ p_{\eta,t}^{\Omega}(x,y) = \int_{\Omega} p_{\tau,t}^{\Omega}(x,z) p_{\eta,\tau}^{\Omega}(z,y) \d z.
    \end{align}
    Furthermore, we have $p_{\eta,t}^{\Omega}(x,y) = 0$ whenever $x \in \Omega^c$ or $y \in \Omega^c$.
   Likewise, for $0<t<\xi\le T$ there exists a function $(t,x,y) \mapsto \hat{p}_{\xi,t}^{\Omega}(x,y)$, such that for any $f \in L^2(\Omega)$ the function
        \begin{align}
    \label{eq:dual-representation}
        (t,x) \mapsto \hat{P}_{\xi,t}^{\Omega} f(x) := \int_{\Omega} \hat{p}_{\xi,t}^{\Omega}(x,y) f(y) \d y
    \end{align}
    solves the dual problem 
        \begin{equation} \label{eq:DirichletSemiGroupDual}
    \left\{ \begin{aligned}
        \partial _t \hat{P}_{\xi,t}^{\Omega}f + \mathcal{L}_t \hat{P}_{\xi,t}^{\Omega}f & =0 & &\text{in} \quad (0,\xi) \times \Omega , \\
           \hat{P}_{\xi,t}^{\Omega}f & =0 & & \text{in} \quad  (0,\xi) \times \Omega ^c , \\
           \hat{P}_{\xi,t}^{\Omega}f & = f & & \text{in} \quad  \{ \xi \} \times \R ^d .
    \end{aligned} \right.
\end{equation}
    Moreover, the following symmetry holds:
    \begin{align}
    \label{eq:dual-relation}
        p_{\eta,t}^{\Omega}(x,y) = \hat{p}_{t,\eta}^{\Omega}(y,x) \qquad \forall\, x,y \in \R^d, ~~ \forall\, 0 \le \eta < t \le T.
    \end{align}
\end{Proposition}

\begin{proof}
    The proof closely follows the arguments of \cite[Proposition A.1]{Liao2024}. The existence and uniqueness of the solution $P_{\eta,t}^{\Omega}f$ to \eqref{eq:DirichletSemiGroup} are well known (see \cite[Theorem 5.3]{Felsinger2014}). Moreover, it holds
    \begin{align*}
        \sup_{t \in (\eta,T)} \int_{\Omega} |P_{\eta,t}^{\Omega} f(x)|^2 \d x \le c \Vert f \Vert_{L^2(\Omega)}^2,
    \end{align*}
    after testing the weak formulation with $P_{\eta,t}^{\Omega} f$ itself. Hence, using local boundedness up to the boundary (\autoref{Prop:LocalBoundedness}), we can show for any $f \in L^2(\Omega)$,
    \begin{align*}
        \Vert P_{\eta,t}^{\Omega} f \Vert_{L^{\infty}(\Omega)} \le c  \Vert f \Vert_{L^2(\Omega)},
    \end{align*}
    where $c = c(\eta,t,\Omega)$. Hence, upon choosing $f = \ind_A$ for Borel sets $A \subset \R^d$, we obtain the existence of $p_{\eta,t}^{\Omega}(x,y)$ by the Radon-Nikodym theorem as in \cite[Proof of Lemma A.4]{Liao2024}. Moreover, the properties in \eqref{eq:p-properties} follow in the exact same way as in \cite[Subsection A.3]{Liao2024} by using the weak maximum principle. Clearly, we have $p_{\eta,t}^{\Omega}(x,y) = 0$, for $x \in \Omega^c$ and $y \in \Omega$, since otherwise, the exterior condition in \eqref{eq:DirichletSemiGroup} would be violated for some $f \in L^2(\Omega)$. All these properties for the dual operator and the dual fundamental solution follow in the same way.
    
    To prove \eqref{eq:dual-relation}, we let $f,g \in L^2(\Omega)$ and $\zeta_k \in C_c^1(0,T)$ be a sequence of functions approximating $\ind_{[t_0,t_1]}$ with $\partial_t \zeta_k \to \delta_{\{t_0\}} - \delta_{\{t_1\}}$ for some $\eta < t_0 < t_1 < T$. Then, we obtain by testing the weak formulation for $P_{\eta,t}^{\Omega}f$ with $\zeta _k\hat{P}_{T,t}^{\Omega}g$ and vice versa, that
    \begin{align*}
        \int_{\eta}^T \partial_t \zeta_k(t) \int_{\Omega} P_{\eta,t}^{\Omega}f(x) \hat{P}_{T,t}^{\Omega}g(x) \d x \d t = 0. 
    \end{align*}
    Hence, by taking $k \to \infty$, we have that
    \begin{align*}
        t \mapsto \int_{\Omega} P_{\eta,t}^{\Omega}f(x) \hat{P}_{T,t}^{\Omega}g(x) \d x
    \end{align*}
    is constant, and from here the proof of \eqref{eq:dual-relation}  continues as in \cite[Subsection A.4]{Liao2024}.
\end{proof}

From now on, we will take $T = 1$ and denote $p_{\Omega}(t,x,y) = p_{0,t}^{\Omega}(x,y)$. Then, we have the following lemma:

\begin{Proposition} \label{Prop:PropertiesDirichletHeatKernel}
For all $y \in \Omega$, the function $(t,x) \mapsto p_{\Omega}(t,x,y)$ solves 
\begin{equation} \label{eq:DirichletHeatKernel}
    \left\{ \begin{aligned}
        \partial _t p_\Omega (t,\cdot ,y)-\mathcal{L}_t p_\Omega (t, \cdot ,y) & =0 & &\text{in} \quad (0,1 ) \times \Omega , \\
           p_\Omega (t,\cdot ,y) & =0 & & \text{in} \quad  (0 ,1 ) \times \Omega ^c , \\
           p_\Omega (t,\cdot ,y) & =\delta _{\{ y \} } & & \text{in} \quad  \{ 0\} \times \R ^d .
    \end{aligned} \right.
\end{equation}
Moreover, for any $x \in \Omega$ and $\tau \in (0,1)$, the function $(t,y) \mapsto p_{\tau-t,\tau}^{\Omega}(x,y)$ solves \eqref{eq:DirichletHeatKernel} with $\mathcal{L}_t$ replaced by $\mathcal{L}_{\tau-t}$ in $(0,\tau) \times \Omega$ and $p_{\tau-t,\tau}^{\Omega}(x,y) = \delta_{\{x\}}$ if $t = 0$.\\
Moreover, it holds for any $a>0$, $x,y \in \Omega$, and $\eta,t \in (0,1)$ with $\eta < t$,
    \begin{align*} 
    p^{\Omega}_{\eta,t}(x,y) &\leq c_1 \left( (t-\eta)^{-\frac{d}{2s}} \wedge\frac{t-\eta}{|x-y|^{d+2s}} \right), \\
    p^{\Omega}_{\eta,t}(x,y) &\geq c_2 \left( (t-\eta)^{-\frac{d}{2s}} \wedge\frac{t-\eta}{|x-y|^{d+2s}} \right), \qquad \text{ if } \min\{d_{\Omega}(x),d_{\Omega}(y)\} \ge a (t-\eta)^{\frac{1}{2s}},
    \end{align*}
    for some constants $c_1=c_1(d,s,\lambda , \Lambda)>0$ and $c_2=c_2(d,s,\lambda ,\Lambda ,a)>0$.
\end{Proposition}

\begin{proof}
    The proof of the first property follows in the same way as in \cite[Lemma 4.3]{Liao2025}. Moreover, \cite[Lemma 4.3]{Liao2025} implies that for $\tau \in (0,t)$, we have that $(t,y) \mapsto p_{t,\tau}^{\Omega}(x,y)$ solves \eqref{eq:DirichletHeatKernel} with $\mathcal{L}_t$ replaced by $- \mathcal{L}_t$ in $(0,\tau) \times \Omega$, and $p_{t,\tau}^{\Omega}(\cdot,x) = \delta_{\{ x \}}$ for $t = \tau$. Hence, we observe that $(t,y) \mapsto p_{\tau - t,\tau}^{\Omega}(x,y)$ solves \eqref{eq:DirichletHeatKernel} in $(0,\tau) \times \Omega$ with $\mathcal{L}_t$ replaced by $\mathcal{L}_{\tau-t}$.

    The upper bound for $p^{\Omega}_{\eta,t}(x,y)$ follows from \cite[Theorem 1.3]{Liao2025} since $p^{\Omega}_{\eta,t}(x,y)$ is the fundamental solution $\tilde{p}_{\Omega}(t,x,y)$ corresponding to $(\mathcal{L}_{\eta+t})_t$ for $t \in (0,T -\eta)$, and by observing that $p^{\Omega}_{\eta,t}(x,y) \le p_{\eta,t}(x,y)$ due to the comparison principle, where $p_{\eta,t}(x,y)$ denotes the fundamental solution to the Cauchy problem, which was constructed in \cite[Proposition 4.1]{Liao2025}.

   Finally, we prove the lower heat kernel bound. Again, it suffices to consider the case $\eta = 0$. Let $x_0,y_0 \in \Omega$ and $t_0 \in (0,1)$ such that $ \min\{d_{\Omega}(x_0),d_{\Omega}(y_0)\} \ge a t_0^{1/(2s)}$. The proof follows by proceeding exactly as in the proof for the global heat kernel in \cite[Proposition 4.7, Lemma 4.8, Lemma 4.9]{Liao2025}.

   \textit{Case 1:} If $|x_0-y_0| \leq a t_0^{1/(2s)}/4$, then by setting $R=at_0^{1/(2s)}/8$, we can choose $z\in \Omega$ such that $x_0, y_0 \in B_{2R}(z)$ and $B_{4R}(z)\subset\Omega$.
    Following the proof of \cite[Lemma 4.8]{Liao2025}, we get
   \begin{equation} \label{eq:Twice-Improved-Harnack-For-LowerBound}
       c \leq \int _{B_{2R}(z)} \hat{p}^\Omega _{(4R/a)^{2s},0} ( y_0 , x ) \intd x = \int _{B_{2R}(z)} p^\Omega _{0,(4R/a)^{2s}} (x, y_0 ) \intd x \lesssim R^d p ^\Omega _{0,(8R/a)^{2s}} (x_0,y_0) ,
   \end{equation}
   where in the first inequality, we have used the improved weak Harnack inequality from \cite[(3.5)]{Liao2025} for the solution to the dual problem \eqref{eq:DirichletSemiGroupDual} with initial datum $f=\ind _{B_{2R}(z)}$. Furthermore, we have used the symmetry from \eqref{eq:dual-relation}, and finally, we have again applied \cite[(3.5)]{Liao2025} to the caloric solution $(t,x) \mapsto p^\Omega _{0,t}(x,y_0)$. Note that \eqref{eq:Twice-Improved-Harnack-For-LowerBound} implies $ t_0^{-d/(2s)} \lesssim p ^\Omega _{0,t_0} (x_0,y_0)$.

   \textit{Case 2:} If $|x_0-y_0| \geq a t_0^{1/(2s)}/4$, then, we set $(4R/a)^{2s}=t_0/4$.
    By the Harnack inequality from \cite[(3.2)]{Liao2025} and the fact that $B_{R}(x_0) \subset \R^d \setminus B_R(y_0)$, we obtain
   \begin{equation*}
       p^\Omega _{0,t_0}(x_0,y_0) \gtrsim  \int _{t_0 - 2 \left( \frac{4R}{a}\right) ^{2s}}^{t_0 -\left( \frac{4R}{a}\right)^{2s}} \int _{\R^d \setminus B_{R}(y_0)} \frac{p^\Omega _{0,\tau }(x_0,z)}{|y_0-z|^{d+2s}} \intd z \intd \tau
       \gtrsim  \int _{t_0/2}^{3t_0/4} \int _{B_{R}(x_0)} \frac{p^\Omega_{0,\tau }(x_0,z)}{|y_0-z|^{d+2s}} \intd z \intd \tau .
   \end{equation*}
    Note that for all $z\in B_R(x_0)$ and $\tau \in (t_0/2 , 3t_0/4 )$, we can use the lower bound for $p^\Omega_{0,\tau }(x_0,z)$ from the first case.
    Following the computation from the proof of \cite[Lemma 4.9]{Liao2025}, we get $p^\Omega _{0,t_0}(x_0,y_0) \gtrsim t_0 / |x_0-y_0|^{d+2s}$.
   Together with the first case, this proves the lower heat kernel bound.
\end{proof}

\subsection{Upper Dirichlet heat kernel bounds}\label{sec:UpperHeatKernelBounds}

In this subsection we prove the upper bound in \autoref{Thm:HeatKernelBounds}. The following lemma demonstrates how upper heat kernel bounds can be deduced from boundary regularity estimates for parabolic equations.

\begin{Lemma}\label{Lem:UpperHeatKernelBounds}
    Fix a bounded domain $\Omega \subset \R^d$ and let $\mathcal{L}_t$ be an operator of the form \eqref{eq:OperatorDivergenceForm}-\eqref{eq:KernelDivergenceForm}.
    Assume that there exists some $\beta \in (0,s]$ and $c_1>0$ such that for any $\tilde{x} \in \partial \Omega$, $t_0 \geq R^{2s} >0$,  
    and any solution $v$ to
    \begin{equation*}
        \left\{ \begin{aligned}
          \partial _t v -\mathfrak{L}_tv & =f & &\text{in} \quad I^\ominus _R(t_0) \times (B_R(\tilde{x})\cap \Omega ) , \\
           v & =0 & & \text{in} \quad  I^\ominus _R(t_0) \times ( B_R(\tilde{x}) \setminus \Omega ) ,
        \end{aligned} \right.
    \end{equation*}
    where $\mathfrak{L}_t=\mathcal{L}_t$ or $\mathfrak{L}_t=\mathcal{L}_{t_0-t}$,
    the following boundary regularity estimate holds
    \begin{equation}\label{eq:BoundRegulAssumUpperHeatKernel}
        [v] _{C^\beta _x (\overline{Q^\Omega_{R/2} (t_0,\tilde{x})})} \leq c_1 R^{-\beta} \left( \| v\| _{L^\infty ( I^\ominus _R(t_0) \times \R^d) }+ R^{2s}\| f\| _{L^\infty (Q^\Omega_R (t_0,\tilde{x}))} \right) .
    \end{equation}
    Then, for all $x,y\in \Omega$ and $t\in (0,1)$, the following upper bound holds
    \begin{equation}\label{eq:UpperHeatKernelBoundWithBeta}
        p_\Omega (t,x,y) \leq c \left( 1 \wedge \frac{d_\Omega (x)}{t^\frac{1}{2s}} \right) ^\beta  \left( 1 \wedge \frac{d_\Omega (y)}{t^\frac{1}{2s}} \right) ^\beta \left( t^{-\frac{d}{2s}} \wedge\frac{t}{|x-y|^{d+2s}} \right)
    \end{equation}
    for some constant $c=c(d,s, \lambda ,\Lambda ,c_1, \beta , \Omega )>0$.
\end{Lemma}

\begin{proof}
    Fix some $t_0 \in (0,1)$, and $x_0 ,y_0\in \Omega$. We will show \eqref{eq:UpperHeatKernelBoundWithBeta} for $(t,x,y)=(t_0,x_0,y_0)$.

    \textbf{Step 1:} Note that by \autoref{Prop:PropertiesDirichletHeatKernel}, we have
    \begin{equation} \label{eq:UpperBoundGlobalHeatKernelInStep1}
        p_\Omega (t_0,x_0,y_0) \leq c \left( t_0^{-\frac{d}{2s}} \wedge\frac{t_0}{|x_0-y_0|^{d+2s}} \right) 
    \end{equation}
    for all $(t_0, x_0,y_0) \in (0,1)\times \Omega \times \Omega$.

    \textbf{Step 2:} Assume that $4d_\Omega (x_0) < t_0^{\frac{1}{2s}}$.
    Set $\tilde{R}=t_0^{\frac{1}{2s}}$ and choose $\tilde{x}\in \partial \Omega$ so that $|x_0-\tilde{x}|=d_\Omega (x_0)$, then $x_0 \in B_{\tilde{R}/4}(\tilde{x})$.
    Define $v(t,x):= p_\Omega (t,x,y_0) \ind _{B_{\tilde{R}}(\tilde{x})} (x)$.
    We start with the following claim
    \begin{equation}\label{eq:ClaimBoundOperatorOnV}
        f(t,x):=(\partial_t - \mathcal{L}_t) v(t,x) \leq c \tilde{R}^{-2s}\left( t_0^{-\frac{d}{2s}} \wedge \frac{t_0}{|x_0-y_0| ^{d+2s}}\right) \qquad\forall (t,x) \in Q_{\tilde{R}/2}(t_0,\tilde{x}) .
    \end{equation}
    To prove \eqref{eq:ClaimBoundOperatorOnV}, we first assume $|x_0-y_0| \leq 4\tilde{R}$.
    Then, using that $p_\Omega ( \cdot ,\cdot ,y_0)$ is caloric (see \eqref{eq:DirichletHeatKernel}), and $p_\Omega(t,z,y_0) \leq c t ^{-\frac{d}{2s}}$ (see \eqref{eq:UpperBoundGlobalHeatKernelInStep1}), as well as
    \begin{equation} \label{eq:zMinusXComparedToZMinusTildeX}
    |z-x| \geq  |z-\tilde{x}| -|\tilde{x}-x|\geq |z-\tilde{x}|-1/2 \tilde{R} \geq |z-\tilde{x}| -1/2 |z-\tilde{x}| =1/2 |z-\tilde{x}| , \quad \forall z \in B_{\tilde{R}}(\tilde{x}) ^c ,
    \end{equation}
    we obtain for all $(t,x) \in Q_{\tilde{R}/2}(t_0,\tilde{x})$
    \begin{align*}
        (\partial_t - \mathcal{L}_t) v(t,x) & = \int _{B_{\tilde{R}}(\tilde{x})^c} \frac{p_\Omega(t,z,y_0)}{|z-x|^{d+2s}} \intd z
        \leq c t^{-\frac{d}{2s}} \int _{B_{\tilde{R}}(\tilde{x})^c} \frac{1}{|z-x|^{d+2s}} \intd z \\
        & \leq c t_0^{-\frac{d}{2s}} \int _{B_{\tilde{R}}(\tilde{x})^c} \frac{1}{|z-\tilde{x}|^{d+2s}} \intd z \leq c t_0^{-\frac{d}{2s}} \tilde{R}^{-2s} .
    \end{align*}
    This proves \eqref{eq:ClaimBoundOperatorOnV} in the case $|x_0-y_0| \leq 4\tilde{R}$.
    
    Next, we assume $4\tilde{R} \leq|x_0-y_0|$.
    Set $R:=|x-y|$.
    Using the fact that $p_\Omega (\cdot ,\cdot ,y_0)$ is caloric (\autoref{Prop:PropertiesDirichletHeatKernel}) and \eqref{eq:UpperBoundGlobalHeatKernelInStep1}, for every $(t,x)\in Q_{\tilde{R}/2}(t_0,\tilde{x})$, we get
    \begin{align*}
        & (\partial _t -\mathcal{L}_t) v (t,x)  =\int _{B_{\tilde{R}}(\tilde{x})^c} \frac{p_\Omega(t,z,y_0)}{|z-x|^{d+2s}} \intd z \\
        &\quad \leq c\int _{\{ \tilde{R} \leq|z-\tilde{x}|  , ~ t^\frac{1}{2s} \leq |z-y_0|\}} \frac{t}{|z-x|^{d+2s}|z-y_0|^{d+2s}} \intd z
         +c\int _{\{ \tilde{R} \leq |z-\tilde{x}|, ~  |z-y_0| \leq t^\frac{1}{2s} \}} \frac{t^{-\frac{d}{2s}}}{|z-x|^{d+2s}} \intd z\\
        &\quad = I_1 +I_2 .
    \end{align*}
    Note that
    \begin{align*}
        I_2 \leq c \int _{\{  |z-y_0| \leq t^\frac{1}{2s} \}} \frac{t^{-\frac{d}{2s}}}{|z-x|^{d+2s}} \intd z \leq c | \{  B_{t^\frac{1}{2s}} (y_0)\} |  \frac{t^{-\frac{d}{2s}}}{R^{d+2s}} \leq c R^{-d-2s},
    \end{align*}
    where we have used that 
    \begin{equation*}
    |z-x| \geq |x_0-y_0| - |x_0 -\tilde{x}| - | \tilde{x} -x| -|z -y_0| \geq R- R/16-R/8-R/4 \geq 9/16R    
    \end{equation*}
    for all $z \in B_{t^\frac{1}{2s}}(y_0)\subset B_{\tilde{R}} (y_0)$.
    For $I_1$, we split $\R^d $ into the two half-spaces $A:= \{ z\in \R^d \mid |z-y_0| \leq |z-x_0| \}$ and $B = \R^d \setminus A$.
    Then, we use that $z\in A$ implies
    \begin{align*}
        & |z-\tilde{x}| \geq |z-x_0| -|x_0-\tilde{x}| \geq R/2 -\tilde{R}/4 \geq R/2 -R/16 \geq 7/16R \geq \tilde{R} ,\\
        & |z-x|\geq |z-\tilde{x}| - |x-\tilde{x}| \geq 7/16R-\tilde{R}/2\geq 7/16R-R/8 \geq 5/16 R,
    \end{align*}
    and $z\in B$ implies
    \begin{align*}
         |z-y_0| \geq R/2 \geq 2\tilde{R}\geq 2 t^\frac{1}{2s} ,
    \end{align*}
    which together with \eqref{eq:zMinusXComparedToZMinusTildeX} yields 
    \begin{align*}
        I_1 &\leq c\int _{A \cap \{ t^\frac{1}{2s} \leq |z-y_0| \}}  \frac{t}{R^{d+2s}|z-y_0|^{d+2s}} \intd z   + 
        c\int _{B\cap \{ \tilde{R} \leq|z-\tilde{x}| \}}  \frac{t}{ |z-\tilde{x}|^{d+2s}R^{d+2s}} \intd z\\
         &\leq c R^{-d-2s} + c R^{-d-2s}t \tilde{R}^{-2s} \leq c R^{-d-2s}.
    \end{align*}    
    This proves the claim from \eqref{eq:ClaimBoundOperatorOnV}.

    Note that by applying \eqref{eq:UpperBoundGlobalHeatKernelInStep1}, we also get
    \begin{equation*}
        \| v \|_{L^\infty (I^\ominus _{\tilde{R}/2}(t_0) \times \R^d)} = \| p_\Omega (\cdot , \cdot , y_0 )\| _{L^\infty (I^\ominus _{\tilde{R}/2}(t_0) \times B_{\tilde{R}}(\tilde{x}))} \leq c \left( t_0^{-\frac{d}{2s}} \wedge \frac{t_0}{|x_0-y_0|^{d+2s}}  \right) .
    \end{equation*}
    Here, we also used that in case $|x_0 - y_0| \ge 2\tilde{R}$, we have for any $x \in B_{\tilde{R}}(\tilde{x})$,
    \begin{align*}
        |x - y_0| \ge |y_0 - x_0| - |\tilde{x} - x_0| - |\tilde{x} - x| \ge |y_0 -x_0| - \tilde{R}/4 - \tilde{R} \ge 3|y_0 - x_0|/8.
    \end{align*}
    Hence, since $v \equiv 0$ in $I_{\tilde{R}}^{\ominus}(t_0) \times (B_{\tilde{R}}(\tilde{x}) \setminus \Omega)$ by construction, we can use \eqref{eq:BoundRegulAssumUpperHeatKernel}, and recalling $x_0 \in B_{\tilde{R}/4}(\tilde{x})$, we conclude
    \begin{align*}
        p_\Omega(t_0,x_0,y_0) &= v(t_0,x_0) - v(t_0,\tilde{x})
        \leq [v]_{C^{\beta }_x (Q_{\tilde{R}/4}(t_0,\tilde{x}))} d^{\beta} _\Omega (x_0)  \\
        &\leq c_1\frac{d^\beta_\Omega (x_0)}{\tilde{R}^\beta } \left( \| v\| _{L^\infty ( I^\ominus _{\tilde{R}/2} (t_0) \times \R^d )} + \tilde{R}^{2s}\| f\| _{L^\infty (Q_{\tilde{R}/2}(t_0,\tilde{x}))} \right) \\
        & \leq  c \left( \frac{d_\Omega (x_0) }{t_0^\frac{1}{2s}} \right) ^\beta  \left( t_0^{-\frac{d}{2s}} \wedge \frac{t_0}{|x_0-y_0| ^{d+2s}}\right) 
    \end{align*}
    under the assumption $4d_\Omega (x_0) < t_0^{\frac{1}{2s}}$.
    This proves
    \begin{equation*}
        p_\Omega(t_0,x_0,y_0) \leq c \left( 1 \wedge \frac{d_\Omega (x_0)}{t_0^\frac{1}{2s}} \right) ^\beta   \left( t_0^{-\frac{d}{2s}} \wedge\frac{t_0}{|x_0-y_0|^{d+2s}} \right)
    \end{equation*}
    for all $(t_0, x_0,y_0) \in (0,1)\times \Omega \times \Omega$.

    \textbf{Step 3:}
     Since, by \autoref{prop:representation}, $p_\Omega(t_0,x_0,y_0)= p_{0,t_0}^{\Omega}(x_0,y_0) = \hat{p}_{t_0,0}^{\Omega}(y_0,x_0)$ where $(t,y) \mapsto \hat{p}_{t_0,t}^{\Omega}(y,x_0)$ is the dual Dirichlet heat kernel, we can repeat the proof from Step 2 by defining
        \begin{align*}
            \hat{v}(t,y) = \hat{p}_{t_0,t_0-t}^{\Omega}(y,x_0) \ind_{B_{\tilde{R}}(\tilde{y})}(y),
        \end{align*}
        where $\tilde{y} \in \partial \Omega$ is so that $|y_0 - \tilde{y}| = d_{\Omega}(y_0)$.
        Following the exact same arguments as in Step 2 for $\hat{v}$ instead of $v$, together with \autoref{Prop:PropertiesDirichletHeatKernel}, yields
        \begin{align}  \label{eq:UpperHeatKernelBoundWithoutDx}
    \begin{split}
        p_\Omega(t_0,x_0,y_0)   &\leq c \left( 1 \wedge \frac{d_\Omega (y_0)}{t_0^\frac{1}{2s}} \right) ^\beta   \left( t_0^{-\frac{d}{2s}} \wedge\frac{t_0}{|x_0-y_0|^{d+2s}} \right)
        \end{split}
    \end{align}
    for all $(t_0, x_0,y_0) \in (0,1)\times \Omega \times \Omega$.

    \textbf{Step 4:}
    The results from the first three steps already imply the desired result unless we are in the case $4d_\Omega (x_0) < t_0^{\frac{1}{2s}}$.  Hence, finally, let us assume that we have $4d_\Omega (x_0) < t_0^{\frac{1}{2s}}$. We will repeat the proof of Step 2 using the upper bound from \eqref{eq:UpperHeatKernelBoundWithoutDx} instead of \eqref{eq:UpperBoundGlobalHeatKernelInStep1}.
    As above, we set $\tilde{R}=t_0^{\frac{1}{2s}}$ and choose $\tilde{x}\in \partial \Omega$ so that $|x_0-\tilde{x}|=d_\Omega (x_0)$.
    Furthermore, we define $v(t,x):= p_\Omega (t,x,y_0) \ind _{B_{\tilde{R}}(\tilde{x})} (x)$.
    Then, following the proof in Step 2, using the upper bound from \eqref{eq:UpperHeatKernelBoundWithoutDx} instead of \eqref{eq:UpperBoundGlobalHeatKernelInStep1}, we get
    \begin{equation*}
        f(t,x):=(\partial_t - \mathcal{L}_t) v(t,x) \leq c \tilde{R}^{-2s} \left( 1 \wedge \frac{d_\Omega (y_0)}{t_0^\frac{1}{2s}} \right) ^\beta \left( t_0^{-\frac{d}{2s}} \wedge \frac{t_0}{|x_0-y_0| ^{d+2s}}\right) 
    \end{equation*}
    for all $(t,x) \in Q_{\tilde{R}/2}(t_0,\tilde{x})$, as well as
    \begin{equation*}
         \| v \|_{L^\infty (I^\ominus _{\tilde{R}/2}(t_0) \times \R^d)}  \leq c  \left( 1 \wedge \frac{d_\Omega (y_0)}{t_0^\frac{1}{2s}} \right) ^\beta\left( t_0^{-\frac{d}{2s}} \wedge \frac{t_0}{|x-y_0|^{d+2s}}  \right) .
    \end{equation*}
    Hence, using \eqref{eq:BoundRegulAssumUpperHeatKernel}, we conclude
    \begin{align*}
        p_\Omega(t_0,x_0,y_0) &\leq [v]_{C^{\beta }_x (Q_{\tilde{R}/4}(t_0,\tilde{x}))} d^{\beta} _\Omega (x_0)  \\
        & \leq c_1\frac{d^\beta_\Omega (x_0)}{\tilde{R}^\beta } \left( \| v\| _{L^\infty ( I^\ominus _{\tilde{R}/2} (t_0) \times \R^d )} + \tilde{R}^{2s}\| f\| _{L^\infty (Q_{\tilde{R}/2}(t_0,\tilde{x}))} \right) \\
        & \leq  c \left( \frac{d_\Omega (x_0) }{t_0^\frac{1}{2s}} \right) ^\beta \left( 1 \wedge \frac{d_\Omega (y_0)}{t_0^\frac{1}{2s}} \right) ^\beta  \left( t_0^{-\frac{d}{2s}} \wedge \frac{t_0}{|x_0-y_0| ^{d+2s}}\right) 
    \end{align*}
    under the assumption $4d_\Omega (x_0) < t_0^{\frac{1}{2s}}$, which proves \eqref{eq:UpperHeatKernelBoundWithBeta}.
\end{proof}

\begin{proof}[Proof of the upper bound in \autoref{Thm:HeatKernelBounds}]
The boundary regularity estimate \eqref{eq:BoundRegulAssumUpperHeatKernel} for $\mathfrak{L}_t=\mathcal{L}_t$ and $\mathfrak{L}_t=\mathcal{L}_{t_0-t}$ with $\beta = s$ was proved in \autoref{Thm:CsBoundaryRegAndHopfLemma}. Hence, the result follows immediately by application of \autoref{Lem:UpperHeatKernelBounds}.  
\end{proof}

\subsection{Lower Dirichlet heat kernel bounds}\label{sec:LowerHeatKernelBounds}
\begin{Lemma} \label{Lem:LowerHeatKernelBounds}
    Let $\Omega \subset \R^d$ be a bounded $C^{1,\alpha}$ domain and let $\mathcal{L}_t$ be an operator of the form \eqref{eq:OperatorDivergenceForm}-\eqref{eq:KernelDivergenceForm} satisfying \eqref{eq:KernelHoelderCont} with $\mathcal{A}= (0,1) \times \Omega '$ for some $\Omega \Subset \Omega ' \subset \R^d$.
    Then, for all $x,y\in \Omega$ and $t\in (0,1)$, the following lower bound holds
    \begin{equation}\label{eq:LowerHeatKernelBound}
        p_\Omega (t,x,y) \geq c \left( 1 \wedge \frac{d_\Omega (x)}{t^\frac{1}{2s}} \right) ^s  \left( 1 \wedge \frac{d_\Omega (y)}{t^\frac{1}{2s}} \right) ^s \left( t^{-\frac{d}{2s}} \wedge\frac{t}{|x-y|^{d+2s}} \right)
    \end{equation}
    for some constant $c=c(d,s, \sigma , \lambda ,\Lambda , \alpha , \Omega , \Omega ' )>0$.
\end{Lemma}
\begin{proof}
    Fix some $t_0 \in (0,1)$, and $x_0 ,y_0\in \Omega$. We will show \eqref{eq:LowerHeatKernelBound} for $(t,x,y)=(t_0,x_0,y_0)$.

    \textbf{Case 1:}
    If $\min \{ d_\Omega(x_0) , d_\Omega (y_0)\} \geq t_0^{\frac{1}{2s}}/4$, then \eqref{eq:LowerHeatKernelBound} follows from \autoref{Prop:PropertiesDirichletHeatKernel}.

    \textbf{Case 2:}
    Now, let us assume that $d_\Omega (x_0)< t_0^{\frac{1}{2s}}/4$ and $d_\Omega (y_0) \geq t_0^{\frac{1}{2s}}/4$.
    We set $\tilde{R}=t_0^{\frac{1}{2s}}$ and choose $\tilde{x}\in \partial \Omega$ so that $|x_0-\tilde{x}|=d_\Omega (x_0)$.
    Let $u$ be the solution to
    \begin{equation} \label{eq:BarrierForLowerHeatKernelEst}
    \left\{ \begin{aligned}
          \partial _t u -\mathcal{L}_tu & = 0 & &\text{in} \quad Q^\Omega _{3\tilde{R}/8} (t_0,\tilde{x}) , \\
          u & =\ind _{B_{\kappa \tilde{R}}(z)} & & \text{in} \quad ( I^\ominus _{\tilde{R}} (t_0) \times  \R^d ) \setminus Q^\Omega _{3\tilde{R}/8 }(t_0,\tilde{x})  ,   
        \end{aligned} \right.
    \end{equation}
    where we choose some unit vector $\nu =\nu (\tilde{x},\Omega ) \in \Sph ^{d-1}$ and $\kappa = \kappa (\Omega ) >0$ and set $z=\tilde{x}+\frac{7\nu}{16}\tilde{R}$ such that $B_{4\kappa \tilde{R}}(z) \subset  \Omega \setminus B _{3\tilde{R}/8}(\tilde{x})$.

    Next, we claim that
    \begin{equation}\label{eq:RescaledHopfLemmaOnBarrier}
        u \geq c_1 \tilde{R}^{-s} d_\Omega ^s \quad \text {in}\quad Q^\Omega_{\tilde{R}/4} (t_0 ,\tilde{x})
    \end{equation}
    for some constant $c_1=c_1(d,s,\lambda,\Lambda,\alpha , \sigma , \Omega)>0$.
    First, after a shift and a rotation, we may assume that $\tilde{x}=0$ and $\nu =e_d$.
    Then, the rescaled function $u_{\tilde{R}}(t,x)=u(\tilde{R} ^{2s}t, \tilde{R}x)$ satisfies
    \begin{equation} \label{eq:BarrierForLowerHeatKernelEstRescaled}
    \left\{ \begin{aligned}
          \partial _t u_{\tilde{R}} -\mathcal{L}^{\tilde{R}}_tu_{\tilde{R}} & = 0 & &\text{in} \quad Q^{\tilde{R} ^{-1}\Omega } _{3/8} (1,0) , \\
          u_{\tilde{R}} & =\ind _{B_{\kappa }(7e_d/16)} & & \text{in} \quad ( I^\ominus _{1} (1) \times  \R^d ) \setminus Q^{\tilde{R} ^{-1}\Omega} _{3/8 }(1,0)  ,   
        \end{aligned} \right.
    \end{equation}
    where $\mathcal{L}^{\tilde{R}}_t$ is the operator with rescaled kernel $K^{\tilde{R}}(t,x,y)=K(\tilde{R}^{2s}t,\tilde{R}x ,\tilde{R}y)$.
    Note that $K^{\tilde{R}}$ also satisfies \eqref{eq:KernelDivergenceForm} and \eqref{eq:KernelHoelderCont} with the same constants.    
    By the Hopf lemma (\autoref{Thm:CsBoundaryRegAndHopfLemma}(ii)), we obtain $u_{\tilde{R}}\geq c_1 d_{\tilde{R}^{-1}\Omega} ^s$ in $Q^{\tilde{R} ^{-1}\Omega} _{1/4 }(1,0)$ for some constant $c_1 = c_1(d,s,\sigma ,\lambda ,\Lambda, \alpha ,\tilde{R}^{-1}\Omega ,u_{\tilde{R}} )>0$ which, a priori, depends on $\tilde{R}.$
    We aim to show that $c_1$ does not depend on $\tilde{R}$, since then
    \begin{equation*}
        u(t,x) = u_{\tilde{R}}(\tilde{R}^{-2s}t,\tilde{R}^{-1}x)\geq c_1 d^s _{\tilde{R}^{-1}\Omega} (\tilde{R}^{-1}x) = c_1 \tilde{R}^{-s} d_\Omega ^s(x), \quad \forall (t,x) \in Q^\Omega_{\tilde{R}/4} (1,0),
    \end{equation*}
     which would imply \eqref{eq:RescaledHopfLemmaOnBarrier}.
     First, note that the dependence of $c_1$ on the domain $\tilde{R}^{-1}\Omega$ is only through an upper bound on the $C^{1,\alpha}$ norm of the boundary, which can be chosen to be independent of $\tilde{R}$ since $\tilde{R}\leq 1$.
    The dependence of $c_1$ on $u_{\tilde{R}}$ can be made precise as follows (see \autoref{Rem:HopfDependenceOnUAndF}).
    There exists some $\eps = \eps (d,s,\lambda, \Lambda , \Omega )>0$ such that $c_1$ depends on $u_{\tilde{R}}$ only through a lower bound $0<\tilde{c} \leq \inf _{Q^{\tilde{R}^{-1}\Omega} _{1/4+\eps}\cap \{ d_{\tilde{R}^{-1}\Omega} \geq \eps \}}u _{\tilde{R}} $.
    Note that we can choose such a $\tilde{c}$ depending only on $d,s,\lambda,\Lambda,\alpha , \sigma  $, and $\Omega$, for example by using the representation formula for solutions and known lower bounds on the Dirichlet heat kernel in the interior (see \autoref{Prop:PropertiesDirichletHeatKernel}).
    This proves \eqref{eq:RescaledHopfLemmaOnBarrier}.
    
    Next, we claim that
    \begin{equation} \label{eq:LowerBoundOnHeatKernelOnBallBz}
        p_\Omega(t,x,y_0) \geq c_2 \left( t_0^{-\frac{d}{2s}} \wedge\frac{t_0}{|x_0-y_0|^{d+2s}} \right) \quad \text{for all}\quad (t,x) \in I^\ominus _{3\tilde{R}/8}(t_0) \times B_{\kappa \tilde{R}}(z).
    \end{equation}
    To prove \eqref{eq:LowerBoundOnHeatKernelOnBallBz}, first, assume $|x_0-y_0|\leq \tilde{R}$. Then, 
    \begin{equation*}
        |x-y_0| \leq |x-x_0| + |x_0-y_0| \leq c \tilde{R} +\tilde{R} \leq c \tilde{R} \leq c t^{\frac{1}{2s}} ,
    \end{equation*}
    hence, by the lower bound from \autoref{Prop:PropertiesDirichletHeatKernel}, we get $p_\Omega (t,x,y_0) \geq  c t ^{-\frac{d}{2s}} \geq c t_0^{-\frac{d}{2s}}$, i.e., \eqref{eq:LowerBoundOnHeatKernelOnBallBz} follows in this case.
    Now, if we assume $|x_0-y_0|\geq \tilde{R}$, then
    \begin{equation*}
        |x-y_0| \geq |x_0-y_0| - |x_0-\tilde{x}| - |\tilde{x}-x| \geq \tilde{R} - \tilde{R}/4 -\tilde{R}/2= \tilde{R}/4 \geq c t^{\frac{1}{2s}},
    \end{equation*}
    and therefore, again by the lower bound from \autoref{Prop:PropertiesDirichletHeatKernel}, we get
    \begin{equation*}
        p_\Omega(t,x,y_0)\geq c \frac{t}{|x-y_0|^{d+2s}} \geq c \frac{t_0}{|x_0-y_0|^{d+2s}} ,
    \end{equation*}
    where in the last inequality we have used $t \asymp t_0$ and $|x-y_0|\leq |x_0-y_0| + |x-\tilde{x}|+|\tilde{x}-x_0| \leq |x_0-y_0| +c \tilde{R} \leq c |x_0-y_0|$.    
    This proves \eqref{eq:LowerBoundOnHeatKernelOnBallBz}.

    Hence, by using the maximum principle for
    \begin{equation*}
        p_\Omega (t,x,y_0)-c_2 \left( t_0^{-\frac{d}{2s}} \wedge\frac{t_0}{|x_0-y_0|^{d+2s}} \right) u \quad \text{on} \quad  Q^\Omega _{3\tilde{R}/8} (t_0,\tilde{x}) ,
    \end{equation*}
    we obtain together with \eqref{eq:RescaledHopfLemmaOnBarrier} and \eqref{eq:LowerBoundOnHeatKernelOnBallBz}
    \begin{equation*}
        p_\Omega(t_0,x_0,y_0 )  \geq c_1 c_2 \left( t_0^{-\frac{d}{2s}} \wedge\frac{t_0}{|x_0-y_0|^{d+2s}} \right) \left( \frac{d_\Omega(x_0)}{t_0^{\frac{1}{2s}}} \right) ^s ,
    \end{equation*}
    which proves \eqref{eq:LowerHeatKernelBound} in the case $d_\Omega (x_0)< t_0^{\frac{1}{2s}}/4$ and $d_\Omega (y_0) \geq t_0^{\frac{1}{2s}}/4$.

    \textbf{Case 3:}
    Now, let us assume that $d_\Omega (y_0)< t_0^{\frac{1}{2s}}/4$ and $d_\Omega (x_0) \geq t_0^{\frac{1}{2s}}/4$.
    Then, \eqref{eq:LowerHeatKernelBound} follows by the same arguments as in Case 2 by using the dual Dirichlet heat kernel (compare also with Step 3 in the proof of \autoref{Lem:UpperHeatKernelBounds}).

    \textbf{Case 4:}
    Finally, let us assume $d_\Omega (x_0)< t_0^{\frac{1}{2s}}/4$ and $d_\Omega (y_0)< t_0^{\frac{1}{2s}}/4$.
    We again set $\tilde{R}=t_0^{\frac{1}{2s}}$ and choose $\tilde{x}\in \partial \Omega$ so that $|x_0-\tilde{x}|=d_\Omega (x_0)$.
    Let $u$ be the solution to \eqref{eq:BarrierForLowerHeatKernelEst}, which satisfies the bound \eqref{eq:RescaledHopfLemmaOnBarrier}.
    Next, from Case 3, we get
    \begin{equation}\label{eq:LowerBoundOnHeatKernelOnBzForCase4}
         p_\Omega(t,x,y_0) \geq c \left( \frac{d_\Omega(y_0)}{t_0^{\frac{1}{2s}}} \right) ^s \left( t_0^{-\frac{d}{2s}} \wedge\frac{t_0}{|x_0-y_0|^{d+2s}} \right) \qquad\forall (t,x) \in I^\ominus _{3\tilde{R}/8}(t_0) \times B_{\kappa \tilde{R}}(z),
    \end{equation}
    where we have also used that $B_{4\kappa\tilde{R}}(z) \subset  \Omega$.
    Hence, using \eqref{eq:RescaledHopfLemmaOnBarrier} and \eqref{eq:LowerBoundOnHeatKernelOnBzForCase4} together with the maximum principle proves \eqref{eq:LowerHeatKernelBound} in this case.
\end{proof}

\begin{proof}[Proof of the lower bound in \autoref{Thm:HeatKernelBounds}]
The result follows immediately from \autoref{Lem:LowerHeatKernelBounds}.
\end{proof}

	
\printbibliography

@Book{RosOton2024Book,
  author    = {Fernández-Real, Xavier and Ros-Oton, Xavier},
  publisher = {Springer Nature Switzerland},
  title     = {Integro-Differential Elliptic Equations},
  year      = {2024},
  isbn      = {9783031542428},
  doi       = {10.1007/978-3-031-54242-8},
  issn      = {2296-505X},
  journal   = {Progress in Mathematics},
}

@Article{RosOton2014,
  author    = {Ros-Oton, Xavier and Serra, Joaquim},
  journal   = {Journal de Mathématiques Pures et Appliquées},
  title     = {The Dirichlet problem for the fractional Laplacian: Regularity up to the boundary},
  year      = {2014},
  issn      = {0021-7824},
  month     = mar,
  number    = {3},
  pages     = {275--302},
  volume    = {101},
  doi       = {10.1016/j.matpur.2013.06.003},
  publisher = {Elsevier BV},
}

@Article{Kim2023,
  author    = {Kim, Panki and Song, Renming and Vondraček, Zoran},
  journal   = {Journal of the European Mathematical Society},
  title     = {Sharp two-sided Green function estimates for Dirichlet forms degenerate at the boundary},
  year      = {2023},
  issn      = {1435-9863},
  month     = jan,
  number    = {6},
  pages     = {2249--2300},
  volume    = {26},
  doi       = {10.4171/jems/1322},
  publisher = {European Mathematical Society - EMS - Publishing House GmbH},
}

@Article{Bogdan1997,
  author    = {Bogdan, Krzysztof},
  journal   = {Studia Mathematica},
  title     = {The boundary Harnack principle for the fractional Laplacian},
  year      = {1997},
  issn      = {1730-6337},
  number    = {1},
  pages     = {43--80},
  volume    = {123},
  doi       = {10.4064/sm-123-1-43-80},
  publisher = {Institute of Mathematics, Polish Academy of Sciences},
}

@Article{RosOton2016a,
  author    = {Ros-Oton, Xavier and Serra, Joaquim},
  journal   = {Duke Mathematical Journal},
  title     = {Boundary regularity for fully nonlinear integro-differential equations},
  year      = {2016},
  issn      = {0012-7094},
  month     = aug,
  number    = {11},
  volume    = {165},
  doi       = {10.1215/00127094-3476700},
  publisher = {Duke University Press},
}

@Article{RosOton2016b,
  author    = {Ros-Oton, Xavier and Serra, Joaquim},
  journal   = {Journal of Differential Equations},
  title     = {Regularity theory for general stable operators},
  year      = {2016},
  issn      = {0022-0396},
  month     = jun,
  number    = {12},
  pages     = {8675--8715},
  volume    = {260},
  doi       = {10.1016/j.jde.2016.02.033},
  publisher = {Elsevier BV},
}

@Article{RosOton2017,
  author    = {Ros-Oton, Xavier and Serra, Joaquim},
  journal   = {Annali di Matematica Pura ed Applicata (1923 -)},
  title     = {Boundary regularity estimates for nonlocal elliptic equations in $C^1$ and $C^{1,\alpha }$ domains},
  year      = {2017},
  issn      = {1618-1891},
  month     = feb,
  number    = {5},
  pages     = {1637--1668},
  volume    = {196},
  doi       = {10.1007/s10231-016-0632-1},
  publisher = {Springer Science and Business Media LLC},
}

@Article{Grubb2014,
  author    = {Grubb, Gerd},
  journal   = {Analysis \& PDE},
  title     = {Local and nonlocal boundary conditions for $\mu$-transmission and fractional elliptic pseudodifferential operators},
  year      = {2014},
  issn      = {2157-5045},
  month     = dec,
  number    = {7},
  pages     = {1649--1682},
  volume    = {7},
  doi       = {10.2140/apde.2014.7.1649},
  publisher = {Mathematical Sciences Publishers},
}

@Article{Grubb2015,
  author    = {Grubb, Gerd},
  journal   = {Advances in Mathematics},
  title     = {Fractional Laplacians on domains, a development of Hörmander’s theory of $\mu$-transmission pseudodifferential operators},
  year      = {2015},
  issn      = {0001-8708},
  month     = jan,
  pages     = {478--528},
  volume    = {268},
  doi       = {10.1016/j.aim.2014.09.018},
  publisher = {Elsevier BV},
}

@Article{Grube2024,
  author        = {Grube, Florian},
  title         = {Boundary regularity and Hopf lemma for nondegenerate stable operators},
  year          = {2024},
  month         = oct,
  archiveprefix = {arXiv},
  copyright     = {arXiv.org perpetual, non-exclusive license},
  doi           = {10.48550/ARXIV.2410.00829},
  eprint        = {2410.00829},
  primaryclass  = {math.AP},
  publisher     = {arXiv},
}

@Article{RosOton2024,
  author        = {Ros-Oton, Xavier and Weidner, Marvin},
  title         = {Optimal regularity for nonlocal elliptic equations and free boundary problems},
  year          = {2024},
  month         = mar,
  archiveprefix = {arXiv},
  copyright     = {arXiv.org perpetual, non-exclusive license},
  doi           = {10.48550/ARXIV.2403.07793},
  eprint        = {2403.07793},
  primaryclass  = {math.AP},
  publisher     = {arXiv},
}

@Article{Kim2024,
  author        = {Kim, Minhyun and Weidner, Marvin},
  journal       = {J. Eur. Math. Soc.},
  title         = {Optimal boundary regularity and Green function estimates for nonlocal equations in divergence form},
  year          = {2024},
  month         = aug,
  note          = {To appear},
  archiveprefix = {arXiv},
  copyright     = {Creative Commons Attribution 4.0 International},
  doi           = {10.48550/ARXIV.2408.12987},
  eprint        = {2408.12987},
  primaryclass  = {math.AP},
  publisher     = {arXiv},
}

@Article{Abatangelo2020,
  author    = {Abatangelo, Nicola and Ros-Oton, Xavier},
  journal   = {Advances in Mathematics},
  title     = {Obstacle problems for integro-differential operators: Higher regularity of free boundaries},
  year      = {2020},
  issn      = {0001-8708},
  month     = jan,
  pages     = {106931},
  volume    = {360},
  doi       = {10.1016/j.aim.2019.106931},
  publisher = {Elsevier BV},
}

@Article{FernandezReal2017,
  author    = {Fernández-Real, Xavier and Ros-Oton, Xavier},
  journal   = {Journal of Functional Analysis},
  title     = {Regularity theory for general stable operators: Parabolic equations},
  year      = {2017},
  issn      = {0022-1236},
  month     = may,
  number    = {10},
  pages     = {4165--4221},
  volume    = {272},
  doi       = {10.1016/j.jfa.2017.02.015},
  publisher = {Elsevier BV},
}

@Article{Cho2025,
  author        = {Cho, Soobin and Song, Renming},
  title         = {Approximate factorizations for non-symmetric jump processes},
  year          = {2025},
  month         = apr,
  archiveprefix = {arXiv},
  copyright     = {Creative Commons Attribution Non Commercial No Derivatives 4.0 International},
  doi           = {10.48550/ARXIV.2504.14763},
  eprint        = {2504.14763},
  primaryclass  = {math.PR},
  publisher     = {arXiv},
}

@Article{Byun2023,
  author    = {Byun, Sun-Sig and Kim, Hyojin and Kim, Kyeongbae},
  journal   = {Journal of Evolution Equations},
  title     = {Higher Hölder regularity for nonlocal parabolic equations with irregular kernels},
  year      = {2023},
  issn      = {1424-3202},
  month     = jul,
  number    = {3},
  volume    = {23},
  doi       = {10.1007/s00028-023-00901-2},
  publisher = {Springer Science and Business Media LLC},
}

@Book{Giaquinta2012,
  author    = {Giaquinta, Mariano and Martinazzi, Luca},
  publisher = {Scuola Normale Superiore},
  title     = {An Introduction to the Regularity Theory for Elliptic Systems, Harmonic Maps and Minimal Graphs},
  year      = {2012},
  isbn      = {9788876424434},
  doi       = {10.1007/978-88-7642-443-4},
}

@Article{Kassmann2024,
  author    = {Kassmann, Moritz and Weidner, Marvin},
  journal   = {Duke Mathematical Journal},
  title     = {The parabolic Harnack inequality for nonlocal equations},
  year      = {2024},
  issn      = {0012-7094},
  month     = nov,
  number    = {17},
  volume    = {173},
  doi       = {10.1215/00127094-2024-0008},
  publisher = {Duke University Press},
}

@Article{Liao2023,
  author    = {Liao, Naian},
  journal   = {Calculus of Variations and Partial Differential Equations},
  title     = {Hölder regularity for parabolic fractional p-Laplacian},
  year      = {2023},
  issn      = {1432-0835},
  month     = dec,
  number    = {1},
  volume    = {63},
  doi       = {10.1007/s00526-023-02627-y},
  publisher = {Springer Science and Business Media LLC},
}

@Article{Byun2023a,
  author    = {Byun, Sun-Sig and Kim, Kyeongbae},
  journal   = {Annali di Matematica Pura ed Applicata (1923 -)},
  title     = {A Hölder estimate with an optimal tail for nonlocal parabolic p-Laplace equations},
  year      = {2023},
  issn      = {1618-1891},
  month     = jul,
  number    = {1},
  pages     = {109--147},
  volume    = {203},
  doi       = {10.1007/s10231-023-01355-6},
  publisher = {Springer Science and Business Media LLC},
}

@Article{Byun2025,
  author    = {Byun, Sun-Sig and Kim, Kyeongbae and Kumar, Deepak},
  journal   = {Journal of Differential Equations},
  title     = {Global Calderón-Zygmund theory for fractional Laplacian type equations},
  year      = {2025},
  issn      = {0022-0396},
  month     = aug,
  pages     = {113319},
  volume    = {436},
  doi       = {10.1016/j.jde.2025.113319},
  publisher = {Elsevier BV},
}

@Article{Liao2025,
  author        = {Liao, Naian and Weidner, Marvin},
  title         = {Nonnegative solutions to nonlocal parabolic equations},
  year          = {2025},
  month         = may,
  archiveprefix = {arXiv},
  copyright     = {Creative Commons Attribution Non Commercial Share Alike 4.0 International},
  doi           = {10.48550/ARXIV.2505.08449},
  eprint        = {2505.08449},
  primaryclass  = {math.AP},
  publisher     = {arXiv},
journal       = {Trans. Amer. Math. Soc.},
  note          = {To appear}
}

@Article{Felsinger2014,
  author    = {Felsinger, Matthieu and Kassmann, Moritz and Voigt, Paul},
  journal   = {Mathematische Zeitschrift},
  title     = {The Dirichlet problem for nonlocal operators},
  year      = {2014},
  issn      = {1432-1823},
  month     = nov,
  number    = {3–4},
  pages     = {779--809},
  volume    = {279},
  doi       = {10.1007/s00209-014-1394-3},
  publisher = {Springer Science and Business Media LLC},
}

@article {FernandezReal2016,
    AUTHOR = {Fern\'andez-Real, Xavier and Ros-Oton, Xavier},
     TITLE = {Boundary regularity for the fractional heat equation},
   JOURNAL = {Rev. R. Acad. Cienc. Exactas F\'is. Nat. Ser. A Mat. RACSAM},
  FJOURNAL = {Revista de la Real Academia de Ciencias Exactas, F\'isicas y
              Naturales. Serie A. Matematicas. RACSAM},
    VOLUME = {110},
      YEAR = {2016},
    NUMBER = {1},
     PAGES = {49--64},
      ISSN = {1578-7303,1579-1505},
   MRCLASS = {35R11 (35B65 47G20)},
  MRNUMBER = {3462074},
       DOI = {10.1007/s13398-015-0218-6},
       URL = {https://doi.org/10.1007/s13398-015-0218-6},
}

@article {Wang2022,
    AUTHOR = {Wang, Pengyan and Chen, Wenxiong},
     TITLE = {Hopf's lemmas for parabolic fractional {$p$}-{L}aplacians},
   JOURNAL = {Commun. Pure Appl. Anal.},
  FJOURNAL = {Communications on Pure and Applied Analysis},
    VOLUME = {21},
      YEAR = {2022},
    NUMBER = {9},
     PAGES = {3055--3069},
      ISSN = {1534-0392,1553-5258},
   MRCLASS = {35K55 (35B09 35B50)},
  MRNUMBER = {4484116},
       DOI = {10.3934/cpaa.2022089},
       URL = {https://doi.org/10.3934/cpaa.2022089},
}

@article {RosOton2018,
    AUTHOR = {Ros-Oton, Xavier and Vivas, Hern\'an},
     TITLE = {Higher-order boundary regularity estimates for nonlocal
              parabolic equations},
   JOURNAL = {Calc. Var. Partial Differential Equations},
  FJOURNAL = {Calculus of Variations and Partial Differential Equations},
    VOLUME = {57},
      YEAR = {2018},
    NUMBER = {5},
     PAGES = {Paper No. 111, 20},
      ISSN = {0944-2669,1432-0835},
   MRCLASS = {35R11 (35B65 47G30 60G52)},
  MRNUMBER = {3828901},
MRREVIEWER = {Nataliya\ Volodymyrivna\ Vasylyeva},
       DOI = {10.1007/s00526-018-1399-6},
       URL = {https://doi.org/10.1007/s00526-018-1399-6},
}

@article {Kukuljan2023,
    AUTHOR = {Kukuljan, Teo},
     TITLE = {{$C^{\{2,\alpha\}}$} regularity of free boundaries in
              parabolic non-local obstacle problems},
   JOURNAL = {Calc. Var. Partial Differential Equations},
  FJOURNAL = {Calculus of Variations and Partial Differential Equations},
    VOLUME = {62},
      YEAR = {2023},
    NUMBER = {2},
     PAGES = {Paper No. 36, 40},
      ISSN = {0944-2669,1432-0835},
   MRCLASS = {35R35 (47G20)},
  MRNUMBER = {4525717},
       DOI = {10.1007/s00526-022-02372-8},
       URL = {https://doi.org/10.1007/s00526-022-02372-8},
}

@incollection {Grubb2018a,
    AUTHOR = {Grubb, Gerd},
     TITLE = {Fractional-order operators: boundary problems, heat equations},
 BOOKTITLE = {Mathematical analysis and applications---plenary lectures},
    SERIES = {Springer Proc. Math. Stat.},
    VOLUME = {262},
     PAGES = {51--81},
 PUBLISHER = {Springer, Cham},
      YEAR = {2018},
   MRCLASS = {35R11 (35J05 35K05 47F05 60G52)},
  MRNUMBER = {3880448},
MRREVIEWER = {Vincenzo\ Ambrosio},
       DOI = {10.1007/978-3-030-00874-1\_2},
       URL = {https://doi.org/10.1007/978-3-030-00874-1_2},
}

@article {Grubb2018b,
    AUTHOR = {Grubb, Gerd},
     TITLE = {Regularity in {$L_p$} {S}obolev spaces of solutions to
              fractional heat equations},
   JOURNAL = {J. Funct. Anal.},
  FJOURNAL = {Journal of Functional Analysis},
    VOLUME = {274},
      YEAR = {2018},
    NUMBER = {9},
     PAGES = {2634--2660},
      ISSN = {0022-1236,1096-0783},
   MRCLASS = {35R11 (35B65 35K20 35K25 47G30 60G52)},
  MRNUMBER = {3771838},
MRREVIEWER = {Nataliya\ Volodymyrivna\ Vasylyeva},
       DOI = {10.1016/j.jfa.2017.12.011},
       URL = {https://doi.org/10.1016/j.jfa.2017.12.011},
}

@article {Jin2015,
    AUTHOR = {Jin, Tianling and Xiong, Jingang},
     TITLE = {Schauder estimates for solutions of linear parabolic
              integro-differential equations},
   JOURNAL = {Discrete Contin. Dyn. Syst.},
  FJOURNAL = {Discrete and Continuous Dynamical Systems. Series A},
    VOLUME = {35},
      YEAR = {2015},
    NUMBER = {12},
     PAGES = {5977--5998},
      ISSN = {1078-0947,1553-5231},
   MRCLASS = {35R09 (35B65 35R11 45K05)},
  MRNUMBER = {3393263},
MRREVIEWER = {Shulin\ Zhou},
       DOI = {10.3934/dcds.2015.35.5977},
       URL = {https://doi.org/10.3934/dcds.2015.35.5977},
}

@misc{Figalli2023,
      title={Regularity theory for nonlocal obstacle problems with critical and subcritical scaling}, 
      author={Alessio Figalli and Xavier Ros-Oton and Joaquim Serra},
      year={2023},
      eprint={2306.16008},
      archivePrefix={arXiv},
      primaryClass={math.AP},
      url={https://arxiv.org/abs/2306.16008}, 
}

@article {Abels2025,
    AUTHOR = {Abels, Helmut and Grubb, Gerd},
     TITLE = {Maximal {$L_p$}-regularity for {$x$}-dependent fractional heat
              equations with {D}irichlet conditions},
   JOURNAL = {Math. Ann.},
  FJOURNAL = {Mathematische Annalen},
    VOLUME = {391},
      YEAR = {2025},
    NUMBER = {3},
     PAGES = {3295--3331},
      ISSN = {0025-5831,1432-1807},
   MRCLASS = {35S15 (35K61 35R11 35S16 47G30)},
  MRNUMBER = {4865217},
MRREVIEWER = {Tilak\ Bhattacharya},
       DOI = {10.1007/s00208-024-02999-2},
       URL = {https://doi.org/10.1007/s00208-024-02999-2},
}

@article {Grubb2019,
    AUTHOR = {Grubb, Gerd},
     TITLE = {Limited regularity of solutions to fractional heat and
              {S}chr\"odinger equations},
   JOURNAL = {Discrete Contin. Dyn. Syst.},
  FJOURNAL = {Discrete and Continuous Dynamical Systems. Series A},
    VOLUME = {39},
      YEAR = {2019},
    NUMBER = {6},
     PAGES = {3609--3634},
      ISSN = {1078-0947,1553-5231},
   MRCLASS = {35R11 (35B65 35S11 47G30 60G52)},
  MRNUMBER = {3959442},
       DOI = {10.3934/dcds.2019148},
       URL = {https://doi.org/10.3934/dcds.2019148},
}

@article {Fall2020,
    AUTHOR = {Fall, Mouhamed Moustapha},
     TITLE = {Regularity results for nonlocal equations and applications},
   JOURNAL = {Calc. Var. Partial Differential Equations},
  FJOURNAL = {Calculus of Variations and Partial Differential Equations},
    VOLUME = {59},
      YEAR = {2020},
    NUMBER = {5},
     PAGES = {Paper No. 181, 53},
      ISSN = {0944-2669,1432-0835},
   MRCLASS = {35J15 (35B65 47G20)},
  MRNUMBER = {4153907},
       DOI = {10.1007/s00526-020-01821-6},
       URL = {https://doi.org/10.1007/s00526-020-01821-6},
}

@article {Dong2024,
    AUTHOR = {Dong, Hongjie and Ryu, Junhee},
     TITLE = {Nonlocal elliptic and parabolic equations with general stable
              operators in weighted {S}obolev spaces},
   JOURNAL = {SIAM J. Math. Anal.},
  FJOURNAL = {SIAM Journal on Mathematical Analysis},
    VOLUME = {56},
      YEAR = {2024},
    NUMBER = {4},
     PAGES = {4623--4661},
      ISSN = {0036-1410,1095-7154},
   MRCLASS = {35B65 (35A01 35R11 45K05 60G52)},
  MRNUMBER = {4767520},
       DOI = {10.1137/23M160061X},
       URL = {https://doi.org/10.1137/23M160061X},
}

@article {Choi2023,
    AUTHOR = {Choi, Jae-Hwan and Kim, Kyeong-Hun and Ryu, Junhee},
     TITLE = {Sobolev regularity theory for the non-local elliptic and
              parabolic equations on {$ C^{1,1} $} open sets},
   JOURNAL = {Discrete Contin. Dyn. Syst.},
  FJOURNAL = {Discrete and Continuous Dynamical Systems. Series A},
    VOLUME = {43},
      YEAR = {2023},
    NUMBER = {9},
     PAGES = {3338--3377},
      ISSN = {1078-0947,1553-5231},
   MRCLASS = {35S16 (35B65 46E35)},
  MRNUMBER = {4613799},
       DOI = {10.3934/dcds.2023050},
       URL = {https://doi.org/10.3934/dcds.2023050},
}

@article {Armstrong2025,
    AUTHOR = {Armstrong, Gavin and Bogdan, Krzysztof and Rutkowski, Artur},
     TITLE = {Caloric functions and boundary regularity for the fractional
              {L}aplacian in {L}ipschitz open sets},
   JOURNAL = {Math. Ann.},
  FJOURNAL = {Mathematische Annalen},
    VOLUME = {391},
      YEAR = {2025},
    NUMBER = {1},
     PAGES = {1199--1252},
      ISSN = {0025-5831,1432-1807},
   MRCLASS = {35S16 (35C15 60J50)},
  MRNUMBER = {4846809},
       DOI = {10.1007/s00208-024-02931-8},
       URL = {https://doi.org/10.1007/s00208-024-02931-8},
}

@article {Biccari2017,
    AUTHOR = {Biccari, Umberto and Warma, Mahamadi and Zuazua, Enrique},
     TITLE = {Local elliptic regularity for the {D}irichlet fractional
              {L}aplacian},
   JOURNAL = {Adv. Nonlinear Stud.},
  FJOURNAL = {Advanced Nonlinear Studies},
    VOLUME = {17},
      YEAR = {2017},
    NUMBER = {2},
     PAGES = {387--409},
      ISSN = {1536-1365,2169-0375},
   MRCLASS = {35R11 (35B65 35S05)},
  MRNUMBER = {3641649},
       DOI = {10.1515/ans-2017-0014},
       URL = {https://doi.org/10.1515/ans-2017-0014},
}

@misc{Abdellaoui2025,
      title={Global regularity results for the fractional heat equation and application to a class of non-linear KPZ problems}, 
      author={Boumediene Abdellaoui and Somia Atmani and Kheireddine Biroud and El-Haj Laamri},
      year={2025},
      eprint={2506.06875},
      archivePrefix={arXiv},
      primaryClass={math.AP},
      url={https://arxiv.org/abs/2506.06875}, 
}

@article {Liao2024b,
    AUTHOR = {Liao, Naian},
     TITLE = {On the modulus of continuity of solutions to nonlocal
              parabolic equations},
   JOURNAL = {J. Lond. Math. Soc. (2)},
  FJOURNAL = {Journal of the London Mathematical Society. Second Series},
    VOLUME = {110},
      YEAR = {2024},
    NUMBER = {3},
     PAGES = {Paper No. e12985, 30},
      ISSN = {0024-6107,1469-7750},
   MRCLASS = {35B65 (35K65 35R09 47G20)},
  MRNUMBER = {4796846},
MRREVIEWER = {Tran\ Ngoc\ Thach},
       DOI = {10.1112/jlms.12985},
       URL = {https://doi.org/10.1112/jlms.12985},
}

@article {Diening2025,
    AUTHOR = {Diening, Lars and Nowak, Simon},
     TITLE = {Calder\'on-{Z}ygmund estimates for the fractional
              {$p$}-{L}aplacian},
   JOURNAL = {Ann. PDE},
  FJOURNAL = {Annals of PDE. Journal Dedicated to the Analysis of Problems
              from Physical Sciences},
    VOLUME = {11},
      YEAR = {2025},
    NUMBER = {1},
     PAGES = {Paper No. 6, 33},
      ISSN = {2524-5317,2199-2576},
   MRCLASS = {35R09 (35B65 35D30 35R11 35S15 47G20)},
  MRNUMBER = {4854881},
MRREVIEWER = {Shulin\ Zhou},
       DOI = {10.1007/s40818-025-00196-1},
       URL = {https://doi.org/10.1007/s40818-025-00196-1},
}

@article {RosOton2019,
    AUTHOR = {Ros-Oton, Xavier and Serra, Joaquim},
     TITLE = {The boundary {H}arnack principle for nonlocal elliptic
              operators in non-divergence form},
   JOURNAL = {Potential Anal.},
  FJOURNAL = {Potential Analysis. An International Journal Devoted to the
              Interactions between Potential Theory, Probability Theory,
              Geometry and Functional Analysis},
    VOLUME = {51},
      YEAR = {2019},
    NUMBER = {3},
     PAGES = {315--331},
      ISSN = {0926-2601,1572-929X},
   MRCLASS = {35R11 (35B45 35B51 35B65 47G20)},
  MRNUMBER = {4023466},
       DOI = {10.1007/s11118-018-9713-7},
       URL = {https://doi.org/10.1007/s11118-018-9713-7},
}

@article {Caffarelli2018,
    AUTHOR = {Caffarelli, Luis A. and Sire, Yannick},
     TITLE = {Bounds on the {G}reen function for integral operators and
              fractional harmonic measure with applications to boundary
              {H}arnack},
   JOURNAL = {Proc. Amer. Math. Soc.},
  FJOURNAL = {Proceedings of the American Mathematical Society},
    VOLUME = {146},
      YEAR = {2018},
    NUMBER = {3},
     PAGES = {1207--1216},
      ISSN = {0002-9939,1088-6826},
   MRCLASS = {35R11 (35A08 35B45 35D30 35J08)},
  MRNUMBER = {3750233},
MRREVIEWER = {Kai\ Diethelm},
       DOI = {10.1090/proc/13815},
       URL = {https://doi.org/10.1090/proc/13815},
}

@article {Felsinger2013,
    AUTHOR = {Felsinger, Matthieu and Kassmann, Moritz},
     TITLE = {Local regularity for parabolic nonlocal operators},
   JOURNAL = {Comm. Partial Differential Equations},
  FJOURNAL = {Communications in Partial Differential Equations},
    VOLUME = {38},
      YEAR = {2013},
    NUMBER = {9},
     PAGES = {1539--1573},
      ISSN = {0360-5302,1532-4133},
   MRCLASS = {35R09 (35B45 35B65 35D30 35K20)},
  MRNUMBER = {3169755},
MRREVIEWER = {Pedro\ Mar\'in Rubio},
       DOI = {10.1080/03605302.2013.808211},
       URL = {https://doi.org/10.1080/03605302.2013.808211},
}

@article {Diening2025b,
    AUTHOR = {Diening, Lars and Kim, Kyeongbae and Lee, Ho-Sik and Nowak,
              Simon},
     TITLE = {Gradient estimates for parabolic nonlinear nonlocal equations},
   JOURNAL = {Calc. Var. Partial Differential Equations},
  FJOURNAL = {Calculus of Variations and Partial Differential Equations},
    VOLUME = {64},
      YEAR = {2025},
    NUMBER = {3},
     PAGES = {Paper No. 98, 86},
      ISSN = {0944-2669,1432-0835},
   MRCLASS = {35R09 (35B65 35K55 47G20)},
  MRNUMBER = {4874974},
       DOI = {10.1007/s00526-025-02957-z},
       URL = {https://doi.org/10.1007/s00526-025-02957-z},
}

@article {Serra2015,
    AUTHOR = {Serra, Joaquim},
     TITLE = {Regularity for fully nonlinear nonlocal parabolic equations
              with rough kernels},
   JOURNAL = {Calc. Var. Partial Differential Equations},
  FJOURNAL = {Calculus of Variations and Partial Differential Equations},
    VOLUME = {54},
      YEAR = {2015},
    NUMBER = {1},
     PAGES = {615--629},
      ISSN = {0944-2669,1432-0835},
   MRCLASS = {35K65 (35B65 45K05)},
  MRNUMBER = {3385173},
       DOI = {10.1007/s00526-014-0798-6},
       URL = {https://doi.org/10.1007/s00526-014-0798-6},
}

@article {Caffarelli2011,
    AUTHOR = {Caffarelli, Luis and Chan, Chi Hin and Vasseur, Alexis},
     TITLE = {Regularity theory for parabolic nonlinear integral operators},
   JOURNAL = {J. Amer. Math. Soc.},
  FJOURNAL = {Journal of the American Mathematical Society},
    VOLUME = {24},
      YEAR = {2011},
    NUMBER = {3},
     PAGES = {849--869},
      ISSN = {0894-0347,1088-6834},
   MRCLASS = {45P05 (35B65 35R11 45K05 47G10)},
  MRNUMBER = {2784330},
MRREVIEWER = {Giovanni\ Anello},
       DOI = {10.1090/S0894-0347-2011-00698-X},
       URL = {https://doi.org/10.1090/S0894-0347-2011-00698-X},
}

@article {Lara2014,
    AUTHOR = {Lara, H\'ector Chang and D\'avila, Gonzalo},
     TITLE = {Regularity for solutions of non local parabolic equations},
   JOURNAL = {Calc. Var. Partial Differential Equations},
  FJOURNAL = {Calculus of Variations and Partial Differential Equations},
    VOLUME = {49},
      YEAR = {2014},
    NUMBER = {1-2},
     PAGES = {139--172},
      ISSN = {0944-2669,1432-0835},
   MRCLASS = {35K55 (35B45 35B65 35D40 35R09)},
  MRNUMBER = {3148110},
MRREVIEWER = {Teemu\ Lukkari},
       DOI = {10.1007/s00526-012-0576-2},
       URL = {https://doi.org/10.1007/s00526-012-0576-2},
}

@article {Kassmann2009,
    AUTHOR = {Kassmann, Moritz},
     TITLE = {A priori estimates for integro-differential operators with
              measurable kernels},
   JOURNAL = {Calc. Var. Partial Differential Equations},
  FJOURNAL = {Calculus of Variations and Partial Differential Equations},
    VOLUME = {34},
      YEAR = {2009},
    NUMBER = {1},
     PAGES = {1--21},
      ISSN = {0944-2669,1432-0835},
   MRCLASS = {35R09 (35B45 35B65 60J75)},
  MRNUMBER = {2448308},
MRREVIEWER = {Cyril\ Imbert},
       DOI = {10.1007/s00526-008-0173-6},
       URL = {https://doi.org/10.1007/s00526-008-0173-6},
}

@article {DiCastro2016,
    AUTHOR = {Di Castro, Agnese and Kuusi, Tuomo and Palatucci, Giampiero},
     TITLE = {Local behavior of fractional {$p$}-minimizers},
   JOURNAL = {Ann. Inst. H. Poincar\'e{} C Anal. Non Lin\'eaire},
  FJOURNAL = {Annales de l'Institut Henri Poincar\'e{} C. Analyse Non
              Lin\'eaire},
    VOLUME = {33},
      YEAR = {2016},
    NUMBER = {5},
     PAGES = {1279--1299},
      ISSN = {0294-1449,1873-1430},
   MRCLASS = {35R11 (35B45 35B65 35D30 47G20)},
  MRNUMBER = {3542614},
MRREVIEWER = {Erwin\ Topp},
       DOI = {10.1016/j.anihpc.2015.04.003},
       URL = {https://doi.org/10.1016/j.anihpc.2015.04.003},
}

@article {Nowak2021,
    AUTHOR = {Nowak, Simon},
     TITLE = {Higher {H}\"older regularity for nonlocal equations with
              irregular kernel},
   JOURNAL = {Calc. Var. Partial Differential Equations},
  FJOURNAL = {Calculus of Variations and Partial Differential Equations},
    VOLUME = {60},
      YEAR = {2021},
    NUMBER = {1},
     PAGES = {Paper No. 24, 37},
      ISSN = {0944-2669,1432-0835},
   MRCLASS = {35R09 (35B65 35D30 47G20)},
  MRNUMBER = {4201647},
MRREVIEWER = {Jens\ Wirth},
       DOI = {10.1007/s00526-020-01915-1},
       URL = {https://doi.org/10.1007/s00526-020-01915-1},
}

@article {Nowak2023,
    AUTHOR = {Nowak, Simon},
     TITLE = {Improved {S}obolev regularity for linear nonlocal equations
              with {VMO} coefficients},
   JOURNAL = {Math. Ann.},
  FJOURNAL = {Mathematische Annalen},
    VOLUME = {385},
      YEAR = {2023},
    NUMBER = {3-4},
     PAGES = {1323--1378},
      ISSN = {0025-5831,1432-1807},
   MRCLASS = {35R09 (35B65 35D30 46E35 47G20)},
  MRNUMBER = {4566696},
       DOI = {10.1007/s00208-022-02369-w},
       URL = {https://doi.org/10.1007/s00208-022-02369-w},
}

@article {Kim2009,
    AUTHOR = {Kim, Panki and Song, Renming and Vondra\v cek, Zoran},
     TITLE = {Boundary {H}arnack principle for subordinate {B}rownian
              motions},
   JOURNAL = {Stochastic Process. Appl.},
  FJOURNAL = {Stochastic Processes and their Applications},
    VOLUME = {119},
      YEAR = {2009},
    NUMBER = {5},
     PAGES = {1601--1631},
      ISSN = {0304-4149,1879-209X},
   MRCLASS = {60J45 (35B45 35B65 60G51 60G52 60J25)},
  MRNUMBER = {2513121},
MRREVIEWER = {Sonia\ Fourati},
       DOI = {10.1016/j.spa.2008.08.003},
       URL = {https://doi.org/10.1016/j.spa.2008.08.003},
}

@article {Kim2014,
    AUTHOR = {Kim, Panki and Song, Renming and Vondra\v cek, Zoran},
     TITLE = {Global uniform boundary {H}arnack principle with explicit
              decay rate and its application},
   JOURNAL = {Stochastic Process. Appl.},
  FJOURNAL = {Stochastic Processes and their Applications},
    VOLUME = {124},
      YEAR = {2014},
    NUMBER = {1},
     PAGES = {235--267},
      ISSN = {0304-4149,1879-209X},
   MRCLASS = {60J45 (60J25 60J50)},
  MRNUMBER = {3131293},
MRREVIEWER = {Kilian\ Raschel},
       DOI = {10.1016/j.spa.2013.07.007},
       URL = {https://doi.org/10.1016/j.spa.2013.07.007},
}

@article {Bogdan2015,
    AUTHOR = {Bogdan, Krzysztof and Kumagai, Takashi and Kwa\'snicki,
              Mateusz},
     TITLE = {Boundary {H}arnack inequality for {M}arkov processes with
              jumps},
   JOURNAL = {Trans. Amer. Math. Soc.},
  FJOURNAL = {Transactions of the American Mathematical Society},
    VOLUME = {367},
      YEAR = {2015},
    NUMBER = {1},
     PAGES = {477--517},
      ISSN = {0002-9947,1088-6850},
   MRCLASS = {60J50 (31E05 60J75)},
  MRNUMBER = {3271268},
MRREVIEWER = {Zoran\ Vondra\v cek},
       DOI = {10.1090/S0002-9947-2014-06127-8},
       URL = {https://doi.org/10.1090/S0002-9947-2014-06127-8},
}

@misc{Cao2024,
      title={Uniform boundary Harnack principle for non-local operators on metric measure spaces}, 
      author={Shiping Cao and Zhen-Qing Chen},
      year={2024},
      eprint={2410.20719},
      archivePrefix={arXiv},
      primaryClass={math.PR},
      url={https://arxiv.org/abs/2410.20719}, 
}

@Article{Kim2018,
  author   = {Kim, Panki and Mimica, Ante},
  journal  = {Electron. J. Probab.},
  title    = {Estimates of {D}irichlet heat kernels for subordinate {B}rownian motions},
  year     = {2018},
  issn     = {1083-6489},
  pages    = {Paper No. 64, 45},
  volume   = {23},
  doi      = {10.1214/18-EJP190},
  fjournal = {Electronic Journal of Probability},
  mrclass  = {60J35 (60J50 60J75)},
  mrnumber = {3835470},
  url      = {https://doi.org/10.1214/18-EJP190},
}

@Article{Nowak2022,
  author  = {Kuusi, Tuomo and Nowak, Simon and Sire, Yannick},
  journal = {Amer. J. Math.},
  title   = {Gradient regularity and first-order potential estimates for a class of nonlocal equations},
  year    = {2022},
  note    = {To appear},
}

@article {Nguyen2024,
    AUTHOR = {Nguyen, Quoc-Hung and Nowak, Simon and Sire, Yannick and
              Weidner, Marvin},
     TITLE = {Potential theory for nonlocal drift-diffusion equations},
   JOURNAL = {Arch. Ration. Mech. Anal.},
  FJOURNAL = {Archive for Rational Mechanics and Analysis},
    VOLUME = {248},
      YEAR = {2024},
    NUMBER = {6},
     PAGES = {Paper No. 126, 42},
      ISSN = {0003-9527,1432-0673},
   MRCLASS = {35Q86 (31B35 35B65 35D30 35K08)},
  MRNUMBER = {4833232},
MRREVIEWER = {Yatao\ Li},
       DOI = {10.1007/s00205-024-02073-w},
       URL = {https://doi.org/10.1007/s00205-024-02073-w},
}

@article {Song2025,
    AUTHOR = {Song, Renming and Wu, Peixue and Wu, Shukun},
     TITLE = {Heat kernel estimates for regional fractional {L}aplacians
              with multi-singular critical potentials in {$C^{1,\beta}$}
              open sets},
   JOURNAL = {Stochastic Process. Appl.},
  FJOURNAL = {Stochastic Processes and their Applications},
    VOLUME = {189},
      YEAR = {2025},
     PAGES = {Paper No. 104727, 21},
      ISSN = {0304-4149,1879-209X},
   MRCLASS = {60G52 (35K08 47D07)},
  MRNUMBER = {4924357},
       DOI = {10.1016/j.spa.2025.104727},
       URL = {https://doi.org/10.1016/j.spa.2025.104727},
}

@misc{Song2025b,
      title={Abnormal boundary decay for the fractional Laplacian}, 
      author={Soobin Cho and Renming Song},
      year={2025},
      eprint={2510.03961},
      archivePrefix={arXiv},
      primaryClass={math.AP},
      url={https://arxiv.org/abs/2510.03961}, 
}

@Article{Kassmann2023,
  author    = {Kassmann, Moritz and Weidner, Marvin},
  journal   = {Calculus of Variations and Partial Differential Equations},
  title     = {Upper heat kernel estimates for nonlocal operators via Aronson’s method},
  year      = {2023},
  issn      = {1432-0835},
  month     = jan,
  number    = {2},
  volume    = {62},
  doi       = {10.1007/s00526-022-02398-y},
  publisher = {Springer Science and Business Media LLC},

}

@incollection {Biccari2018,
    AUTHOR = {Biccari, Umberto and Warma, Mahamadi and Zuazua, Enrique},
     TITLE = {Local regularity for fractional heat equations},
 BOOKTITLE = {Recent advances in {PDE}s: analysis, numerics and control},
    SERIES = {SEMA SIMAI Springer Ser.},
    VOLUME = {17},
     PAGES = {233--249},
 PUBLISHER = {Springer, Cham},
      YEAR = {2018},
   MRCLASS = {35R11 (35B65 35S05)},
  MRNUMBER = {3888964},
}

@Misc{Kim2023a,
  author    = {Kim, Minhyun and Lee, Se-Chan},
  title     = {Supersolutions and superharmonic functions for nonlocal operators with Orlicz growth},
  year      = {2023},
  copyright = {Creative Commons Attribution 4.0 International},
  doi       = {10.48550/ARXIV.2311.01246},
  publisher = {arXiv},
}

@article {Bass2009,
    AUTHOR = {Bass, Richard F.},
     TITLE = {Regularity results for stable-like operators},
   JOURNAL = {J. Funct. Anal.},
  FJOURNAL = {Journal of Functional Analysis},
    VOLUME = {257},
      YEAR = {2009},
    NUMBER = {8},
     PAGES = {2693--2722},
      ISSN = {0022-1236,1096-0783},
   MRCLASS = {47G30 (31B35 35B65 47F05)},
  MRNUMBER = {2555009},
MRREVIEWER = {Erkan\ Nane},
       DOI = {10.1016/j.jfa.2009.05.012},
       URL = {https://doi.org/10.1016/j.jfa.2009.05.012},
}

@article {Bogdan2010,
    AUTHOR = {Bogdan, Krzysztof and Grzywny, Tomasz and Ryznar, Micha\l},
     TITLE = {Heat kernel estimates for the fractional {L}aplacian with
              {D}irichlet conditions},
   JOURNAL = {Ann. Probab.},
  FJOURNAL = {The Annals of Probability},
    VOLUME = {38},
      YEAR = {2010},
    NUMBER = {5},
     PAGES = {1901--1923},
      ISSN = {0091-1798,2168-894X},
   MRCLASS = {60J35 (31B25 60J50 60J75)},
  MRNUMBER = {2722789},
MRREVIEWER = {Victoria\ Knopova},
       DOI = {10.1214/10-AOP532},
       URL = {https://doi.org/10.1214/10-AOP532},
}

@Article{Diening2024,
  author    = {Diening, Lars and Kim, Kyeongbae and Lee, Ho-Sik and Nowak, Simon},
  journal   = {Journal of the European Mathematical Society},
  title     = {Nonlinear nonlocal potential theory at the gradient level},
  year      = {2025},
  issn      = {1435-9863},
  month     = sep,
  doi       = {10.4171/jems/1706},
  publisher = {European Mathematical Society - EMS - Publishing House GmbH},
}

@Article{Kim2022,
  author    = {Kim, Panki and Song, Renming and Vondraček, Zoran},
  journal   = {Journal of Functional Analysis},
  title     = {Potential theory of Dirichlet forms with jump kernels blowing up at the boundary},
  year      = {2025},
  issn      = {0022-1236},
  month     = aug,
  number    = {4},
  pages     = {110934},
  volume    = {289},
  doi       = {10.1016/j.jfa.2025.110934},
  publisher = {Elsevier BV},
}

@Article{Liao2024,
  author    = {Liao, Naian and Weidner, Marvin},
  journal   = {Proceedings of the London Mathematical Society},
  title     = {Time‐insensitive nonlocal parabolic Harnack estimates},
  year      = {2025},
  issn      = {1460-244X},
  month     = may,
  number    = {5},
  volume    = {130},
  doi       = {10.1112/plms.70051},
  publisher = {Wiley},
}

@Article{Cozzi2017,
  author    = {Cozzi, Matteo},
  journal   = {Annali di Matematica Pura ed Applicata (1923 -)},
  title     = {Interior regularity of solutions of non-local equations in {S}obolev and {N}ikol'skii spaces},
  year      = {2016},
  issn      = {1618-1891},
  month     = jul,
  number    = {2},
  pages     = {555--578},
  volume    = {196},
  doi       = {10.1007/s10231-016-0586-3},
  publisher = {Springer Science and Business Media LLC},
  url       = {https://doi.org/10.1007/s10231-016-0586-3},
}

@Article{Bogdan2014,
  author     = {Bogdan, Krzysztof and Grzywny, Tomasz and Ryznar, Michał},
  journal    = {Stochastic Process. Appl.},
  title      = {Dirichlet heat kernel for unimodal {L}\'evy processes},
  year       = {2014},
  issn       = {0304-4149,1879-209X},
  month      = nov,
  number     = {11},
  pages      = {3612--3650},
  volume     = {124},
  doi        = {10.1016/j.spa.2014.06.001},
  fjournal   = {Stochastic Processes and their Applications},
  mrclass    = {60J35 (31C25 60J50 60J75)},
  mrnumber   = {3249349},
  mrreviewer = {Tomasz\ \.Zak},
  publisher  = {Elsevier BV},
  url        = {https://doi.org/10.1016/j.spa.2014.06.001},
}

@Article{Brasco2017,
  author     = {Brasco, Lorenzo and Lindgren, Erik},
  journal    = {Adv. Math.},
  title      = {Higher {S}obolev regularity for the fractional {$p$}-{L}aplace equation in the superquadratic case},
  year       = {2017},
  issn       = {0001-8708,1090-2082},
  month      = jan,
  pages      = {300--354},
  volume     = {304},
  doi        = {10.1016/j.aim.2016.03.039},
  fjournal   = {Advances in Mathematics},
  mrclass    = {35R11 (35B65 35J70 35R09)},
  mrnumber   = {3558212},
  mrreviewer = {Adam\ Kubica},
  publisher  = {Elsevier BV},
  url        = {https://doi.org/10.1016/j.aim.2016.03.039},
}

@Article{Brasco2018,
  author     = {Brasco, Lorenzo and Lindgren, Erik and Schikorra, Armin},
  journal    = {Adv. Math.},
  title      = {Higher {H}\"older regularity for the fractional {$p$}-{L}aplacian in the superquadratic case},
  year       = {2018},
  issn       = {0001-8708,1090-2082},
  month      = nov,
  pages      = {782--846},
  volume     = {338},
  doi        = {10.1016/j.aim.2018.09.009},
  fjournal   = {Advances in Mathematics},
  mrclass    = {35R11 (35B65 35J70 35J92)},
  mrnumber   = {3861716},
  mrreviewer = {Vincenzo\ Ambrosio},
  publisher  = {Elsevier BV},
  url        = {https://doi.org/10.1016/j.aim.2018.09.009},
}

@Article{Chen2025,
  author    = {Chen, Zhen-Qing and Hu, Eryan and Zhao, Guohuan},
  journal   = {Journal of Functional Analysis},
  title     = {Dirichlet heat kernel estimates for rectilinear stable processes},
  year      = {2025},
  issn      = {0022-1236},
  month     = mar,
  number    = {6},
  pages     = {110812},
  volume    = {288},
  doi       = {10.1016/j.jfa.2024.110812},
  publisher = {Elsevier BV},
}

@Article{Chen2010,
  author    = {Chen, Zhen-Qing and Kim, Panki and Song, Renming},
  journal   = {Journal of the European Mathematical Society},
  title     = {Heat kernel estimates for the Dirichlet fractional Laplacian},
  year      = {2010},
  issn      = {1435-9863},
  month     = aug,
  number    = {5},
  pages     = {1307--1329},
  volume    = {12},
  doi       = {10.4171/jems/231},
  publisher = {European Mathematical Society - EMS - Publishing House GmbH},
}

@Article{Chen2012,
  author    = {Chen, Zhen-Qing and Kim, Panki and Song, Renming},
  journal   = {The Annals of Probability},
  title     = {Dirichlet heat kernel estimates for fractional Laplacian with gradient perturbation},
  year      = {2012},
  issn      = {0091-1798},
  month     = nov,
  number    = {6},
  volume    = {40},
  doi       = {10.1214/11-aop682},
  publisher = {Institute of Mathematical Statistics},
}

@Article{Chen2014,
  author     = {Chen, Zhen-Qing and Kim, Panki and Song, Renming},
  journal    = {Proc. Lond. Math. Soc. (3)},
  title      = {Dirichlet heat kernel estimates for rotationally symmetric {L}\'evy processes},
  year       = {2014},
  issn       = {0024-6115,1460-244X},
  month      = feb,
  number     = {1},
  pages      = {90--120},
  volume     = {109},
  doi        = {10.1112/plms/pdt068},
  fjournal   = {Proceedings of the London Mathematical Society. Third Series},
  mrclass    = {60J35 (47G20 60J75)},
  mrnumber   = {3237737},
  mrreviewer = {Anita\ Diana\ Behme},
  publisher  = {Wiley},
  url        = {https://doi.org/10.1112/plms/pdt068},
}

@Article{Chen2022,
  author     = {Chen, Xin and Kim, Panki and Wang, Jian},
  journal    = {Math. Ann.},
  title      = {Two-sided {D}irichlet heat kernel estimates of symmetric stable processes on horn-shaped regions},
  year       = {2022},
  issn       = {0025-5831,1432-1807},
  month      = oct,
  number     = {1-2},
  pages      = {373--418},
  volume     = {384},
  doi        = {10.1007/s00208-021-02272-w},
  fjournal   = {Mathematische Annalen},
  mrclass    = {60G51 (35K05 60G52 60J25 60J76)},
  mrnumber   = {4476227},
  mrreviewer = {Liping\ Li},
  publisher  = {Springer Science and Business Media LLC},
  url        = {https://doi.org/10.1007/s00208-021-02272-w},
}

@Article{Fall2022,
  author    = {Fall, Mouhamed Moustapha and Mengesha, Tadele and Schikorra, Armin and Yeepo, Sasikarn},
  journal   = {Partial Differ. Equ. Appl.},
  title     = {Calder\'on-{Z}ygmund theory for non-convolution type nonlocal equations with continuous coefficient},
  year      = {2022},
  issn      = {2662-2963,2662-2971},
  month     = mar,
  number    = {2},
  pages     = {Paper No. 24, 27},
  volume    = {3},
  doi       = {10.1007/s42985-022-00161-8},
  fjournal  = {Partial Differential Equations and Applications},
  mrclass   = {35R11 (35J99 46E35 47G30)},
  mrnumber  = {4398497},
  publisher = {Springer Science and Business Media LLC},
  url       = {https://doi.org/10.1007/s42985-022-00161-8},
}

@Article{KassmannWeidner2022.1,
  author    = {Kassmann, Moritz and Weidner, Marvin},
  journal   = {Mathematische Annalen},
  title     = {Nonlocal operators related to nonsymmetric forms I: Hölder estimates},
  year      = {2025},
  issn      = {1432-1807},
  month     = sep,
  number    = {2},
  pages     = {2307--2389},
  volume    = {393},
  doi       = {10.1007/s00208-025-03237-z},
  publisher = {Springer Science and Business Media LLC},
}

@Article{KassmannWeidner2022.2,
  author    = {Kassmann, Moritz and Weidner, Marvin},
  journal   = {Analysis \& PDE},
  title     = {Nonlocal operators related to nonsymmetric forms, II: Harnack inequalities},
  year      = {2024},
  issn      = {2157-5045},
  month     = nov,
  number    = {9},
  pages     = {3189--3249},
  volume    = {17},
  doi       = {10.2140/apde.2024.17.3189},
  publisher = {Mathematical Sciences Publishers},
}

@Article{Grzywny2020,
  author     = {Grzywny, Tomasz and Kim, Kyung-Youn and Kim, Panki},
  journal    = {Stochastic Process. Appl.},
  title      = {Estimates of {D}irichlet heat kernel for symmetric {M}arkov processes},
  year       = {2020},
  issn       = {0304-4149,1879-209X},
  month      = jan,
  number     = {1},
  pages      = {431--470},
  volume     = {130},
  doi        = {10.1016/j.spa.2019.03.017},
  fjournal   = {Stochastic Processes and their Applications},
  mrclass    = {31B25 (60J35 60J50 60J76)},
  mrnumber   = {4035035},
  mrreviewer = {Jie-Ming\ Wang},
  publisher  = {Elsevier BV},
  url        = {https://doi.org/10.1016/j.spa.2019.03.017},
}

@Article{Kim2014b,
  author     = {Kim, Kyung-Youn and Kim, Panki},
  journal    = {Stochastic Process. Appl.},
  title      = {Two-sided estimates for the transition densities of symmetric {M}arkov processes dominated by stable-like processes in {$C^{1,\eta}$} open sets},
  year       = {2014},
  issn       = {0304-4149,1879-209X},
  month      = sep,
  number     = {9},
  pages      = {3055--3083},
  volume     = {124},
  doi        = {10.1016/j.spa.2014.04.004},
  fjournal   = {Stochastic Processes and their Applications},
  mrclass    = {60J35 (47D07 60J75)},
  mrnumber   = {3217433},
  mrreviewer = {Fangjun\ Xu},
  publisher  = {Elsevier BV},
  url        = {https://doi.org/10.1016/j.spa.2014.04.004},
}

@Article{Mengesha2021,
  author    = {Mengesha, Tadele and Schikorra, Armin and Yeepo, Sasikarn},
  journal   = {Adv. Math.},
  title     = {Calderon-{Z}ygmund type estimates for nonlocal {PDE} with {H}\"older continuous kernel},
  year      = {2021},
  issn      = {0001-8708,1090-2082},
  month     = jun,
  pages     = {Paper No. 107692, 64},
  volume    = {383},
  doi       = {10.1016/j.aim.2021.107692},
  fjournal  = {Advances in Mathematics},
  mrclass   = {35R11 (35J99 46E35 47G30)},
  mrnumber  = {4233278},
  publisher = {Elsevier BV},
  url       = {https://doi.org/10.1016/j.aim.2021.107692},
}

\enlargethispage{1.3cm}

\end{document}